\documentclass{amsart}

\usepackage{amsmath,amsfonts,amsthm,amssymb,array,graphicx,enumerate,tikz-cd,xspace,semtkzX,comment}
\usepackage{caption}
\def\graphdslabelscale{1.2}
\def\rootscale{0.75}
\def\rootsep{0.045}
\def\clustersep{0.09}
\def\cnamescale{0.6}
\def\cdepthscale{0.6}

\theoremstyle{plain}
\newtheorem{theorem}[subsubsection]{Theorem}
\newtheorem{corollary}[subsubsection]{Corollary}
\newtheorem{lemma}[subsubsection]{Lemma}
\newtheorem{proposition}[subsubsection]{Proposition}
\theoremstyle{definition}
\newtheorem{construction}[subsubsection]{Construction}
\newtheorem{remark}[subsubsection]{Remark}
\newtheorem{example}[subsubsection]{Example}
\newtheorem{definition}[subsubsection]{Definition}
\newtheorem{notation}[subsubsection]{Notation}
\newtheorem*{metricvariant*}{Metric Version}

\numberwithin{equation}{subsection}
\makeatletter
\@addtoreset{equation}{section}
\makeatother

\renewcommand{\bar}{\overline}
\newcommand{\num}[1]{n(#1)}

\def\R{\mathbb R}
\def\Q{\mathbb Q}
\def\Z{\mathbb Z}
\def\N{\mathbb N}
\def\C{\mathbb C}
\def\F{\mathbb F}
\def\E{\mathbb E}
\def\A{\mathbb A}
\def\P{\mathbb P}
\def\H{\mathrm H}
\def\G{\mathcal G}
\def\I{\mathrm I}
\renewcommand{\O}{\mathcal O}
\newcommand\mcal{\mathcal}

\def\isoarrow{\overset{\sim}{\longrightarrow}}
\def\Hom{\mathrm{Hom}}
\def\Aut{\mathrm{Aut}}

\def\Ind{\mathrm{Ind}}

\def\cR{\mathcal{R}}
\def\s{\mathfrak{s}}

\def\Frob{\mathrm{Frob}}

\def\coker{\mathrm{coker}}
\def\gNeron{\Phi}
\def\usual{\mathrm{usual}}
\def\Fil{\mathrm{Fil}}
\def\Gr{\mathrm{Gr}}
\def\Berk{\mathrm{Berk}}
\def\Jac{\mathrm{Jac}}
\newcommand\Div{\mathrm{Div}}

\def\nr{\mathrm{nr}}

\newcommand\id{\mathrm{id}}
\newcommand\even{\mathrm{even}}

\def\disc{\mathrm{disc}}
\def\im{\mathrm{im}}

\def\sm{\mathrm{sm}}
\def\GL{\mathrm{GL}}
\def\SL{\mathrm{SL}}

\def\s{\mathfrak{s}}

\usepackage{color}

\title[Tamagawa numbers of semistable hyperelliptic curves]{Variation of Tamagawa numbers of Jacobians of hyperelliptic curves with semistable reduction}
\author{L.\ Alexander Betts}
\thanks{This paper was written while the author was supported by EPSRC grant EP/M016838/2: \emph{Arithmetic of hyperelliptic curves}.}
\date\today
\address{Max Planck Institut f\"ur Mathematik, Vivatsgasse 7, Bonn 53111}
\email{betts@mpim-bonn.mpg.de}

\begin{document}

\maketitle

\begin{abstract}
We study how Tamagawa numbers of Jacobians of hyperelliptic curves vary as one varies the base field or the curve, in the case of semistable reduction. We find that there are strong constraints on the behaviour that appears, some of which are unexpected and specific to hyperelliptic curves. Our methods are explicit and allow one to write down formulae for Tamagawa numbers of infinite families of hyperelliptic curves, of the kind used in proofs of the parity conjecture for Jacobians of curves of small genus.
\end{abstract}

\tableofcontents

\newpage

\section{Introduction}

Fix a finite extension $K$ of $\Q_p$ with ring of integers $\O_K$ and residue field $k$, and let $X/K$ be a (smooth, projective, geometrically integral) semistable\footnote{Throughout this paper, we will use ``semistable'' as a shorthand for ``with semistable reduction'', i.e.\ possessing a regular model whose special fibre is a reduced normal crossings divisor.} curve. The Jacobian $\Jac(X)$ is then also semistable, meaning that the special fibre of its N\'eron model is an extension of a finite $k$-group-scheme $\Phi_{\Jac(X)/K}$ by an abelian variety by a torus. The number $c_{\Jac(X)/K}:=\#\Phi_{\Jac(X)/K}(k)$ of $k$-rational points of this group is known as the \emph{Tamagawa number} of $\Jac(X)$, and provides a crude numerical invariant of the reduction type of $\Jac(X)$. Although this captures very little information about the reduction type, Tamagawa numbers turn out to have lots of arithmetic content, and appear famously in the leading terms for the $L$-series of abelian varieties predicted by the Birch--Swinnerton-Dyer Conjecture.

The object of this paper is to study how the Tamagawa number $c_{\Jac(X)/K}$ of the Jacobian of $X$ -- henceforth simply called the Tamagawa number $c_{X/K}$ of $X$ -- changes as we vary both the curve $X$ and the field $K$, in the special case that $X/K$ is hyperelliptic\footnote{For our purposes, a hyperelliptic curve is a smooth, projective, geometrically connected curve $X/K$ endowed with a degree $2$ map to a genus $0$ curve over $K$. We don't assume that the genus of $X$ is $\geq2$, or that $X/\iota\simeq\P^1_K$.}.

\subsection{Variation in field extensions}

Our first and most striking result governs the behaviour of the Tamagawa number $c_{X/L}:=c_{X_L/L}$ as $L$ varies over finite extensions of $K$, for a fixed semistable hyperelliptic curve $X/K$. Our main theorem says that $c_{X/L}$ is a function of $L$ of a very specific form.

\begin{theorem}[Theorem~\ref{thm:combinatorial_growth_in_towers}, cf.\ Lemma~\ref{lem:change_of_BY_trees_in_field_extensions}]\label{thm:growth_in_towers}
Let $p$ be a prime, $K$ a finite extension of $\Q_p$, and $X/K$ a semistable hyperelliptic curve. Then there are elements $(a_d,r_d,s_d)\in\N\times\N_0\times\Z$ for each $d\in\N$ (equal to $(1,0,0)$ for almost all $d$) such that
\[
c_{X/L} = \prod_{d\mid f}\left(a_d\cdot e^{r_d}\cdot\gcd(e,2)^{s_d}\right)^{\varphi(d)}
\]
for all finite extensions $L/K$, where $e=e(L/K)$ and $f=f(L/K)$ are the ramification and residue class degrees, respectively, and $\varphi$ is Euler's totient function.
\end{theorem}

Here is a sample consequence.

\begin{corollary}\label{cor:tamagawa_numbers_higher_powers}
Suppose that we are in the setup of theorem \ref{thm:growth_in_towers}. Let $q$ be a prime, and let $K_q/K$ be the unramified extension of degree $q$. Then $c_{X/K_q}/c_{X/K}$ is a $(q-1)$th power.
\begin{proof}
Apply theorem~\ref{thm:growth_in_towers} to $X/K$ and $X_{K_q}/K_q$.
\end{proof}
\end{corollary}

What is particularly surprising about this result is that it is no longer true if we remove the assumption that $X$ is hyperelliptic: we will give an example in example~\ref{ex:non-hyperelliptic} of a semistable non-hyperelliptic curve $X/K$ such that $c_{X/K_5}=121$ but $c_{X/K}=1$, whose ratio is not a $4$th power. Thus, theorem~\ref{thm:growth_in_towers} shows that the functions $L\mapsto c_{X/L}$ for semistable hyperelliptic curves $X/K$ are not generic among the corresponding functions for all semistable curves.

\subsection{Variation in degenerating families}\label{ss:degeneration}

Our second result governs the behaviour of the Tamagawa number $c_{X/K}$ as $X$ varies in a degenerating family, in the case that $K$ is a finite extension of $\Q_p$ for $p\neq2$. Suppose that $f_0\in K[x]$ is a polynomial of degree $d\geq5$ with at worst double roots, whose splitting field is unramified over $K$. If $f\in K[x]$ is a squarefree polynomial of the same degree $d$, then we will describe the behaviour of the Tamagawa numbers of the curves $X_f$ with affine equation $y^2=f(x)$ as $f$ approaches $f_0$ in the $p$-adic topology, at least when $X_f$ is semistable.

These Tamagawa numbers are controlled very precisely by the relative position of the roots of $f$. It we enumerate the double roots of $f_0$ as $\alpha_1,\dots,\alpha_m$, then any $f$ sufficiently close to $f_0$ has exactly two roots $\beta_i,\gamma_i$ in a small disc about each $\alpha_i$ (see proposition~\ref{prop:close_polys_have_close_roots} for a precise statement). We write $d_i(f):=v_K(\beta_i-\gamma_i)\in\frac12\Z$, where $v_K$ is the valuation on the algebraic closure of $K$, normalised so that the valuation of a uniformiser of $K$ is $1$. These quantities entirely control the Tamagawa numbers of the curves $X_f$, when semistable, in a precise manner.

\begin{theorem}[see \S\ref{ss:degeneration_proof}]\label{thm:degeneration}
Keep notation ($p$, $K$, $f_0$, $d$, $X_f$, $d_i(f)$) and hypotheses as above, and let $|\cdot|\colon K[x]\rightarrow\R^{\geq0}$ denote the Gauss norm on polynomials (the largest norm\footnote{The normalisation of the norm on $K$ doesn't matter in this theorem; later to fix notation we will use the norm on $K$ for which a uniformiser has norm $p^{-1}$.} of a coefficient). Then there is a positive $\delta$, depending on $f_0$, such that:
\begin{enumerate}
	\item\label{thmpart:degeneration_semistability} Either:
	\begin{itemize}
		\item for all squarefree polynomials $f\in K[x]$ of degree $d$ with $|f-f_0|<\delta$, the curve $X_f$ is semistable; or
		\item for all squarefree polynomials $f\in K[x]$ of degree $d$ with $|f-f_0|<\delta$, the curve $X_f$ is not semistable.
	\end{itemize}
	\item\label{thmpart:degeneration_formula} In the first case (all $X_f$ semistable), there exists a rational polynomial $P$ in $m$ variables and a function $s\colon\F_2^m\rightarrow\Z$ such that
	\[
	c_{X_f/K} = 2^{s(\rho(2d_1(f),\dots,2d_m(f)))}\cdot P(d_1(f),\dots,d_m(f))
	\]
	for all squarefree polynomials $f\in K[x]$ of degree $d$ with $|f-f_0|<\delta$, where $\rho\colon\Z^m\to\F_2^m$ is the reduction map modulo $2$.
\end{enumerate}
Moreover, there is an explicit procedure to determine which of the two possibilities occurs in part~\eqref{thmpart:degeneration_semistability} and the polynomial $P$ and function $s$ in part~\eqref{thmpart:degeneration_formula}.
\end{theorem}

%\begin{remark}
%Informally, theorem~\ref{thm:degeneration} describes the Tamagawa numbers of semistable hyperelliptic curves near a certain boundary point on the moduli stack as a function in the valuations of local parameters at that boundary point. Up to powers of $2$, the Tamagawa number is a polynomial in the valuations of the local parameters. The author doesn't know whether the appearance of the number $2$ here is specific to the setting of hyperelliptic curves, or whether something similar holds in the non-hyperelliptic case.
%\end{remark}

\begin{remark}
In the case that $f_0$ is squarefree, theorem~\ref{thm:degeneration} shows that if $X_{f_0}$ is semistable, then there is some $\delta>0$ such that whenever $|f-f_0|<\delta$, we have that $f$ is also squarefree, and that $X_f$ is semistable with the same Tamagawa number as $X_{f_0}$.
\end{remark}

\begin{remark}
We can think of theorem~\ref{thm:degeneration}\eqref{thmpart:degeneration_formula} as describing Tamagawa numbers of semistable hyperelliptic curves in a neighbourhood of certain points on the boundary of the moduli space of hyperelliptic curves. However, some care is needed in doing so, since there are several inequivalent ways to compactify the moduli space of hyperelliptic curves \cite[Remark~4.3]{arsie-vistoli:stacks_of_covers}.

Let $\A_g\cong\A^{2g+3}$ denote the affine space (over $\Q$) parametrising homogenous binary forms in two variables $x,z$, and let $\A_{g,\sm}$ denote the open subscheme consisting of squarefree forms. There is an action of $\GL_2$ on $\A_g$ by change of coordinates, under which the subscheme $\mu_{g+1}$ of $(g+1)$th roots of unity (embedded diagonally in $\GL_2$) acts trivially. Arsie and Vistoli prove that the moduli stack $\mathcal H_{g,\sm}$ of (smooth) hyperelliptic curves is canonically equivalent to the quotient stack $[\A_{g,\sm}/(\GL_2/\mu_{g+1})]$ \cite[Corollary~4.2]{arsie-vistoli:stacks_of_covers}. The universal hyperelliptic curve on $\mathcal H_{g,\sm}$ pulls back to the hyperelliptic curve on $\A_{g,\sm}$ with affine equation $y^2=F(x,1)$, where $F(x,z)$ is the universal homogenous form of degree $2g+2$ on $\A_{g,\sm}$ (if $a_0,\dots,a_{2g+2}$ are the coordinates on $\A_g\cong\A^{2g+3}$, then~$F$ is the form $a_{2g+2}x^{2g+2}+a_{2g+1}x^{2g+1}z+\dots+a_0z^{2g+2}$).

One can produce various partial compactifications of the moduli space $\mathcal H_{g,\sm}$ by finding $\GL_2/\mu_{g+1}$-equivariant dense open embeddings of $\A_{g,\sm}$ into larger schemes $\overline\A_g$ -- the quotient stack $[\overline\A_g/(\GL_2/\mu_{g+1})]$ then contains $\mathcal H_{g,\sm}$ as a dense open substack. For instance, taking  the scheme $\A_{g,0}$ of all non-zero binary forms produces a partial compactification $\mathcal H_g$ of $\mathcal H_{g,\sm}$ which is an Artin stack with finite diagonal and quasiprojective moduli space \cite[p654]{arsie-vistoli:stacks_of_covers}. Alternatively, one can use a procedure due to Kirwan to construct a $\GL_2/\mu_{g+1}$-equivariant morphism $K_g\to\A_{g,0}$, and the corresponding quotient stack $\overline{\mathcal H}_g=[K_g/(\GL_2/\mu_{g+1})]$ is an Artin stack with finite diagonal and projective moduli space (hence, a compactification of $\mathcal H_{g,\sm}$) \cite[Remark~4.3]{arsie-vistoli:stacks_of_covers}.

Instead of these, theorem~\ref{thm:degeneration} takes place on the partial compactification $\mathcal H_{g,\leq2}:=[\A_{g,\leq2}/(\GL_2/\mu_{g+1})]$ of $\mathcal H_{g,\sm}$, where $\A_{g,\leq2}$ is the space of cubefree homogenous binary forms of degree $2g+2$. This is a dense open substack of both $\mathcal H_g$ and\footnote{This latter requires some justification. Kirwan's procedure, applied to a projective variety $X$ with an action of a reductive group $G$ on its affine cone, gives an iterated blowup of $X$ whose exceptional locus is contained in the locus of non-semistable points, together with the closure of the locus of semistable points fixed by a reductive subgroup of positive dimension \cite[\S\S5--6~\&~Lemma~4.3]{kirwan:partial_desingularisations}. In the particular case $X=\P(\A_{g,0})$, $G=\SL_2$, as in \cite[\S9]{kirwan:partial_desingularisations}, the non-semistable locus consists of the homogenous forms with a root of multiplicity $\geq g+2$ in $\P^1$, while the semistable points fixed by a positive-dimensional reductive subgroup are the homogenous forms with two distinct roots of multiplicity $g+1$ each. The closure of these loci is disjoint from $\A_{g,\leq2}$, which is thus an open subscheme of $K_g$. This makes $\mathcal H_{g,\leq2}$ an open substack of $\overline{\mathcal H}_g$.} $\overline{\mathcal H}_g$. The cubefree polynomial $f_0$ of degree $d=2g+1$ or $2g+2$ determines a $K_v$-point on $\A_{g,\leq2}$, and theorem~\ref{thm:formula} describes the Tamagawa number of a semistable hyperelliptic curve corresponding to a point in the intersection of $\A_{g,\sm}$ with a small $p$-adic neighbourhood of $f_0$.
\end{remark}

\subsection{Reduction types and Tamagawa numbers}

What unites the above two settings -- enlargement of the base field and degeneration of the curve -- is that in both cases we are interested in studying infinite families of curves whose reduction types belong to an infinite parametrised family. As an illustration of what we mean by this, consider the elliptic curve $E/\Q_p$ with affine equation $y^2=(x-x^2)(x-p)$. Over $\Q_p$, this curve has split multiplicative reduction of Kodaira type $\I_1$, but over a finite extension $L$ of $\Q_p$ the reduction type instead becomes $\I_{e(L/\Q_p)}$. Equally, in the family of elliptic curves $E_\alpha/\Q_p$ with equations $y^2=(x-x^2)(x-\alpha)$ for $\alpha\in p\Z_p\setminus\{0\}$, all the curves have split multiplicative reduction of Kodaira type $\I_n$ for some $n$, but the value of $n$ depends on the value of $\alpha$ (in fact, $n=v_p(\alpha)$). Thus, the Tamagawa numbers of $E_L/L$ for varying $L$ as well as the Tamagawa numbers of the curves $E_\alpha/\Q_p$ for varying $\alpha$ are both determined by the following fact: the Tamagawa number of an elliptic curve of split type $\I_n$ reduction is $n$.

To generalise this observation to hyperelliptic curves, we will replace the Kodaira type of a semistable elliptic curve $E/K$ with the dual graph $\G$ of the geometric special fibre of the minimal regular model $\mathfrak X$ of $X/K$ (definition~\ref{def:dual_graph}) together with its induced Frobenius action. In the setup of theorem~\ref{thm:growth_in_towers}, it turns out that the dual graph associated to $X_L/L$ is a subdivision of the dual graph associated to $X/K$ (lemma~\ref{lem:change_of_BY_trees_in_field_extensions}). Equally, in the setup of theorem~\ref{thm:degeneration}\eqref{thmpart:degeneration_formula}, the dual graph associated to $X_f$ is a subdivision of a fixed graph determined by $f_0$ (corollary~\ref{cor:close_polys_have_similar_BY_trees}). Thus, the bulk of the work in proving our two main theorems will be contained in a result of the following kind: for a fixed graph~$\G_0$ with automorphism, the Tamagawa number of any semistable hyperelliptic curve~$X$ whose dual graph is a subdivision of~$\G_0$ is given by an explicit formula, depending only on the number of times each edge of~$\G_0$ is subdivided in the dual graph of~$X$.

% Thus, what unites our two main theorems is that, in both cases, we wish to compute Tamagawa numbers of infinite families of semistable curves whose dual graphs are all subdivisions of a fixed graph.

\smallskip

To prove a result of this kind, we use the well-known fact that the Tamagawa number of a semistable curve~$X/K$ can be read off from its dual graph~$\G$ (with the induced Frobenius action). Over the maximal unramified extension $K^\nr$ of $K$, the description is simple: the Tamagawa number of $X_{K^\nr}/K^\nr$ is equal to the order of the \emph{Jacobian} or \emph{sandpile group} $\gNeron_\G$ of $\G$, a certain finite group defined purely combinatorially in terms of $\G$ (see definition~\ref{def:intersection}). Moreover, the order of $\gNeron_\G$ also admits a purely graph-theoretic characterisation as the number of spanning trees in $\G$ \cite{levine-propp:sandpile?}.

Over the ground field $K$ however, the Tamagawa number is instead equal to the order of the Frobenius-invariant subgroup of $\gNeron_\G$. This quantity is in general much more subtle than the order of the entire group $\gNeron_\G$ \cite{bosch-liu:rational_neron}, and in particular doesn't seem to have a natural graph-theoretic interpretation in terms of spanning trees. Moreover, it is far from clear how this quantity behaves under subdivision of edges of $\G$: \textit{a priori}, the Jacobian $\gNeron_\G$ is the cokernel of a matrix whose dimension depends on the number of vertices of $\G$, and then computing the number of invariant elements of this cokernel involves complicated calculations with the cohomology of lattice representations of cyclic groups.

We overcome these difficulties by making some simplifications particular to the setting of hyperelliptic curves. The dual graph $\G$ of the geometric special fibre of a semistable hyperelliptic curve has the property that the quotient of $\G$ by the hyperelliptic involution is a tree $T$ (lemma~\ref{lem:hyperelliptic_reduction}) -- such a graph is called a \emph{hyperelliptic graph} in \cite{baker-norine:hyperelliptic,m2d2} -- and $\G$ can be reconstructed from this quotient tree together with the ramification locus $S$ of the quotient map $\G\twoheadrightarrow T$ (proposition~\ref{prop:hyperelliptic-BY_equivalence}). Moreover, the Frobenius on~$\G$ induces a certain ``signed'' automorphism $\epsilon F$ of the pair $(T,S)$ (definition~\ref{def:BY_tree}).

This data of a tree $T$ and a subgraph $S\subseteq T$ is called a \emph{BY tree}\footnote{BY since the subgraph $S$ is usually drawn in blue and the tree $T$ in yellow.} (definition~\ref{def:BY_tree}, following \cite{semistable_types}), and provides the combinatorial framework we will use for studying Tamagawa numbers of semistable hyperelliptic curves. We will define an invariant $c_{T,\epsilon F}$ of a BY tree $T=(T,S)$ and signed automorphism $\epsilon F$, called its \emph{Tamagawa number} (definition~\ref{def:BY_intersection}), which, when $(T,S,\epsilon F)$ arises from a semistable hyperelliptic curve~$X/K$, recovers the Tamagawa number~$c_{X/K}$ of~$X/K$ (\S\ref{sss:BY_tree_of_curve}). The main technical input in this paper, then, is a purely combinatorial result describing how the Tamagawa number of a BY tree changes as one subdivides its edges.

\begin{theorem}[=corollary~\ref{cor:qualitative}]\label{thm:subdivision_formula}
Let $T_0=(T_0,S_0)$ be a BY tree and $\epsilon F=(F,\epsilon)$ a signed isomorphism of $T_0$ (see definition~\ref{def:BY_tree}). Enumerate the $F$-orbits of edges in $T_0$ as $\omega_1,\dots,\omega_m$, and for a tuple $l=(l_1,\dots,l_m)\in\N^m$ let $T_0^{(l)}$ denote the subdivision of $T_0$ formed by replacing each edge in the $F$-orbit $\omega_i$ by a chain of $l_i$ edges. The signed automorphism $\epsilon F$ induces a signed automorphism of $T_0^{(l)}$, which we also denote by $\epsilon F$.

Then there is a homogenous polynomial $P\in\Q[t_1,\dots,t_m]$ and a function $s\colon\F_2^m\rightarrow\Z$ such that the Tamagawa number $c_{T_0^{(l)},\epsilon F}$ of $(T_0^{(l)},\epsilon F)$ is given by
\[
c_{T_0^{(l)},\epsilon F} = 2^{s(\rho(l))}\cdot P(l)\,,
\]
where $\rho\colon\Z^m\to\F_2^m$ is the reduction map modulo $2$. Moreover, for $T_0^{(l)}$ and $\epsilon F$ both even (definition~\ref{def:parity} -- this condition is automatic for BY trees arising for hyperelliptic curves) there is an explicit description of both $P$ and $s$ in terms of $T_0$.
\end{theorem}

\begin{remark}
The fact that theorem~\ref{thm:formula} gives an explicit description of $P$ and $s$ is significant in the context of proofs of cases of the parity conjecture \cite{dokchitsers:parity_conjecture,dokchitsers:parity_with_isogeny,dokchitsers:regulator_constants,dokchitser-maistret:abelian_surfaces}. Tamagawa numbers play an important role in these proofs: one first relates information about parities of ranks of abelian varieties to global Tamagawa numbers (as in e.g.\ \cite[Proof of Theorem 6]{dokchitsers:parity_with_isogeny} or \cite[Corollary~2.21]{dokchitsers:parity_conjecture}), and then relates global Tamagawa numbers to global root numbers by expressing them as a product of local terms. What makes this latter step particularly subtle is that the local Tamagawa numbers and local root numbers do not in general match up place-by-place, and so one needs to show that the total discrepancy over all places vanishes, for example via Artin symbols \cite[Proof of Theorem 2]{dokchitsers:parity_with_isogeny} or by taking an appropriate combination of local invariants over field extensions \cite[Proposition~3.3]{dokchitsers:parity_conjecture}. Proving these relations between local Tamagawa numbers and root numbers often proceeds via a case-by-case analysis of the possible reduction types involved, see e.g.\ the proof of \cite[Proposition 3.3]{dokchitsers:parity_conjecture}, or especially \cite[Theorem 4.5]{dokchitser-maistret:abelian_surfaces}.

Of course, in order to be able to carry out such a case-by-case analysis, it is necessary to have explicit expressions for the Tamagawa numbers of all possible reduction types that might appear. In general, the list of such reduction types is infinite, but is made up of finitely many lists indexed by integer parameters (e.g.\ the Kodaira classification of reduction types of elliptic curves contains two infinite families $\I_n$ and $\I_n^*$ and finitely many exceptional cases). In the setting of hyperelliptic curves, there are infinitely many graphs arising as the dual graph of a semistable hyperelliptic curve $X/K$ of some fixed genus, but these are all subdivisions of a finite list of graphs (see \cite[Table~9.3]{semistable_types} for the complete list in genus~$2$). Thus theorem~\ref{thm:subdivision_formula} allows us, at least in principle, to write down explicit formulas describing the Tamagawa numbers of \emph{every} semistable hyperelliptic curve of a given genus (as a function of its dual graph).
\end{remark}

\begin{remark}
By combining the explicit description of Tamagawa numbers in theorem~\ref{thm:formula} with the ``cluster picture'' machinery of \cite{m2d2}, one obtains an explicit way to read off the Tamagawa number of a semistable hyperelliptic curve $X/K$ from an explicit affine equation $y^2=f(x)$ (for $\deg(f)\geq5$ and $p\neq2$). Indeed, \cite[\S1]{m2d2} explains how to associate to $f$ a certain combinatorial object called a \emph{cluster picture}, from which one can read off a great number of arithmetic quantities of $X$, including whether it is semistable \cite[Theorem~7.1]{m2d2} and the BY tree associated to the dual graph of its special fibre \cite[Theorem~5.18]{m2d2}. Plugging this BY tree into theorem~\ref{thm:formula} yields the Tamagawa number of $X/K$.

This algorithm has been implemented in Sage by Alex Best and Raymond van Bommel as part of a wider implementation of the theory of clusters \cite{users_guide,users_guide_implementation}. Algorithms to efficiently compute Tamagawa numbers are of relevance, for example, in the context of numerical verification of the Birch--Swinnerton-Dyer Conjecture, as in \cite{van_bommel:numerical_verification}.
\end{remark}

\begin{remark}
One can apply theorem~\ref{thm:formula} also in the case when the signed automorphism $\epsilon F$ is trivial, which gives a formula for the size of the Jacobian group of a hyperelliptic graph $\G$ in terms of its associated BY tree. It turns out that in this case, theorem~\ref{thm:formula} essentially just asserts that the order of the Jacobian group is the number of spanning trees in $\G$, as per the matrix-tree theorem \cite[Corollary~II.3.13]{bollobas:graph_theory}. Moreover, the proof of theorem~\ref{thm:formula} essentially just recovers one of the standard proofs of this fact: the order of the Jacobian group is given by the discriminant of a certain pairing (proposition~\ref{prop:positive_coker}) which is equal to any cofactor of the Laplacian matrix (proposition~\ref{prop:jacobian_is_jacobian}), and by changing the lengths of edges to specific values one describes this discriminant combinatorially (lemma~\ref{lem:discriminant}).
\end{remark}

\noindent\textbf{Acknowledgments.} I am very grateful to Vladimir Dokchitser for suggesting this problem to me and, along with Tim Dokchitser, C\'eline Maistret and Adam Morgan, taking the time to explain to me the cluster picture machinery developed in \cite{m2d2}. I would also like to thank Anna Somoza for helping me understand moduli stacks of hyperelliptic curves, Omri Faraggi and Simone Muselli for pointing out a mistake in the earlier statement of theorem~\ref{thm:formula}, and the referee for their many helpful suggestions.

The author was employed under EPSRC grant EP/M016846/2 during the preparation of this manuscript.
\section{BY forests and Tamagawa numbers}
\label{s:BY_forests}

To begin with, we recall how one computes Tamagawa numbers of semistable curves via graph theory, and the particular simplifications one can make in the case of hyperelliptic curves, following \cite{m2d2}. Our definitions differ slightly from those in \cite{m2d2,semistable_types}, and we shall highlight such discrepancies when they arise.

\subsubsection{Conventions}\label{sss:graph_conventions}

Throughout this paper, by a \emph{graph} $\G$ we will formally mean a graph in the sense of \cite[\S2.1]{serre:trees}, so a set $V(\G)$ of \emph{vertices} and a set $\tilde E(\G)$ of \emph{oriented edges}, together with a \emph{source} map $s\colon\tilde E(\G)\to V(\G)$ and an \emph{edge-inversion} map $\bar\cdot\colon\tilde E(\G)\to\tilde E(\G)$ which is an involution without fixed points. All graphs appearing will be finite: that is, the sets $V(\G)$ and $\tilde E(\G)$ will always be finite. The \emph{target} map $t\colon \tilde E(\G)\to V(\G)$ is the map $e\mapsto s(\bar e)$. We write $E(\G)$ for the quotient of $\tilde E(\G)$ under the identification $e\sim\bar e$, and refer to $E(\G)$ as the set of \emph{edges} of $\G$. Note that a graph in this sense may have loops (edges $e$ with $s(e)=t(e)$) and parallel edges (distinct edges $e_1,e_2$ with $s(e_1)=s(e_1)$ and $t(e_1)=t(e_2)$). By an (iso)morphism of graphs, we mean an (invertible) morphism in the sense of \cite[\S2.1]{serre:trees}, so a map on vertex- and oriented edge-sets which preserves all relevant structure.

If we are given a function $l\colon E(\G)\rightarrow\N$ (equivalently an edge-inversion-invariant function $l\colon\tilde E(\G)\rightarrow\N$), then we define the \emph{subdivision} $\G^{(l)}$ to be the graph formed by replacing every edge $e$ of $\G$ with a chain of $l(e)$ edges. Formally, this is defined as follows.
\begin{itemize}
	\item The oriented edges of $\G^{(l)}$ are pairs $(e,i)$ with $e$ an oriented edge of $\G$ and $0\leq i<l(e)$ an integer. The inverse of edge $(e,i)$ is $(\bar e,l(e)-i-1)$.
	\item The vertices of $\G^{(l)}$ are the vertices of $\G$ together with new vertices $s(e,i)$ for each edge $(e,i)$ of $\G^{(l)}$, quotiented by the identifications $s(e,i)\sim s(\bar e,l(e)-i)$ for $0<i<l(e)$ and $s(e,0)\sim s(e)$. The map $(e,i)\mapsto s(e,i)$ is the source function on $\G^{(l)}$.
\end{itemize}

\smallskip

We will also give versions of our main results for metric graphs. A \emph{metric} on a graph $\G$ is a function $l\colon E(\G)\rightarrow\R_{>0}$ (equivalently an edge-inversion-invariant function $l\colon\tilde E(\G)\rightarrow\R_{>0}$), assigning to each edge its \emph{length} -- we say that the metric is \emph{integral} if it takes values in $\N$. An \emph{(integral-)metric graph} is a graph with an (integral) metric. A graph without a metric can be thought of as an integral-metric graph where each edge has length $1$ -- when we discuss the metric on such a graph, we always mean this metric.

A metric graph $\G$ has an underlying metric space\footnote{If $\G$ is disconnected, then the distance between two points in different connected components should be taken to be $\infty$. All graphs appearing in this paper will be connected.} $|\G|$ -- we will often denote this simply $\G$ -- which is the quotient of $V(\G)\sqcup\coprod_{e\in\tilde E(\G)}[0,l(e)]$ by identifying $(e,0)\sim s(e)$ for all $e\in\tilde E(\G)$ and identifying $(e,\tau)\sim(\bar e,l(e)-\tau)$ for all $e\in\tilde E(\G)$ and $\tau\in[0,l(e)]$. Here, $(e,\tau)\in V(\G)\sqcup\coprod_{e\in\tilde E(\G)}[0,l(e)]$ denotes the element $\tau\in[0,l(e)]$ in the $e$th component of the coproduct. The metric on $|\G|$ is the length metric. An \emph{isometry} between (metric) graphs is an isometry between their underlying metric spaces.

\begin{remark}\label{rmk:metric_vs_subdivisions}
The fact that we use the same notation $l\colon E(\G)\to\N$ for a function determining a subdivision $\G^{(l)}$ and an integral metric on $\G$ is not a coincidence: the subdivision $\G^{(l)}$, viewed as a graph in which every edge has length $1$, is canonically isometric to the metric graph $(\G,l)$. The invariants we will study in this paper turn out to be isometry-invariant, and we will often calculate these invariants for the subdivision $\G^{(l)}$ by instead calculating them for the integral-metric graph $(\G,l)$, where the computations are easier. See for example the proof of corollary~\ref{cor:qualitative}.
\end{remark}

\begin{remark}\label{rmk:isometries_vs_isos}
An isometry of graphs need not necessarily map vertices to vertices. For instance, if $\G$ is a cycle, then its underlying metric space is a circle and has self-isometries (e.g.\ irrational rotations) which do not preserve the set of vertices.

This example shows that there are isometries of graphs which are not isomorphisms. However, the difference between isometries and isomorphisms is very mild. Every isomorphism is an isometry, and the converse is true between graphs which do not have a cycle as a connected component.
\end{remark}

\subsection{Graph Jacobians and semistable curves}\label{ss:semistable_dual_graphs}

A fundamental invariant of a graph $\G$ of arithmetic relevance is its \emph{Jacobian} $\gNeron_\G$, which is a finite abelian group functorially associated to $\G$. There are several standard definitions of this group, either as a quotient of the degree $0$ divisors on $\G$ by the principal divisors (see e.g.\ \cite[\S1.3]{baker-norine:riemann-roch}), or as the group of recurrent positions for the abelian sandpile model on $\G$ (see e.g.\ \cite[Corollary~2.16]{many:chip-firing}). The former definition explains the name \emph{Jacobian}, while with the latter definition it is more commonly known as the \emph{sandpile group}~\cite{levine-propp:sandpile?}.

For our purposes, it will be useful to use a third definition of the Jacobian, which is metric in nature.

\begin{definition}\label{def:intersection}
Let $\G$ be a graph, and let $\Z\cdot\tilde E(\G)$ denote the free $\Z$-module on the oriented edges of $\G$. The \emph{first homology} $\H_1(\G,\Z)$ of $\G$ is the quotient $\mathrm Z_1(\G,\Z)/\mathrm B_1(\G,\Z)$, where $\mathrm Z_1(\G,\Z)$ is the kernel of the map $\Z\cdot\tilde E(\G)\to\Z\cdot V(\G)$ sending an oriented edge $e$ to $t(e)-s(e)$, and $\mathrm B_1(\G,\Z)$ is the submodule generated by $e+\bar e$ for $e\in\tilde E(\G)$. The assignment $\G\mapsto\H_1(\G,\Z)$ is functorial\footnote{Throughout this paper, when we say that some construction on the objects of a category is ``functorial'', we have in mind a \emph{particular} way to extend it to a functor on that category.} with respect to graph isomorphisms.

Suppose now that we fix an orientation of each edge of $\G$ (a splitting $\sigma$ of the map $\tilde E(\G)\to E(\G)$). This endows the underlying topological space $|\G|$ of $\G$ with the structure of a $\Delta$-complex \cite[p103]{hatcher}. Via the induced splitting of the map $\Z\cdot\tilde E(\G)\twoheadrightarrow\Z\cdot E(\G)$, the homology $\H_1(\G,\Z)$ is isomorphic to the simplicial homology of this $\Delta$-complex \cite[p105]{hatcher}, and hence to the first singular homology $H_1(|\G|,\Z)$ of $|\G|$ \cite[Theorem~2.27]{hatcher}. One can check easily that this isomorphism $\H_1(\G,\Z)\cong\H_1(|\G|,\Z)$ is independent of the choice of orientation $\sigma$ and natural with respect to graph isomorphisms. In particular, it makes the construction $\G\mapsto\H_1(\G,\Z)$ functorial with respect to continuous maps between underlying topological spaces.

\smallskip

Suppose now that we have fixed an orientation $\sigma$ as above, and write $\Lambda:=\H_1(\G,\Z)$ for short. The lattice $\Lambda$ admits a $\Z$-valued bilinear \emph{intersection-length pairing}, given by
\[
\left\langle\sum_{e\in E(\G)}a_e\cdot e,\sum_{e\in E(\G)}b_e\cdot e\right\rangle := \sum_{e\in E(\G)}a_eb_e \,.
\]
(Informally, $\langle\gamma,\gamma'\rangle$ is the signed length of the intersection $\gamma\cap\gamma'$.) The intersection-length pairing is positive-definite, being the restriction of the pairing on $\Z\cdot E(\G)$ given by the identity matrix, and hence induces an injective map $\Lambda\hookrightarrow\Lambda^\vee=\Hom(\Lambda,\Z)$. We define the \emph{Jacobian} $\gNeron_\G:=\Lambda^\vee/\Lambda$ of $\G$ to be the cokernel of this map, which is a finite abelian group. This construction is functorial with respect to isomorphisms of graphs (i.e.\ the isomorphism $\H_1(\G,\Z)\isoarrow\H_1(\G',\Z)$ induced by an isomorphism $\G\isoarrow\G'$ of graphs is automatically compatible with the pairings), so in particular $\gNeron_\G$ carries an action of $\Aut(\G)$. If $\Frob$ is an automorphism of $\G$, we define the \emph{Tamagawa number}
\[
c_{\G,\Frob} := \#\gNeron_\G^\Frob
\]
of $(\G,\Frob)$ to be the number of $\Frob$-fixed elements in $\gNeron_\G$.
\end{definition}

\begin{metricvariant*}
%\label{rmk:intersection_metric_property}
The above definition of the intersection-length pairing extends naturally to metric graphs via the formula
\[
\left\langle\sum_{e\in E(\G)}a_e\cdot e,\sum_{e\in E(\G)}b_e\cdot e\right\rangle := \sum_{e\in E(\G)}a_eb_el(e) \,.
\]
This is also positive-definite, for the same reasons as above. If the metric on $\G$ is integral, then the intersection-length pairing induces an embedding $\Lambda\hookrightarrow\Lambda^\vee$, and we define the \emph{Jacobian} $\gNeron_\G:=\Lambda^\vee/\Lambda$ of $\G$ exactly as above. We will see shortly (proposition~\ref{prop:functoriality}) that the construction of $\Phi_\G$ is functorial with respect to isometries of integral-metric graphs, and hence if $\Frob$ is a self-isometry of $\G$, then we may define the Tamagawa number in exactly the same way as for non-metrised graphs above.

%All of these constructions (relative homology, intersection-length pairing, Jacobian group, Tamagawa number) are easily seen to be isometry-invariant.
\end{metricvariant*}

%All of these constructions (relative homology, intersection-length pairing, Jacobian group, Tamagawa number) are isometry-invariant, as in the following proposition.

\begin{proposition}\label{prop:functoriality}
Suppose that $f\colon\G\isoarrow\G'$ is an isometry of metric graphs. Then the induced isomorphism $f_*\colon\H_1(\G,\Z)\isoarrow\H_1(\G',\Z)$ preserves the intersection-length pairing. In particular, if the metric graphs are integral, then there is an induced isomorphism $f_*\colon\gNeron_\G\isoarrow\gNeron_{\G'}$, and if $\Frob$ and $\Frob'$ are isomorphisms of $\G$ and $\G'$ respectively which are conjugate via $f$, then $c_{\G,\Frob}=c_{\G',\Frob'}$.
\begin{proof}
The set $V(\G)$ of vertices of $\G$ is by definition a finite subset of the underlying metric space $|\G|$, whose complement is isometric to the disjoint union of the open intervals $(0,l(e))$ for (unoriented) edges $e$ of $\G$. The same is true of $\G'$. Let us say that $f$ is a \emph{subdivision} just when $f(V(\G))\subseteq V(\G')$. It is easy to check that every isometry is the composite of a subdivision followed by the inverse of a subdivision, and hence it suffices to prove the proposition when $f$ is a subdivision, which we now assume.

An oriented edge $e$ determines a locally\footnote{The embedding is not isometric, for example, when $e$ is a loop edge.} isometric embedding $\iota_e\colon(0,l(e))\hookrightarrow|\G|$ whose image is the connected component of $|\G|\setminus V(\G)$ corresponding to $e$. The intersection $V(\G')\cap\im(f\circ\iota_e)$ is a finite set of degree $2$ vertices of $\G'$, whose complement in $\im(f\circ\iota_e)$ is the disjoint union of a finite number of components of $|\G'|\setminus V(\G')$. It follows that there is a finite subset $f_*(\{e\})\subseteq\tilde E(\G')$ such that $\im(f\circ\iota_e)\setminus(V(\G')\cap\im(f\circ\iota_e))=\coprod_{e'\in f_*(\{e\})}\im(\iota_{e'})$. In particular, we see from this that $l(e)=\sum_{e'\in f_*(\{e\})}l(e')$.

Changing the orientations on the edges $e'$ if necessary, we may ensure that the orientation on each $e'\in f_*(\{e\})$ is compatible with that on $e$, in the sense that the isometric embedding $\iota_e^{-1}\circ f^{-1}\circ\iota_{e'}\colon(0,l(e'))\hookrightarrow(0,l(e))$ is order-preserving. This determines the set $f_*(\{e\})$ uniquely. One checks straightforwardly that the map $f_*\colon\Z\cdot\tilde E(\G)\to\Z\cdot\tilde E(\G')$ given by $e\mapsto f_*(e)=\sum_{e'\in f_*(\{e\})}e'$ induces the pushforward map $f_*\colon\H_1(\G,\Z)\isoarrow\H_1(\G',\Z)$ in singular homology. Verifying that this map preserves the intersection-length pairing is then simply a matter of chasing the definitions.
\end{proof}
\end{proposition}

\begin{proposition}\label{prop:jacobian_is_jacobian}
The Jacobian $\gNeron_\G$ of $\G$, as defined in definition~\ref{def:intersection}, agrees with the definition in \cite[\S1.3]{baker-norine:riemann-roch}\footnote{Technically, \cite{baker-norine:riemann-roch} assumes that $\G$ has no loop edges, though there is a natural extension of their definition to permit loop edges. In any case, both the generalised version of the definition in \cite{baker-norine:riemann-roch} and our definition in \ref{def:intersection} are unchanged if one removes all loop edges from $\G$, so it suffices to deal with this case.}. That is, let $\gNeron^\usual_\G$ denote the Jacobian as defined in \cite[\S1.3]{baker-norine:riemann-roch}, which is also functorial with respect to isomorphisms of graphs. Then there is a canonical isomorphism $\gNeron_\G\cong\gNeron^\usual_\G$ for every graph $\G$, natural with respect to graph isomorphisms.
\begin{proof}
We freely use the notation of \cite[\S1.3]{baker-norine:riemann-roch}. If we fix an orientation of the edges of $\G$, then the boundary map $\partial\colon\Z^{\oplus E(\G)}\rightarrow\Z^{\oplus V(\G)}=\Div(\G)$ has kernel $\Lambda=\H_1(\G,\Z)$ and image $\Div^0(\G)$. We thus have a diagram
\begin{center}
\begin{tikzcd}
0 \arrow{r} & \H_1(\G,\Z) \arrow{r} & \Z^{\oplus E(\G)} \arrow{r}\arrow{d}{\wr} & \Div^0(\G) \arrow{r} & 0 \\
0 \arrow{r} & \mathcal{M}(\G)/\Z \arrow{r} & \Z^{E(\G)} \arrow{r} & \H^1(\G,\Z) \arrow{r} & 0
\end{tikzcd}
\end{center}
with exact rows, where the bottom row is the dual of the top row and the central map is the isomorphism identifying the evident bases of either side. One can check that the top-left-to-bottom-right composite is the map $\H_1(\G,\Z)\hookrightarrow\H^1(\G,\Z)$ defined by the intersection-length pairing, and that the bottom-left-to-top-right composite is the Laplacian $\Delta\colon\mathcal{M}(\G)\rightarrow\Div^0(\G)$. The cokernels of these maps are, by definition, $\gNeron_\G$ and $\gNeron^\usual_\G$ respectively, so the third isomorphism theorem provides an isomorphism $\gNeron_\G\cong\gNeron^\usual_\G$. One can check straightforwardly that this isomorphism is independent of the orientations on the edges of $\G$. Since the above diagram is functorial with respect to isomorphisms of oriented graphs, it also follows that the isomorphism $\gNeron^\usual\G\cong\gNeron^\usual_\G$ is natural with respect to graph automorphisms.
\end{proof}
\end{proposition}

\begin{remark}
For our purposes, the advantage of definition~\ref{def:intersection} over the more standard one is that it behaves well with respect to subdivision of edges, or equivalently, change of metric (remark~\ref{rmk:metric_vs_subdivisions}). Suppose we are given a graph $\G_0$ and a function $l\colon E(\G_0)\rightarrow\N$, and let $\G_0^{(l)}$ denote the subdivision of $\G_0$ where each edge $e$ is replaced by a chain of $l(e)$ edges. This subdivision $\G_0^{(l)}$ is isometric to the metric graph $(\G_0,l)$ by remark~\ref{rmk:metric_vs_subdivisions}, and hence $\G_0^{(l)}$ and $(\G_0,l)$ have the same Jacobian by proposition \ref{prop:functoriality}. From definition~\ref{def:intersection} (metric version), this common Jacobian is given by a presentation
\[
0\rightarrow\H_1(\G_0,\Z)\xrightarrow{\beta_l}\H^1(\G_0,\Z)\rightarrow\gNeron_{(\G_0,l)}\rightarrow0 \,,
\]
where $\beta_l$ is the map induced by the intersection length pairing on the metric graph $(\G_0,l)$.

For fixed $\G_0$ and varying $l$, this presentation exhibits $\gNeron_{\G_0^{(l)}}=\gNeron_{(\G_0,l)}$ as the cokernel of a variable map between fixed $\Z$-modules, and this presentation turns out to be well-suited to studying the dependence of $c_{\G_0^{(l)}}$ on~$l$ (see \S\ref{ss:qualitative_proof} for example). By contrast, the standard definition of the Jacobian \cite[\S1.3]{baker-norine:riemann-roch} exhibits it as the cokernel
\[
\mcal{M}(\G_0^{(l)}) \xrightarrow\Delta \Div^0(\G_0^{(l)}) \rightarrow \gNeron^\usual_{\G_0^{(l)}} \rightarrow 0
\]
of a map whose domain and codomain also depend on $l$. This makes the latter presentation much less well-suited to studying problems where we vary~$l$ for fixed~$\G_0$.
\end{remark}

\subsubsection{Relation to geometry}

The significance of graph Jacobians in the context of arithmetic geometry is that they compute groups of geometric components of special fibres of N\'eron models of Jacobians of semistable curves $X/K$, and hence also their Tamagawa numbers. The key object which enables the passage from arithmetic geometry to combinatorics is the dual graph of the geometric special fibre of the minimal regular model of $X$.

\begin{definition}[Dual graphs]\label{def:dual_graph}
Let $X/K$ be a semistable curve of genus~$\geq1$, with minimal regular model $\mathfrak X/\O_K$. The \emph{dual graph} $\G$ of the geometric special fibre $\mathfrak X_{\overline k}$ is the  graph whose vertices are the connected components of the normalisation $\widetilde{\mathfrak X}_{\overline k}$ of $\mathfrak X_{\overline k}$, and whose oriented edges are the $\overline k$-points of $\widetilde{\mathfrak X}_{\overline k}$ lying over singular points of $\mathfrak X_{\overline k}$. If $e$ is a $\overline k$-point of $\widetilde{\mathfrak X}_{\overline k}$ lying over a singular point, then its source $s(e)$ is the connected component of $\widetilde{\mathfrak X}_{\overline k}$ containing it, and its inverse $\bar e$ is the (unique) other point of $\widetilde{\mathfrak X}_{\overline k}$ mapping to the same point in $\mathfrak X_{\overline k}$. The action of the absolute Galois group $G_K$ on $\mathfrak X_{\overline k}$ induces an action on the graph $\G$, acting on vertices and edges in the obvious way. This action is unramified, so is determined by the action of any chosen (arithmetic) Frobenius in $G_K$.
\end{definition}

\begin{remark}
Note that the graph $\G$ in definition~\ref{def:dual_graph} is the dual graph of the \emph{geometric} special fibre $\mathfrak X_{\overline k}$, as opposed to the special fibre $\mathfrak X_k$.
\end{remark}

\begin{theorem}\label{thm:tamagawa_numbers_via_graphs}
Let $X/K$ be a semistable curve of genus~$\geq1$ with minimal regular model $\mathfrak X/\O_K$, and let $\G$ denote the dual graph of the geometric special fibre $\mathfrak X_{\bar k}$. Let $\Phi_{X/K}$ denote the group-scheme of connected components of the special fibre of the N\'eron model of $\Jac(X)$. Then there is a canonical isomorphism
\[
\Phi_{X/K}(\overline k) \cong \gNeron_\G
\]
which is equivariant for the natural actions of Frobenius on either side (on the left, induced from the action on $\overline k$, and on the right, induced from the action on $\G$ from definition~\ref{def:dual_graph}). In particular, $c_{X/K}=c_{\G,\Frob}$.
\begin{proof}
Let $\gNeron^\usual_\G$ denote the group defined in \cite[\S1.3]{baker-norine:riemann-roch}, which is $\Aut(\G)$-equivariantly isomorphic to $\gNeron_\G$ by proposition~\ref{prop:jacobian_is_jacobian}. \cite[Theorem~9.6.1]{bosch-luetkebohmert-raynaud:neron_models} gives the canonical identification between the group of geometric components of the N\'eron model of $\Jac(X)$ and $\gNeron^\usual_\G$, and hence with $\gNeron_\G$. That this identification is Galois-equivariant is proved in \cite[Theorem~1.1]{bosch-liu:rational_neron}. The equality of Tamagawa numbers then follows by definition: $c_{X/K} := \#\gNeron(k) = \#\gNeron(\overline k)^\Frob$.
\end{proof}
\end{theorem}

We illustrate theorem~\ref{thm:tamagawa_numbers_via_graphs} with the following example, which also shows that the assumption that $X$ is hyperelliptic in theorem~\ref{thm:growth_in_towers} and corollary~\ref{cor:tamagawa_numbers_higher_powers} is necessary.

\begin{example}\label{ex:non-hyperelliptic}
Let $K$ be a $p$-adic number field with residue field $k$; write $k_5$ for the degree $5$ extension of $k$ and $K_5/K$ for the corresponding unramified extension of $K$. Let $X/K$ be a (smooth, projective, geometrically integral) curve with a regular semistable model\footnote{In other words, $\mathfrak X/\O_K$ is proper, flat and regular, and its special fibre $\mathfrak X_k$ is a reduced normal crossings divisor.} $\mathfrak X/\O_K$ such that the dual graph $\G$ of the geometric special fibre $\mathfrak X_{\bar k}$ is the following five-spoked wheel graph
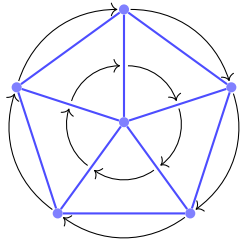
\begin{center}
\begin{tikzpicture}
	\BlueVertices
	\Vertex[x=0.0,y=0.0,L=\relax]{o}
	\Vertex[x=0.0,y=1.5,L=\relax]{n}
	\Vertex[x=1.427,y=0.464,L=\relax]{ne}
	\Vertex[x=0.881,y=-1.214,L=\relax]{se}
	\Vertex[x=-0.881,y=-1.214,L=\relax]{sw}
	\Vertex[x=-1.427,y=0.464,L=\relax]{nw}
	
	\BlueEdges
	\Edge(o)(n)
	\Edge(o)(ne)
	\Edge(o)(se)
	\Edge(o)(sw)
	\Edge(o)(nw)
	\Edge(n)(ne)
	\Edge(ne)(se)
	\Edge(se)(sw)
	\Edge(sw)(nw)
	\Edge(nw)(n)
	
	\EArr{o}{n}{o}{ne}{in=108,out=0}
	\EArr{o}{ne}{o}{se}{in=36,out=288}
	\EArr{o}{se}{o}{sw}{in=324,out=216}
	\EArr{o}{sw}{o}{nw}{in=252,out=144}
	\EArr{o}{nw}{o}{n}{in=180,out=72}
	\VArr{n}{ne}{in=108,out=0}{1}
	\VArr{ne}{se}{in=36,out=288}{1}
	\VArr{se}{sw}{in=324,out=216}{1}
	\VArr{sw}{nw}{in=252,out=144}{1}
	\VArr{nw}{n}{in=180,out=72}{1}
\end{tikzpicture}
\end{center}
where the edge-lengths are all $1$, the vertices all have genus $0$, and the induced action of Frobenius rotates the wheel by a one-fifth turn. In other words, the normalisation of $\mathfrak X_k$ is $\P^1_k\sqcup\P^1_{k_5}$, where the copy of $\P^1_k$ meets the copy of $\P^1_{k_5}$ transversely at a point of degree $5$ over $k$, and the copy of $\P^1_{k_5}$ meets itself transversely at a (different) point of degree $5$ over $k$ such that, after base-changing from $k$ to $k_5$, each of the five components of $(\P^1_{k_5})_{k_5}\cong(\P^1_{k_5})^{\sqcup 5}$ meets its Frobenius conjugate.

The existence of such a curve $\mathfrak X/\O_K$ with special fibre $\mathfrak X_k$ is provided by the classical deformation theory of stable curves, for instance as summarised in \cite[pp.~79--81]{deligne-mumford}. The completed local ring of $\mathfrak X_k$ at each of its two singular (closed) points is isomorphic to $k_5[\![x,y]\!]/(xy)$, which admits a deformation to a regular complete local ring over $\O_K$, for instance $\O_{K_5}[\![x,y]\!]/(xy-\varpi)$ with $\varpi$ a uniformiser of $\O_K$. This determines an $\O_K$-point of the moduli space of deformations of the local rings of $\mathfrak X_k$, denoted $\mathcal M_{lo}$ in \cite{deligne-mumford}. There is also a moduli space $\mathcal M_{gl}$ of deformations of the curve $\mathfrak X_k$, i.e.\ flat projective schemes $\hat{\mathfrak X}$ over local Artin rings with residue field $k$ whose special fibre is $\mathfrak X_k$. There is a map $\mathcal M_{gl}\rightarrow\mathcal M_{lo}$ sending a formal deformation $\hat{\mathfrak X}$ of $\mathfrak X_k$ to its local rings at points of its special fibre $\mathfrak X_k$, and moreover this map $\mathcal M_{gl}\rightarrow\mathcal M_{lo}$ has a splitting \cite[Proposition~1.5]{deligne-mumford}. Thus there is a formal deformation $\hat{\mathfrak X}$ of $\mathfrak X_k$ over $\O_K$ having the prescribed local rings at the singular points of its special fibre. Since $\hat{\mathfrak X}$ is projective, it is the formal completion of a (regular) projective scheme $\mathfrak X/\O_K$ at its special fibre \cite[Th\'eor\`eme~5.4.5]{EGA3}. This is the desired curve $\mathfrak X$.

%It then follows from \cite[Proposition~1.5]{deligne-mumford} that there is a deformation of $\mathfrak X_k$ to a flat projective formal scheme $\hat{\mathfrak X}/\O_K$ such that the completed local rings of $\hat{\mathfrak X}$ at the singular points of its special fibre are isomorphic to $\O_{K_5}[\![x,y]\!]/(xy-\varpi)$, in particular regular. Since $\hat{\mathfrak X}$ is projective, it is the formal completion of a (regular) projective scheme $\mathfrak X/\O_K$ at its special fibre \cite[Th\'eor\`eme~5.4.5]{EGA3}. This is the desired curve $\mathfrak X$.

%The curve $\mathfrak X_k$ has a universal formal deformation over $\mathcal M_{\mathrm{gl}}=\Spec(A_{\mathrm{gl}})$ with $A_{\mathrm{gl}}$ a power series ring over the Witt vectors $\Witt(k)$ of $k$. Also, for each closed point $x$ of $\mathfrak X_k$ there is a versal formal deformation of the completed local ring $\widehat\O_{\mathfrak X_k,x}$ over $\mathcal M_x=\Spec(A_x)$ with $A_x$ again a power series ring over the Witt vectors of $k$ (and $A_x=\Witt(k)$ when $x$ is smooth). If we write $A_{\mathrm{lo}}$ for the completed tensor product of the $\Witt(k)$-algebras $A_x$, then the natural map $\mathcal M_{\mathrm{gl}}\rightarrow\mathcal M_{\mathrm{lo}}:=\Spec(A_{\mathrm{lo}})$ sending a formal deformation $\mathfrak X$ of $\mathfrak X_k$ to the completed tensor product of its completed local rings is split surjective \cite[Proposition~1.5]{deligne-mumford}.

\smallskip

We can compute the Tamagawa number of $X$ over $K$ and $K_5$ using theorem~\ref{thm:tamagawa_numbers_via_graphs}, noting that $X$ has genus $5$ and $\mathfrak X$ is its minimal regular model. With respect to the basis of $\H_1(\G,\Z)$ consisting of the five small triangles, the intersection-length pairing is given by the matrix
\[
\begin{pmatrix}
3 & -1 & 0 & 0 & -1 \\
-1 & 3 & -1 & 0 & 0 \\
0 & -1 & 3 & -1 & 0 \\
0 & 0 & -1 & 3 & -1 \\
-1 & 0 & 0 & -1 & 3
\end{pmatrix}\,.
\]
The Jacobian group of $\G$ is then the cokernel of the map given by this matrix, which a Smith normal form calculation shows is $\F_{11}\oplus\F_{11}$. To determine the Frobenius action, we note that the above description presents the Jacobian as the quotient $\F_{11}[T]/(T^2-3T+1)$, where the Frobenius acts via multiplication by $T$. Factorising $T^2-3T+1$ over $\F_{11}$ thus shows that the action of Frobenius on the Jacobian group is diagonalisable, with eigenvalues $-2$ and $5$.

We see from this that the Frobenius has only one fixed point on $\gNeron_\G$, but its fifth power acts trivially. Thus $c_{X/K}=1$ and $c_{X/K_5}=121$.
\end{example}

\subsection{Hyperelliptic graphs}

Among all graphs, there is a certain subclass of \emph{hyperelliptic graphs} which exhibit similar behaviour to hyperelliptic curves \cite{baker-norine:hyperelliptic}. Although there are several equivalent ways to define these (see \cite[Theorem~5.12]{baker-norine:hyperelliptic}), we will only use one which is purely topological in nature.

\begin{definition}[cf.\ {\cite[Definition~3.2]{semistable_types}}]\label{def:hyperelliptic_graph}
A \emph{hyperelliptic graph} is a pair $(\G,\iota)$ where $\G$ is a connected graph and $\iota$ is an involution of $\G$ such that the topological quotient $\G/\iota$ (quotient of the underlying topological space) is a tree, i.e.\ contractible.
\end{definition}

%\begin{remark}
%Unlike \cite{baker-norine:hyperelliptic}, we regard the involution $\iota$ as part of the data of a hyperelliptic graph, rather than merely requiring its existence. This is not a major difference, since for almost all graphs, such a hyperelliptic involution, if it exists, is unique \cite[Corollary~5.15]{baker-norine:hyperelliptic}.
%\end{remark}

\begin{remark}
Our definition of a hyperelliptic graph is slightly more general than that in \cite[Definition~3.2]{semistable_types}, in that we don't endow $\G$ with a genus function, nor do we place any requirements on the degrees of its vertices. The reason for this is simply that these extra data and conditions are irrelevant for the study of Tamagawa numbers, so we omit them to avoid overloading the notation.
\end{remark}

\subsubsection{Relation to geometry}

In fact, the similarities between hyperelliptic graphs and hyperelliptic curves are not just formal, and hyperelliptic graphs are indeed those graphs which arise as dual graphs of the reductions of semistable hyperelliptic curves.

\begin{lemma}\label{lem:hyperelliptic_reduction}
Let $X/K$ be a semistable hyperelliptic curve of genus~$g\geq1$ with minimal regular model $\mathfrak X/\O_K$, and let $\G$ denote the dual graph of the geometric special fibre $\mathfrak X_{\bar k}$. Then $\G$, with the involution $\iota$ induced from the hyperelliptic involution on $X$, is a hyperelliptic graph.
\begin{proof}
\cite[Theorem 5.18]{m2d2} proves this in the case when $p\neq2$, $g\geq2$ and $X/\iota\simeq\P^1_K$. We will sketch a proof in general using the theory of analytic geometry in the sense of Berkovich. Let $X_{\C_K}^\Berk$ denote the Berkovich analytification of $X$ over a completed algebraic closure $\C_K$ of $K$, and let $\iota$ denote the hyperelliptic involution on $X_{\C_K}^\Berk$. Formula $(\ast)$ in the proof of \cite[Proposition 3.4.6]{berkovich:analytic_geometry} describes the fibres of the map $X_{\C_K}^\Berk\twoheadrightarrow(X/\iota)_{\C_K}^\Berk\simeq\P^{1,\Berk}_{\C_K}$: if $x$ is a point of $\P^{1,\Berk}_{\C_K}$ with residue field $\mathcal H(x)$, then its fibre is the Berkovich spectrum of a finite $\mathcal H(x)$-algebra of dimension $2$ or $1$ whose $\iota$-fixed subalgebra is exactly $\mathcal H(x)$.

It follows from this description that the map $|X_{\C_K}^\Berk|\twoheadrightarrow|\P^{1,\Berk}_{\C_K}|$ between underlying topological spaces is a set-theoretic quotient by the hyperelliptic involution $\iota$ -- since the domain is compact and the codomain is Hausdorff, it is automatically a topological quotient.

Now it is well-known that the dual graph $\G$ of $\mathfrak X_{\bar k}$ embeds naturally inside $|X_{\C_K}^\Berk|$ (see e.g.\ \cite[Segment 4.9]{baker-payne-rabinoff:berkovich_curves}). The image of this embedding is $\iota$-stable by naturality, and hence the quotient $\G/\iota$ is canonically identified with the image of $\G$ under $|X_{\C_K}^\Berk|\twoheadrightarrow|\P^{1,\Berk}_{\C_K}|$. This image is a topological tree, since on the one hand it is topologically a connected graph, and on the other, being a connected subspace of $|\P^{1,\Berk}_{\C_K}|$, it is contractible by \cite[Theorem 4.2.1]{berkovich:analytic_geometry}. Thus $(\G,\iota)$ is hyperelliptic.
\end{proof}
\end{lemma}

\subsection{BY trees}

From the perspective of controlling Tamagawa numbers of hyperelliptic curves, it will be convenient to replace hyperelliptic graphs with yet-simpler combinatorial objects. The pertinent notion is that of a \emph{BY tree}, as defined in \cite{semistable_types}.

\begin{definition}[cf.\ {\cite[Definition~3.18]{semistable_types}}]\label{def:BY_tree}
A \emph{BY tree} $T=(T,S)$ consists of a tree $T$ (connected, acyclic graph) together with a non-empty subgraph $S\subseteq T$. A \emph{signed isomorphism} $\epsilon F\colon(T,S)\isoarrow(T',S')$ between two BY trees consists of a pair $(F,\epsilon)$ where $F\colon T\isoarrow T'$ is an isomorphism taking $S$ isomorphically onto $S'$, and $\epsilon$ is a function $\pi_0(T\setminus S)\rightarrow\{\pm1\}$, where $T\setminus S$ denotes the complement of $S$ in the underlying topological space of $T$. The \emph{composite} of two signed isomorphisms $\epsilon F\colon(T,S)\isoarrow(T',S')$, $\epsilon'F'\colon(T',S')\isoarrow(T'',S'')$ is the signed isomorphism $\epsilon''F''\colon(T,S)\isoarrow(T'',S'')$ where $F''=F'\circ F$ and $\epsilon''(C)=\epsilon(C)\cdot\epsilon'(F(C))$ for all connected components $C$ of $T\setminus S$.

The collection of BY trees with signed isomorphisms forms a category in which every morphism is invertible. We write $\Aut^{\pm}(T)$ for the group of signed automorphisms of a BY tree $T$.
\end{definition}

\begin{metricvariant*}
One can also define an \emph{integral-metric BY tree} to be a BY tree $T$ endowed with an integral metric $l$. One defines a (signed) \emph{isometry} of BY trees exactly as above. It is easy to check that the category of BY trees with (signed) isomorphisms is equivalent to the category of integral-metric BY trees with (signed) isometries, by endowing every non-metric BY tree with the metric whereby each edge has length~$1$.
\end{metricvariant*}

\begin{remark}
This definition of BY trees is again slightly more general than that in \cite[Definition~3.18]{semistable_types}, in that we do not endow $T$ with a genus function, nor do we place any restrictions on the degrees of vertices. Again, this is because these extra restrictions do not affect Tamagawa numbers.
\end{remark}

It turns out that the category of BY trees defined above is slightly too large, and includes BY trees that don't correspond to any hyperelliptic graph. For this reason, we introduce certain technical parity conditions on BY trees and their automorphisms, which we will feel free to use when necessary.

\begin{definition}[Parity conditions]\label{def:parity}
An automorphism $F$ of a tree $T$ either fixes a vertex or inverts an edge, but not both; we say that $F$ is \emph{even} when it fixes a vertex. If $\epsilon F$ is a signed automorphism of a BY tree $T=(T,S)$, then we say that $\epsilon F$ is even just when $F$ is even as an automorphism of $T$. We write $\Aut_\even^{\pm}(T)$ for the group of even signed automorphisms of $T$.

We say that a vertex $v$ of a BY tree $T=(T,S)$ is \emph{special} just when either:
\begin{itemize}
	\item $v$ has degree~$\geq3$;
	\item $v$ has degree~$2$, $v\in S$ and at most one of the edges incident to $v$ lies in $S$;
	\item $v$ has degree $1$ and either $v\notin S$ or both $v$ and its incident edge are in $S$; or
	\item $v$ has degree $0$ (in the degenerate case that $T$ consists of a single vertex).
\end{itemize}
We say that $T$ is \emph{even} just when all special vertices lie an even distance from one another.
\end{definition}

\begin{metricvariant*}
One can extend the definitions of evenness also to integral-metric BY trees (definition~\ref{def:BY_tree}, metric version). If $F$ is a self-isometry of an integral-metric tree $T$, then either the distance from $v$ to $F(v)$ is even for all vertices $v$, or the distance is odd for all $v$; we say that $F$ is \emph{even} in the former case.

If $T=(T,S)$ is an integral-metric BY tree, then one defines \emph{special vertices} exactly as in definition~\ref{def:parity} above, and says that $T$ is \emph{even} just when all of its special vertices lie an even distance from one another.

It is easy to check that these definitions extend the earlier definitions for non-metric trees, and are isometry-invariant (e.g.\ if two integral-metric BY trees are isometric, then one is even if and only if the other is).
\end{metricvariant*}

\begin{remark}\label{rmk:even_autos}
It is \emph{almost} true that every automorphism of an even BY tree $T$ is even, the only exception being when $T$ consists of a chain of an odd number of edges connecting the two points of $S$, and $F$ reverses this chain of edges. (To see that this is the only exception, note that an automorphism of an even BY tree preserves the set of special vertices, so is even \emph{provided that~$T$ has at least one special vertex}.)
\end{remark}

\subsubsection{Relation to hyperelliptic graphs}

As stated above, BY trees provide a convenient combinatorial framework for analysing hyperelliptic graphs. The precise relation between these two concepts is made precise in the following construction.

\begin{construction}\label{cons:hyperelliptic_graph_to_BY_tree}
Let $(\G,\iota)$ be a hyperelliptic graph. We endow the topological quotient $T:=\G/\iota$ with the structure of a BY tree as follows. Let $\G^{(2)}$ denote the graph formed by subdividing each edge of $\G$ into two edges, as in \S\ref{sss:graph_conventions}. We define $T$ to be the graph-theoretic quotient\footnote{By the \emph{graph-theoretic quotient} of a graph $\G$ by an action of a group $G$, we mean the graph whose vertices and edges are the $G$-orbits of vertices and edges in $\G$, with the obvious incidence relations. In order for this to be well-defined, we require that no element of $G$ inverts any edge of $\G$ -- this is automatic when $\G$ has a bipartition whose classes are preserved by $G$.} $\G^{(2)}/\iota$, which is canonically homeomorphic to $\G/\iota$. The ramification locus of the topological quotient map $\G\twoheadrightarrow\G/\iota\cong T$ is a subgraph $S\subseteq T$.

We now claim that the BY tree $T=(T,S)$ is even. There is a natural bipartition\footnote{A bipartition of a graph is a division of its vertex-set into two classes such that adjacent vertices lie in different classes.} $V(T)=V_0(T)\sqcup V_1(T)$ of the vertices of $T$, where $V_0(T)$ consists of those vertices of $T$ which are the image of vertices of $\G$, and $V_1(T)$ consists of those vertices which are the image of midpoints of edges of $\G$. Every special vertex $v$ lies in $V_0(T)$, for, if $v$ were the image of the midpoint of an edge $e$ of $\G$, then we would have one of three possibilities:
\begin{itemize}
	\item $e$ is inverted by the hyperelliptic involution;
	\item $e$ is sent to another edge by the hyperelliptic involution;
	\item $e$ is pointwise fixed by the hyperelliptic involution.
\end{itemize}
In the first case, $v$ would lie in $S$ and have one incident edge, which is not in $S$. In the second case, $v$ would have degree $2$ and not lie in $S$. In the third case, $v$ would have degree $2$ and both its incident edges would lie in $S$. All three are impossible if $v$ is special. Since all special vertices lie in $V_0(T)$, they lie an even distance from one another in $T$.

Now let $\phi\colon(\G,\iota)\isoarrow(\G',\iota')$ be an isometry of hyperelliptic graphs, i.e.\ an isometry of graphs compatible with the hyperelliptic involutions. Then $\phi$ induces an isometry $F\colon T\isoarrow T'$ between the corresponding BY trees. This is necessarily an isomorphism. We make this into a signed isomorphism as follows. Choose splittings $\sigma$, $\sigma'$ of the topological quotient maps $\G\twoheadrightarrow\G/\iota$, $\G'\twoheadrightarrow\G'/\iota'$. We then define the sign function $\epsilon\colon\pi_0(T\setminus S)\rightarrow\{\pm1\}$ by
\[
\epsilon(C):=\begin{cases}+1&\text{if $\sigma(F(C))=\phi(\sigma(C))$,}\\-1&\text{if $\sigma(F(C))=\iota\phi(\sigma(C))$.}\end{cases}
\]

Finally, if $\phi$ is an automorphism of $(\G,\iota)$, then the induced automorphism of the corresponding BY tree $T$ preserves each class of the bipartition $V(T)=V_0(T)\sqcup V_1(T)$ above. Since it preserves a bipartition, it cannot invert an edge, and hence is even.
\end{construction}

\begin{proposition}[cf.\ {\cite[Proposition~4.11]{semistable_types}}]\label{prop:hyperelliptic-BY_equivalence}
Construction~\ref{cons:hyperelliptic_graph_to_BY_tree} induces an equivalence of categories from hyperelliptic graphs with isometries (respecting the hyperelliptic involutions) to even BY trees with signed isomorphisms. Under this equivalence, the automorphisms of a hyperelliptic graph correspond to even automorphisms of the corresponding BY tree.
\begin{proof}
Fix a choice of splitting $\sigma$ of the quotient map $\G\twoheadrightarrow\G/\iota$ for each hyperelliptic graph $(\G,\iota)$. Construction~\ref{cons:hyperelliptic_graph_to_BY_tree} then provides a functor from hyperelliptic graphs to even BY trees. We verify that this functor is fully faithful and essentially surjective.

For faithfulness, it suffices to show that for any hyperelliptic graph $(\G,\iota)$ with associated BY tree $T$, the map $\Aut(\G)\rightarrow\Aut^{\pm}(T)$ is injective. But an element of the kernel is an automorphism $\phi$ of $\G$ which induces the identity on $T=\G/\iota$ and respects the chosen splitting $\sigma$ for $\G$; it is easy to see that this implies that $\phi$ is the identity. This proves injectivity, hence faithfulness.

For fullness, we note that any even signed isomorphism $\epsilon F$ between the BY trees $T$, $T'$ of two hyperelliptic graphs $\G$, $\G'$ factors as $\epsilon\cdot\id_T$ followed by the unsigned isomorphism $F$. It is easy to see that $\epsilon\cdot\id_T$ lifts to an automorphism of $\G$, namely the deck transformation of the ramified cover $\G\twoheadrightarrow\G/\iota=T$ which acts as the identity over each component of $T\setminus S$ where $\epsilon=1$, and as $\iota$ over each component where $\epsilon=-1$. Equally, it is easy to see that $F$ lifts to an isometry $\G\isoarrow\G'$, namely the unique isometry which is compatible with the splittings and induces the isomorphism $F$ on the quotients by the respective hyperelliptic involutions. Combining these shows that $\epsilon F$ lifts to an isomorphism $\G\isoarrow\G'$, which proves fullness.

Finally, to show essential surjectivity, we provide an inverse construction on the level of objects. Given an even BY tree $(T,S)$, we let $\G^{(2)}:=T\cup_ST$ denote the graph formed by gluing two copies of $T$ along their copies of $S$. This comes with an involution $\iota$ exchanging the two copies of $T$. Since $T$ is even, we may choose a bipartition $V(T)=V_0(T)\sqcup V_1(T)$ of its vertices such that $V_0(T)$ contains all special vertices of $T$. This induces an $\iota$-stable bipartition $V(\G^{(2)})=V_0(\G^{(2)})\sqcup V_1(\G^{(2)})$ of the vertices of $\G^{(2)}$, for which every vertex in $V_1(\G^{(2)})$ has degree~$2$. Thus $\G^{(2)}$ is the subdivision of a graph $\G$. It is easy to see that $(\G,\iota)$ is a hyperelliptic graph, and that its associated BY tree is isomorphic to $T$, proving essential surjectivity.
\end{proof}
\end{proposition}

\begin{remark}
Note that proposition~\ref{prop:hyperelliptic-BY_equivalence} describes the category of hyperelliptic graphs with isometries, rather than isomorphisms. These are not quite the same category, due to the existence of isometries which are not isomorphisms in a few edge cases. For example if $\G$ has a single vertex and a single loop edge inverted by the hyperelliptic involution $\iota$, then $(\G,\iota)$ has a self-isometry which is not an automorphism, namely the half-rotation.
\end{remark}

\begin{remark}
The equivalence of categories in proposition~\ref{prop:hyperelliptic-BY_equivalence} is non-canonical, since it depends on a choice of splitting for each hyperelliptic graph. However, this non-canonicity is relatively mild: the map on objects is independent of the choice of splittings, and if a hyperelliptic graph $\G$ corresponds to a BY tree $T$, then the isomorphism $\Aut(\G)\simeq\Aut_\even^{\pm}(T)$ is independent of the choice of splitting, up to conjugation.
\end{remark}

\subsubsection{Jacobians of BY trees}

Just as we defined the Jacobian group of a graph, we will want to define a corresponding group in the setting of BY trees, which corresponds to the original group across the equivalence of categories in proposition~\ref{prop:hyperelliptic-BY_equivalence}. The definition is very similar to that in definition~\ref{def:intersection}.

\begin{definition}[cf.\ {\cite[Definition~3.31]{semistable_types}}]\label{def:BY_intersection}
Let $T=(T,S)$ be a BY tree. We let $\Lambda_T=\H_1(T,S,\Z)$ denote the relative homology lattice. This carries an inner product given by ``intersection length in $T\setminus S$'', defined formally as in definition~\ref{def:intersection} (neglecting edges in $S$). This induces an embedding $\Lambda_T\hookrightarrow\Lambda_T^\vee$, and we call the cokernel $\Lambda_T^\vee/\Lambda_T$ the \emph{Jacobian} $\gNeron_T$ of $T$.

The group of signed automorphisms of $T$ acts on $\Lambda_T$, and hence $\gNeron_T$, in a natural way: if $\gamma$ is a relative homology class supported on the closure of a component $C$ of $T\setminus S$, then we set $\epsilon F(\gamma) := \epsilon(C)\cdot F_*(\gamma)$ for every signed automorphism $\epsilon F$. For a given signed automorphism $\epsilon F$, we define the Tamagawa number
\[
c_{T,\epsilon F} := \#\gNeron_T^{\epsilon F}
\]
to be the number of fixed points in the Jacobian. We often write $c_T$ if the signed automorphism $\epsilon F$ is clear.
\end{definition}

\begin{metricvariant*}
One defines the \emph{Jacobian} of an integral-metric BY tree $T$ in exactly the same way as above, where the intersection length pairing on $\Lambda_T$ involves the metric similarly to in definition~\ref{def:intersection}, metric version. We then define the \emph{Tamagawa number} $c_{T,\epsilon F}:=\#\gNeron_T^{\epsilon F}$ exactly as above.

It is easy to see that the group $\gNeron_T$ is functorial in signed isometries, and hence the Tamagawa number $c_{T,\epsilon F}$ is invariant under isometries (of pairs $(T,\epsilon F)$).
\end{metricvariant*}

\begin{proposition}\label{prop:hyperelliptic-BY_equivalence_jacobians}
If a hyperelliptic graph $\G$ and an even BY tree $T$ correspond under the equivalence of categories from proposition~\ref{prop:hyperelliptic-BY_equivalence}, then there is an isomorphism
\[
\gNeron_\G \simeq \gNeron_T
\]
of Jacobian groups, which is equivariant for the action of $\Aut(\G)\simeq\Aut^{\pm}_\even(T)$. In particular, if $\Frob\in\Aut(\G)$ corresponds to $\epsilon F\in\Aut^{\pm}_\even(T)$, then $c_{\G,\Frob}=c_{T,\epsilon F}$.
\begin{proof}
Let $\sigma$ be the section of $\G\twoheadrightarrow\G/\iota\cong T$ chosen in the proof of proposition~\ref{prop:hyperelliptic-BY_equivalence}. Now $\G$ is covered by $\sigma T$ and $\iota\sigma T$, with intersection $S$, so by excision $\sigma$ induces an isomorphism\[\H_1(T,S,\Z)\isoarrow\H_1(\G,\iota\sigma T,\Z)\]on relative homology. But since $\iota\sigma T$ is contractible, the exact sequence on homology of a pair provides the canonical map\[\H_1(\G,\Z)\isoarrow\H_1(\G,\iota\sigma T,\Z)\]is an isomorphism. We thus obtain an isomorphism $\H_1(T,S,\Z)\isoarrow\H_1(\G,\Z)$ by combining the above isomorphisms. In detail, this sends the class of a cycle $\gamma$ on $T$ with $\partial\gamma\subseteq S$ to the class of the cycle $\sigma\gamma-\iota\sigma\gamma$ on $\G$. It is then easy to check that this isomorphism preserves the inner product on either side and is compatible with the actions of $\Aut(\G)$ and $\Aut^{\pm}(T)$ on either side.
\end{proof}
\end{proposition}

\subsubsection{The BY tree associated to a semistable hyperelliptic curve}\label{sss:BY_tree_of_curve}

We will use the theory of BY trees to study Tamagawa numbers of semistable hyperelliptic curves. If $X/K$ is a semistable hyperelliptic curve of genus~$\geq1$, then we define the BY tree $T=(T,S)$ \emph{associated to $X/K$} to be the BY tree corresponding to the dual graph $\G$ of the geometric special fibre of the minimal regular model of $X$ under the equivalence of categories in proposition~\ref{prop:hyperelliptic-BY_equivalence} (i.e.\ $T=\G/\iota$ is the quotient of $\G$ by the hyperelliptic involution, as in construction~\ref{cons:hyperelliptic_graph_to_BY_tree}). The Galois action on $\G$ induced from the action on the geometric special fibre induces an action on $T$ by signed automorphisms. These actions are unramified, and we write $\epsilon F\in\Aut^{\pm}(T)$ for the signed automorphism of $T$ corresponding to the action of Frobenius. It follows from proposition~\ref{prop:hyperelliptic-BY_equivalence} that $T$ is automatically even, and so too is the signed automorphism $\epsilon F$. Theorem~\ref{thm:tamagawa_numbers_via_graphs} and proposition~\ref{prop:hyperelliptic-BY_equivalence_jacobians} ensure that the Tamagawa number of $X/K$ is equal to the Tamagawa number of the pair $(T,\epsilon F)$.% We will prove 

\begin{remark}
With appropriate conventions, the results we prove in this paper are still valid when $X/K$ has genus~$0$. Namely, we adopt the convention that its dual graph~$\G$ is a single vertex, and that its BY tree~$T$ consists of a single vertex, which lies in~$S$. The induced action of Frobenius is the identity. Since all our results are essentially trivial to prove in the genus~$0$ case (the Tamagawa number of~$X/K$ is~$1$, since $\Jac(X)=0$), we will often implicitly assume that~$X/K$ has genus~$\geq1$ in what follows.
\end{remark}

\subsubsection{Dependence of Tamagawa numbers on signs}

To conclude this section, let us prove a basic property of BY trees which will be used in the sequel.

%\begin{proposition}\label{prop:autos_fix_vertices}
%Let $F$ be an even automorphism of a tree $T$. Then $F$ fixes a vertex of $T$.
%\begin{proof}
%An automorphism of a tree either fixes a vertex or inverts an edge, but not both. Since $F$ preserves a bipartition of $T$, it does not invert an edge.
%\end{proof}
%\end{proposition}

\begin{lemma}\label{lem:product_of_signs}
Let $T$ be a BY tree and let $\epsilon F$, $\epsilon'F$ be two signed automorphisms, only differing in their sign functions. Suppose that for every $F$-orbit $C,F(C),F^2(C),\dots,F^{q-1}(C)$ in $\pi_0(T\setminus S)$ we have
\[
\prod_{r=0}^{q-1}\epsilon(F^r(C)) = \prod_{r=0}^{q-1}\epsilon(F^r(C))\,.
\]
Then $\epsilon F$ and $\epsilon'F$ are conjugate in $\Aut^{\pm}(T)$. In particular, we have $c_{T,\epsilon F}=c_{T,\epsilon'F}$.
\begin{proof}
It suffices to prove the result in the special case that
\[
\epsilon'(C)=\begin{cases}-\epsilon(C)&\text{if $C=C_0$ or $F^{-1}C_0$,}\\\epsilon(C)&\text{otherwise,}\end{cases}
\]
for some non-$F$-fixed component $C_0$ of $T\setminus S$. If we set
\[
\epsilon''(C)=\begin{cases}-\epsilon(C)&\text{if $C=C_0$,}\\\epsilon(C)&\text{otherwise,}\end{cases}
\]
then it follows that $\epsilon''\id_T\circ\epsilon F = \epsilon'F\circ\epsilon''\id_T$. Thus $\epsilon F$ and $\epsilon'F$ are conjugate. This implies that they fix the same number of points of the Jacobian $\gNeron_T$, so that $c_{T,\epsilon F}=c_{T,\epsilon'F}$ as claimed.
\end{proof}
\end{lemma}
\section{The Tamagawa number formula}

With all the preliminaries above set up, we now come to the main technical result of this paper: an explicit formula describing the Tamagawa number of a BY tree with signed isomorphism in terms of purely graph-theoretic invariants. The full statement is as follows (the proof will be given in \S\ref{ss:proof_of_formula} and some illustrative examples in \S\ref{ss:examples}).

\begin{construction}\label{cons:negative_part}
Let $(T,S)$ be a BY tree and let $\epsilon F$ be an even signed automorphism. If $v$ is a vertex lying in $T\setminus S$, then we write $q_v$ for the size of the $F$-orbit containing $v$. We write
\[
\epsilon_v := \prod_{j=0}^{q_v-1}\epsilon(F^jC)
\]
where $C$ is the component of $T\setminus S$ containing $v$. If $e$ is an edge lying in $T\setminus S$, then we define $q_e$ and $\epsilon_e$ similarly. We write $\hat S\subseteq T$ for the subgraph\footnote{$\hat S$ is indeed a subgraph, for if $v$ is the endpoint of an edge $e$ in $\hat S$, then either $v\in S$, or $\epsilon_e$ is a power of $\epsilon_v$. In the latter case, we have $\epsilon_e=-1$, hence $\epsilon_v=-1$ and so $v\in\hat S$.} consisting of $S$ together with all vertices $v$ and edges $e$ with $\epsilon_v=-1$ and $\epsilon_e=-1$, respectively.
\end{construction}

\begin{theorem}\label{thm:formula}
Let $T=(T,S)$ be an even BY tree and let $\epsilon F$ be an even signed automorphism of $T$. Write $T'\supseteq\hat S'\supseteq S'$ for the quotients of $T\supseteq\hat S\supseteq S$ by the action of $F$, respectively. Then the Tamagawa number of $T$ is given by
\[
c_T := Q\cdot\tilde c\cdot\sum_{\{e'_1,\dots,e'_r\}\in R}\,\prod_{j=1}^r\frac1{q_{e'_j}} \,,
\]
where:
\begin{itemize}
	\item $Q$ is the product of the sizes of the $F$-orbits of connected components of $\hat S$;
	\item $\tilde c:=\prod_{C\in\pi_0(\hat S'\setminus S')}\tilde c(C)$ is a product of terms $\tilde c(C)$ over connected components $C$ of $\hat S'\setminus S'$, where:
	\begin{itemize}
		\item $\tilde c(C)=2^{a-1}$ if the closure of $C$ contains $a>0$ points of $S'$ lying an even distance from a vertex of degree $\geq3$;
		\item $\tilde c(C)=\gcd(l,2)$ if the closure of $C$ consists of two points of $S'$ a distance $l$ apart;
		\item $\tilde c(C)=\gcd(b,2)$ otherwise, where $b$ is the number of points of $S'$ in the closure of $C$;
	\end{itemize}
	\item $r=\#\pi_0(\hat S')-1$ is the number of connected components of $\hat S'$, minus $1$; and
	\item $R$ is the set of unordered $r$-tuples of edges in $T'\setminus\hat S'$ whose removal disconnects the $r+1$ components of $\hat S'$ from one another (meaning that there is no path in $T'\setminus\{e'_1,\dots,e'_r\}$ whose endpoints lie in different components of~$\hat S'$).
\end{itemize}
Here, by mild abuse of notation, we write $q_{e'}$ for $q_e$ where $e$ is any lift of $e'$ to an edge of $T$ (so $q_{e'}$ is the size of the $F$-orbit of edges corresponding to $e'$).
\end{theorem}

\begin{metricvariant*}%\label{rmk:metric_formula}
There is also a variant of theorem~\ref{thm:formula} which works also for even integral-metric BY trees $T$ with respect to even signed automorphisms (self-isometries) $\epsilon F$, and this version is often more useful for applications (see e.g.\ the proof of corollary~\ref{cor:qualitative}). The metric version of theorem~\ref{thm:formula} says that
\[
c_T := Q\cdot\tilde c\cdot\sum_{\{e'_1,\dots,e'_r\}\in R}\,\prod_{j=1}^r\frac{l(e')}{q_{e'_j}} \,,
\]
where $l(e')=l(e)$ for any edge $e$ of $T$ lying over $e'$. Here, $Q$, $\tilde c$, $r$, $T$ and $q_{e'}$ are as defined in theorem~\ref{thm:formula} above, where $T'$ is viewed as a integral-metric graph with edge-length function $l$ (for the purpose of interpreting ``distance'' in the definition of $\tilde c$).

This metric version is easy to deduce from theorem~\ref{thm:formula} directly (by computing the Tamagawa number of the non-metric BY tree isometric to $T$), or one can easily adapt the proof we give in~\S\ref{ss:proof_of_formula}.
\end{metricvariant*}

\subsection{Consequences of the formula}\label{ss:consequences}

For the purposes of the main results of this paper, the most important consequence of theorem~\ref{thm:formula} is that it affords us a qualitative description of how Tamagawa numbers of BY trees are changed under subdivision of edges. To explain this, let us fix a BY tree $T_0=(T_0,S_0)$ and a signed automorphism $\epsilon F$ of $T_0$. We let $T_0':=T_0/F$ and $S_0'/F$ denote the quotients of $T_0$ and $S_0$ by the action of $F$, as usual, so that the vertices and edges of $T_0'$ correspond to $F$-orbits of vertices and edges in $T_0$, respectively. Given an element $l\in\N^{E(T_0')}$, thought of as an $F$-invariant function $l\colon E(T_0)\rightarrow\N$, we let $T_0^{(l)}$ denote the subdivision of $T_0$ where we replace each edge $e$ of $T_0$ with a chain of $l(e)$ edges. The signed automorphism $\epsilon F$ of $T_0$ induces a corresponding signed automorphism of each $T_0^{(l)}$.

Theorem~\ref{thm:formula} allows us to give a qualitative description of the Tamagawa numbers of each $T_0^{(l)}$, as a function of $l$. This description plays a central role in both of the main results of this paper: it implies the dependence on $e$ in Theorem~\ref{thm:growth_in_towers}; as well as the second part of Theorem~\ref{thm:degeneration}.

\begin{corollary}[to theorem~\ref{thm:formula}]\label{cor:qualitative}
Keeping notation as above, there exists a homogenous polynomial map $P\colon\Q^{E(T_0')\setminus E(S_0')}\rightarrow\Q$ and a function $s\colon \F_2^{E(T_0')\setminus E(S_0')}\rightarrow\Z$ such that the Tamagawa number of $T_0^{(l)}$ is given by
\[
c_{T_0^{(l)},\epsilon F} = 2^{s(l)}\cdot P(l)
\]
for all $l\in\N^{E(T_0')}$.
\begin{proof}[Proof (even case)]
We prove the result under the additional assumption that $T_0$ and $\epsilon F$ are even (see definition~\ref{def:parity}), and only for those values of $l$ such that $T_0^{(l)}$ is even\footnote{This implies that $\epsilon F$ is also an even automorphism of $T_0^{(l)}$ by remark~\ref{rmk:even_autos}}. Since these hypotheses are automatically satisfied for BY trees of semistable hyperelliptic curves by proposition~\ref{prop:hyperelliptic-BY_equivalence}, this already suffices for our number-theoretic applications. For a proof in the general case, see~\S\ref{ss:qualitative_proof}.

Recall from remark~\ref{rmk:metric_vs_subdivisions} that $T_0^{(l)}$ is isometric to the integral-metric BY tree $(T_0,l)$, and hence they have the same Tamagawa number. We compute the Tamagawa number of the latter using theorem~\ref{thm:formula} (metric version), finding that
\[
c_{(T_0,l)} = Q\cdot\tilde c(l)\cdot \sum_{\{e'_1,\dots,e'_r\}\in R}\prod_{j=1}^r\frac{l(e')}{q_{e'_j}} \,.
\]
Of the values appearing in this formula, $Q$, $r$, $R$ and the $q_{e'}$ are independent of $l$ (they are the same as the values in theorem~\ref{thm:formula} for $T=T_0$), while $\tilde c(l)$ is always a power of $2$, depending only on the parity of the values of $l$ on edges outside $S_0'$. This implies the result, with
\[
P(l) := Q\cdot\sum_{\{e'_1,\dots,e'_r\}\in R}\prod_{j=1}^r\frac{l(e'_j)}{q_{e'_j}} \text{  and  }s(l) := \log_2(\tilde c(l)) \,.
\]
\end{proof}
\end{corollary}

\subsection{Proof of theorem~\ref{thm:degeneration}}\label{ss:degeneration_proof}

Using corollary~\ref{cor:qualitative}, we can now prove theorem~\ref{thm:degeneration} on the behaviour of Tamagawa numbers of semistable hyperelliptic curves, using the ``cluster picture'' machinery of \cite{m2d2,semistable_types}. We recall the setup from \cite[\S1]{m2d2}. Suppose that $K$ is a finite extension of $\Q_p$ ($p\neq2$) and $X_f/K$ is a hyperelliptic curve of genus $\geq2$ given by an equation $y^2=f(x)$ ($f$ of degree $\geq5$ and squarefree). We let $\cR\subseteq\overline K$ denote the set of roots of $f$, and define a \emph{proper cluster} of $f$ \cite[Definition~1.1]{m2d2} to be a subset $\s\subseteq\cR$ of size $\geq2$ which is cut out by a $p$-adic disc, i.e.\ such that there is some $z\in\overline K$ and $d\in\Q$ such that
\begin{equation}\label{eq:clusters}
\s = \{r\in\cR \:|\: v_K(r-z)\geq d\}\,,
\end{equation}
where $v_K$ denotes the valuation on $\overline K$, normalised so that the valuation of a uniformiser of $K$ is $1$. The set of proper clusters of $f$ is partially ordered by containment, and carries a natural action of the absolute Galois group $G_K$ induced from the action on $\cR$.

One can attach to a proper cluster $\s$ a number of extra data, including:
\begin{itemize}
	\item the \emph{size} $\#\s$ of $\s$;
	\item the \emph{depth} $d_\s$ of $\s$, i.e.\ the largest value of $d$ in \eqref{eq:clusters}, or equivalently the smallest value of $v_K(r-r')$ for $r,r'\in\s$ \cite[Definition~1.1]{m2d2};
	\item for each element $\sigma\in G_K$ of the absolute Galois group of $K$, a sign $\epsilon_\s(\sigma)\in\{\pm1\}$ \cite[Definition~1.12]{m2d2}\footnote{Strictly speaking, \cite[Definition~1.12]{m2d2} only defines the signs $\epsilon_\s(\sigma)$ for certain $\s$. For our purposes, it does no harm to attach signs to all proper clusters in exactly the same way, even though some of these signs have no number-theoretic content.}. (These signs depend on certain auxiliary choices of square roots.)
\end{itemize}

These cluster data (proper clusters, containment, Galois action, sizes, depths and signs) determine a wealth of arithmetic information about the hyperelliptic curve $X_f$ with affine equation $y^2=f(x)$, including whether $X_f$ is semistable \cite[Theorem~1.8(1)]{m2d2} (this also involves the leading coefficient and splitting field of~$f$), and its BY tree \cite[Theorem~5.18]{m2d2} with signed Galois action \cite[Theorem~6.9]{m2d2}.

\smallskip

In proving theorem~\ref{thm:degeneration}, it will be convenient to also consider cluster data associated to the polynomial $f_0$. Let $\cR_0$ denote the multiset\footnote{We won't concern ourselves with exactly what is meant by a ``multiset''. For our purposes it suffices that $\cR_0$ is a set with a map $\cR_0\rightarrow\overline K$ such that the size of the fibre over some $r\in\overline K$ is the multiplicity of $r$ as a root of $f_0$.} of roots of $f_0$. As above, we define a \emph{proper cluster} of $f_0$ to be a sub-multiset $\s_0\subseteq\cR_0$ of size $\geq2$ which is cut out by a $p$-adic disc. One can define sizes, depths, Galois action and signs on proper clusters of $f_0$ just as for $f$, with the caveat that certain proper clusters may have infinite depth. Specifically, any double root $\alpha_i$ of $f_0$ corresponds to a proper cluster $\{\alpha_i,\alpha_i\}$ of size~$2$ whose depth is $\infty$. We refer to such clusters as \emph{degenerate clusters} of $f_0$.

The main technical result needed in the proof of theorem~\ref{thm:degeneration} (corollary~\ref{cor:close_polys_have_similar_clusters}) asserts that when $f$ is sufficiently close to $f_0$, then the cluster pictures of $f$ and $f_0$ are isomorphic, up to the depths of clusters. This is a consequence of the following well-known proposition, which says that $p$-adically close polynomials have $p$-adically close roots.

\begin{proposition}\label{prop:close_polys_have_close_roots}
Let $f_0\in\overline K[x]$ be a non-zero polynomial of degree $d$, and let $\varepsilon>0$. Then there is a $\delta$ (depending on $f_0$ and $\varepsilon$) such that, for all $f\in\overline K[x]$ of degree $d$ with $|f-f_0|<\delta$, $f$ and $f_0$ have the same number of roots in every closed disc $D\subseteq\overline K$ of radius $\varepsilon$, counting multiplicity. Here, $|\cdot|$ denotes the Gauss norm on $\overline K[x]$, i.e.\ $\left|\sum_{i\geq0}a_ix^i\right|=\max\{|a_i|\}$ for $|\cdot|=p^{-v_K(\cdot)}$ the norm on~$\overline K$.
\begin{proof}
Performing an appropriate change of variables, it suffices to show that if $f_0$ and $f$ are polynomials with $|f-f_0|<|f_0|$, then $f$ and $f_0$ have the same number of roots in the closed unit disc, which we denote~$n$ and~$n_0$, respectively. Indeed, if we write $f=\sum_{i\geq0}a_ix^i$, then the theory of Newton polygons implies that $n$ is equal to the greatest index $i$ such that $|a_i|=|f|$, and similarly for $f_0$. The assumption that $|f-f_0|<|f_0|$ ensures that $|f|=|f_0|$, that $|a_{n_0}|=|f_0|$, and $|a_i|<|f_0|$ for $i>n_0$. This implies that $n=n_0$, as claimed.
\end{proof}
\end{proposition}

\begin{corollary}\label{cor:close_polys_have_similar_clusters}
Let notation be as in theorem~\ref{thm:degeneration}. Then there is a $\delta>0$ such that whenever $f\in K[x]$ is squarefree of degree $d$ and $|f-f_0|<\delta$, there is a bijection
\[
\psi\colon\{\text{proper clusters of $f_0$}\} \leftrightarrow \{\text{proper clusters of $f$}\}
\]
preserving containment, Galois action, sizes and signs, as well as depths of non-degenerate clusters. The depth of a cluster of $f$ corresponding to a degenerate cluster $\{\alpha_i,\alpha_i\}$ of $f_0$ is $d_i(f)$.
\begin{proof}
Fix a positive $\varepsilon$ less than the distance between any two distinct roots of $f_0$, and choose a positive $\delta$ as in proposition~\ref{prop:close_polys_have_close_roots}. We assume moreover that $\delta<|c_{f_0}|$, where $c_{f_0}$ is the leading coefficient of $f_0$. We will show the corollary holds for this value of $\delta$.

Fix a squarefree $f$ of the same degree as~$f_0$, with $|f-f_0|<\delta$. It follows from proposition~\ref{prop:close_polys_have_close_roots} that there is a bijection $\psi\colon\cR_0\xrightarrow\sim\cR$ such that $|\psi(r_0)-r_0|\leq\varepsilon$ for all $r_0\in\cR_0$. We claim that $\psi$ induces the desired bijection on proper clusters, proceeding in several steps.

(0) \textit{$\psi$ preserves sizes and containment of subsets.} (Obvious.)

%The first part is obvious. The second follows by definition, since $\psi(\{\alpha_i,\alpha_i\})=\{\beta_i,\gamma_i\}$, where $\beta_i$ and $\gamma_i$ are the roots of $f$ in the disc of radius $p^{-\epsilon}$ about $\alpha_i$.

(1) \textit{If $\s_0$ is a proper cluster of $f_0$, then $\psi(\s_0)$ is a proper cluster of~$f$.}

There is a closed disc $D_0\subset\overline K$ such that $\cR_0\cap D_0=\s_0$. Since the distance between distinct roots of $f_0$ is $>\varepsilon$, we may assume without loss of generality that the radius of $D_0$ is at least $\varepsilon$. This ensures that $\psi(\s_0)=\cR\cap D_0$, which is a proper cluster of~$f$.

(2) \textit{If $\s$ is a proper cluster of $f_0$, then $\psi^{-1}(\s)$ is a proper cluster of~$f_0$.}

Again, there is a closed disc $D\subseteq\overline K$ such that $\cR\cap D=\s$. It follows from proposition~\ref{prop:close_polys_have_close_roots} that every disc containing at least three roots of $f$ has radius $>\varepsilon$, and hence we may assume without loss of generality that the radius of $D$ is at least $\varepsilon$. This ensures that $\psi^{-1}(\s)=\cR_0\cap D$, which is a proper cluster of~$f_0$.

(3) \textit{$\psi$ preserves depths of non-degenerate clusters. If $\s_0=\{\alpha_i,\alpha_i\}$ is a degenerate cluster, then the depth of $\psi(\s_0)$ is equal to $d_i(f)$.}

The second part follows by definition of $d_i$. For the first part, suppose that $\s_0$ is a non-degenerate cluster. Its depth is then $<-\log_p(\varepsilon)$, and hence the disc $D_0$ in the proof of~(1) may be taken to have radius $p^{-d_{\s_0}}$. Since $\psi(\s_0)=\cR\cap D_0$, it follows that $d_{\psi(\s_0)}\geq d_{\s_0}$. For the reverse inequality, we pick elements $r_0,r_0'\in\s_0$ such that $v_K(r_0'-r_0)=d_{\s_0}<-\log_p(\varepsilon)$. The ultrametric triangle inequality implies that $v_K(\psi(r_0')-\psi(r_0))=v_K(r_0'-r_0)=d_{\s_0}$, and hence $d_{\psi(\s_0)}\leq d_{\s_0}$ by definition.

(4) \textit{The action of $\psi$ on proper clusters is Galois-equivariant.}

If $\s_0$ is a proper cluster of $f_0$, let $D_0\subset\overline K$ denote the smallest closed disc containing $\s_0$ whose radius is $\geq\varepsilon$. It follows from the proof of (1) that $\cR_0\cap D_0=\s_0$ and $\cR\cap D_0=\psi(\s_0)$. The assignment $\s_0\mapsto D_0$ is Galois-equivariant by construction, and the function $D_0\mapsto(\cR\cap D_0)$ is clearly Galois-equivariant.

(5) \textit{The action of $\psi$ preserves the signs $\epsilon_{\s_0}(\sigma)\in\{\pm1\}$ attached to elements $\sigma$ of the Galois group \cite[Definition~1.12]{m2d2}. More precisely, for any choice of elements $\theta_{\s_0}\in\overline K$ for proper clusters $\s_0$ of $f_0$ \cite[Definition~1.12]{m2d2}, there is a corresponding choice of elements $\theta_\s$ for proper clusters $\s$ of $f$, such that with respect to these choices we have
\[
\epsilon_{\psi(\s_0)}(\sigma) = \epsilon_{\s_0}(\sigma)
\]
for all proper clusters $\s_0$ of $f_0$ and all elements $\sigma\in G_K$.}

Fix a choice of centre $z_{\s_0}\in\overline K$ for each proper cluster $\s_0$ of $f_0$ \cite[Definition~1.9]{m2d2}. It follows from the above discussion that $z_{\s_0}$ is also a centre of $\psi(\s_0)$. Our assumptions ensure that $|c_f-c_{f_0}|<|c_{f_0}|$ (where $c_f$ is the leading coefficient of $f$) and $|r_0-\psi(r_0)|<|z_{\s_0}-r_0|$ for every $r_0\in\cR_0\setminus\s_0$, and hence the quantity
\[
\Delta_{\s_0} := \frac{c_f\cdot\prod_{r\in\cR\setminus\psi(\s_0)}(z_{\s_0}-r)}{c_{f_0}\cdot\prod_{r_0\in\cR_0\setminus\s_0}(z_{\s_0}-r_0)} \in 1+\mathfrak m_{\overline K}
\]
is a principal unit. In particular, $\Delta_{\s_0}$ has a canonical square root which is also a principal unit, and we define $\theta_{\psi(\s_0)}:=\sqrt{\Delta_{\s_0}}\cdot\theta_{\s_0}\in\overline K$.

Using these choices of $\theta_\s$ in the definition of the signs $\epsilon_\s(\sigma)$, we find that
\[
\epsilon_{\psi(\s_0)}(\sigma)\equiv\frac{\sigma(\theta_{\psi(\s_0)^*})}{\theta_{\sigma\psi(\s_0)^*}} = \frac{\sigma(\sqrt{\Delta_{\s_0^*}})}{\sqrt{\Delta_{\sigma(\s_0^*)}}}\cdot\frac{\sigma(\theta_{\s_0^*})}{\theta_{\sigma(\s_0)^*}} \equiv \epsilon_{\s_0}(\sigma) \,,
\]
where all congruences are modulo the maximal ideal $\mathfrak m_{\overline K}$.
\end{proof}
\end{corollary}

The ``cluster picture'' machinery of \cite{m2d2,semistable_types} allows us to translate corollary~\ref{cor:close_polys_have_similar_clusters} into a corresponding result on BY trees. In order to state this, let $\Sigma_0$ denote the set of all proper clusters of $f_0$, along with all singletons $\{r\}$ for $r\in\cR_0$. This is a \emph{cluster picture} in the sense of~\cite[Definition~D.1]{m2d2}\cite[Definition~3.33]{semistable_types}. The Galois action on proper clusters and the signs $\epsilon_{\s_0}(\cdot)$ determine an action of $G_K$ on $\Sigma_0$ by \emph{(signed) automorphisms} \cite[Definition~D.4]{m2d2}\cite[Definition~3.41]{semistable_types}. The depths of proper clusters determines a \emph{metric} on $\Sigma_0$ in the sense of \cite[Definition~3.45]{semistable_types}, except that the metric is valued in $\R_{>0}\cup\{\infty\}$ rather than $\R_{>0}$.

The construction in \cite[Definition~D.6]{m2d2} explains how to construct an integral-metric BY tree~$T_0=(T_0,S_0)$ with a signed Galois action from the cluster picture~$\Sigma_0$. Roughly speaking, the vertices of~$T_0$ are the proper clusters, each proper cluster except $\cR_0$ has an edge connecting it to the smallest cluster strictly containing it, the parities of the sizes of clusters determines the subgraph $S_0$, and the relative depths of clusters determine the edge-lengths. In particular, $T_0$ has one edge~$e_i$ of infinite length for each double-root $\alpha_i$ of $f_0$. The Galois action on~$T_0$ is the natural one induced by the action on clusters, and is upgraded to a signed action using the signs $\epsilon_{\s_0}(\sigma)$.

\begin{corollary}\label{cor:close_polys_have_similar_BY_trees}
Let notation be as in theorem~\ref{thm:degeneration}, and let $T_0$ be the integral-metric BY tree produced from~$f_0$ as above.

Then there is a $\delta>0$ with the following property. Whenever $f\in K[x]$ is squarefree of degree $d$ with $|f-f_0|<\delta$ and $X_f$ semistable, the BY tree associated to $X_f$ (see \S\ref{sss:BY_tree_of_curve}) is $G_K$-equivariantly isomorphic to the subdivision of $T_0$ formed by replacing each edge~$e$ of finite length with a chain of $l(e)$ edges, and replacing the edge $e_i$ corresponding to a double root $\alpha_i$ with a chain of $2\left(d_i(f)-a_i\right)$ edges, where $a_i$ is the depth of the smallest cluster of $f_0$ strictly containing $\{\alpha_i,\alpha_i\}$.
\begin{proof}
We take the same value of $\delta$ as in corollary \ref{cor:close_polys_have_similar_clusters}. Let us write $\Sigma_f$ for the cluster picture of $f$; that is, the set of all proper clusters of~$f$, along with the singletons $\{r\}$ for $r\in\cR$. Let $T_f$ denote the integral-metric BY tree produced from $\Sigma_f$ via the construction in \cite[Definition~D.4]{m2d2}, which comes with a signed $G_K$-action induced from the signed $G_K$-action on $\Sigma_f$. It follows from \cite[Theorems~5.18~\&~6.9]{m2d2} that $T_f$ is $G_K$-equivariantly isometric to the BY tree of $X_f$ (cf.\ \cite[Definition~D.9]{m2d2} and proposition~\ref{cons:hyperelliptic_graph_to_BY_tree}). Yet it follows from the construction and corollary~\ref{cor:close_polys_have_similar_clusters} that $T_f$ is $G_K$-equivariantly isometric to the claimed subdivision of $T_0$.
\end{proof}
\end{corollary}

\begin{proof}[Proof of theorem~\ref{thm:degeneration}]
We take the same value of $\delta$ as in corollary~\ref{cor:close_polys_have_similar_clusters}.

\eqref{thmpart:degeneration_semistability}: We use the semistability criterion \cite[Theorem~1.8(1)]{m2d2}. Part~(1) of the criterion is automatically satisfied, since the action of the inertia group $I_K$ fixes all roots of $f_0$ by assumption, and hence acts on the roots of $f$ with order $\leq2$. Part~(2) of the criterion only depends on the inertia action on clusters of size $\geq3$, and hence is satisfied for $f$ if and only if it is satisfied for $f_0$. Likewise, part~(3) of the criterion only depends on the valuation of the leading coefficient of $f$ and the depths of proper clusters, and hence is satisfied for $f$ if and only if it is satisfied for $f_0$. In summary, we see that $X_f$ is semistable if and only if $f_0$ satisfies parts~(2) and~(3) of the semistability criterion. In particular, whether or not $X_f$ is semistable depends only on $f_0$, not $f$.

\eqref{thmpart:degeneration_formula}: This follows from corollary~\ref{cor:close_polys_have_similar_BY_trees} and corollary~\ref{cor:qualitative}.
\end{proof}

\subsection{Proof of the formula}
\label{ss:proof_of_formula}

We now prove theorem~\ref{thm:formula}.

\subsubsection{Reduction to simple BY trees}\label{ss:reduction}

Our first step in the proof is to reduce to the case when the BY tree in question is of a particularly simple form. Fix for this section a BY tree $(T,S)$ with an automorphism $\epsilon F$. Let $I$ denote the set of $F$-orbits in $\pi_0(T\setminus S)$, and write $q_i$ for the size of the $i$th orbit. For each $i$, we choose a representative $C_i$ in the $i$th $F$-orbit, and write $T_i$ for the closure of $C_i$ in $T$ and $S_i:=T_i\cap S$. The pair $T_i=(T_i,S_i)$ is a BY tree, and $(\epsilon F)^{q_i}$ restricts to a signed automorphism of $T_i$. The following basic facts are easy to see.

\begin{proposition}\label{prop:basic_reduction}
For each $i\in I$, the pair $T_i=(T_i,S_i)$ is a BY tree, and $(\epsilon F)^{q_i}$ restricts to a signed automorphism of $T_i$. The subgraph $S_i$ is a non-empty set of degree~$1$ vertices of $T_i$.

If $T$ is even, then so is each $T_i$; if $\epsilon F$ is even then so is the restriction of $(\epsilon F)^{q_i}$ to $T_i$ for each $i$.
\end{proposition}

Our strategy is to reduce the proof of theorem~\ref{thm:formula} for $T$ to the corresponding results for each $T_i$. This is done by showing that the Tamagawa number of $c_T$ factorises as the product of the Tamagawa numbers of the $T_i$, as follows.

\begin{lemma}\label{lem:reduction_LHS}
Keeping notation as above, there is an $\epsilon F$-equivariant isomorphism
\[
\gNeron_T\simeq\bigoplus_{i\in I}\Ind_{(\epsilon F)^{q_i}}^{\epsilon F}\gNeron_{T_i} \,,
\]
where $\Ind_{(\epsilon F)^{q_i}}^{\epsilon F}\gNeron_{T_i} := \bigoplus_{j=0}^{q_i-1}(\epsilon F)^j\gNeron_{T_i}$ with the evident action of $\epsilon F$ induced from the action of $(\epsilon F)^{q_i}$ on $\gNeron_{T_i}$. In particular, taking $\epsilon F$-invariants of either side, we have
\[
c_{T,\epsilon F} = \prod_{i\in I}c_{T_i,(\epsilon F)^{q_i}} \,.
\]
\begin{proof}
By Mayer--Vietoris and excision, the inclusions $(T_i,S_i)\hookrightarrow(T,S)$ induce an isomorphism
\[
\H_1(T,S,\Z) \cong \bigoplus_{i\in I}\bigoplus_{j=0}^{q_i-1}\H_1(F^jT_i,F^jS_i,\Z)
\]
in relative homology, which is an orthogonal direct sum with respect to the intersection-length pairings on the groups involved. It follows that the map $\H_1(T,S,\Z)\hookrightarrow\H^1(T,S,\Z)$ induced by the intersection-length pairing is the direct sum of the corresponding maps $\H_1(F^jT_i,F^jS_i,\Z)\hookrightarrow\H^1(F^jT_i,F^jS_i,\Z)$, and hence we have an isomorphism $\gNeron_T\cong\bigoplus_{i\in I}\bigoplus_{j=0}^{q_i-1}\gNeron_{F^jT_i}$. Under the obvious identification $\Ind_{(\epsilon F)^{q_i}}^{\epsilon F}\gNeron_{T_i}\cong\bigoplus_{j=0}^{q_i-1}\gNeron_{F^jT_i}$, this yields the desired isomorphism, which is easily seen to be $\epsilon F$-equivariant.
\end{proof}
\end{lemma}

To complete the reduction step, we need to prove a corresponding factorisation result for the right-hand side in theorem~\ref{thm:formula}.

\begin{proposition}\label{prop:reduction_RHS}
Let $T=(T,S)$ be a BY tree and let $\epsilon F$ be an even signed automorphism. Define
\begin{align*}
c'_T &:= Q\cdot\tilde c\cdot\sum_{\{e'_1,\dots,e'_r\}\in R}\,\prod_{j=1}^r\frac1{q_{e'_j}} \,, \\
c'_{T_i} &:= Q_i\cdot\tilde c_i\cdot\sum_{\{e'_1,\dots,e'_r\}\in R_i}\,\prod_{j=1}^{r_i}\frac1{q_{i,e'_j}} \,,
\end{align*}
where $Q$, $\tilde c$, $r$, $R$ and $q_{e'}$ denote the quantities defined in theorem~\ref{thm:formula} for $(T,\epsilon F)$, and $Q_i$, $\tilde c_i$, $r_i$, $R_i$ and $q_{i,e'}$ denote the corresponding quantities for $(T_i,(\epsilon F)^{q_i})$. Then we have
\[
c_T' = \prod_ic_{T_i}' \,.
\]
\begin{proof}
Let $T'\supseteq\hat S'\supseteq S'$ denote the graphs defined in construction~\ref{cons:negative_part} for $T$, and $T_i'\supseteq\hat S_i'\supseteq S_i'$ the corresponding graphs for $T_i$. The connected components of $T'\setminus S'$ are in canonical bijection with the $F$-orbits of connected components of $T\setminus S$, and the closures of these components are canonically identified with the trees $T_i'$. Under this identification, we have $S_i'=T_i'\cap S'$ and $\hat S_i'=T_i'\cap\hat S'$.

We will prove the following five assertions, which together imply the result:
\begin{enumerate}
	\item\label{pfpart:Q} $Q=\prod_i\left(Q_i\cdot q_i^{r_i}\right)$;
	\item\label{pfpart:c} $\tilde c=\prod_i\tilde c_i$;
	\item\label{pfpart:r} $r=\sum_ir_i$;
	\item\label{pfpart:R} the map $\prod_iR_i\rightarrow R$ sending a choice of an $r_i$-tuple of edges in $T_i'\setminus S_i'$ for each $i$ to their union is bijective; and
	\item\label{pfpart:q} if $e'$ is an edge in $T_i'\setminus S_i'$, then $q_{e'}=q_i\cdot q_{i,e'}$.
\end{enumerate}

Of these,~\eqref{pfpart:c} and~\eqref{pfpart:q} are easy to see, and~\eqref{pfpart:R} follows once we know~\eqref{pfpart:r}. For instance, for the proof of~\eqref{pfpart:R}, we note that removing an $r$-tuple $\{e'_1,\dots,e'_r\}$ of edges in $T'\setminus\hat S'$ disconnects the components of $\hat S'$ from one another if and only if, for each $i$, removing those edges which lie in $T_i'$ disconnects the components of $\hat S_i'$ from one another. Since $r=\sum_ir_i$ by~\eqref{pfpart:r}, this is only possible if exactly $r_i$ of the edges $e_1,\dots,e_r$ lie in $T_i'$ for each $i$, and hence the map $\prod_iR_i\rightarrow R$ is bijective.

\smallskip

For the remaining two parts, we fix an $F$-fixed vertex $v_0$ of $T$, and for each $i$ let $v_{i,0}$ denote the vertex of $T_i$ closest to $v_0$.

\textit{Claim:} Let $C$ be a component of $\hat S$ not containing $v_0$. Then there is a unique index $i$ and a unique $0\leq j<q_i$ such that $C$ meets $F^jT_i$ but doesn't contain $F^jv_{i,0}$.

\textit{Proof of claim:} Let $v$ be the vertex of $C$ closest to $v_0$, and let $e$ denote the first edge on the shortest path from $v$ to $v_0$. The edge $e$ is not contained in $\hat S$, and so lies in $F^jT_i$ for some $i$ and some $0\leq j<q_i$. It is easy to check that $C$ meets $F^jT_i$ but doesn't contain $F^jv_{i,0}$, and that $i$ and $j$ are unique with this property.

\smallskip

Using this claim, we find that there is an $F$-equivariant bijection
\begin{equation}\label{eq:counting_argument}\tag{$\dag$}
\pi_0(\hat S)^\circ\cong\coprod_i\coprod_{j=0}^{q_i-1}\pi_0(F^j\hat S_i)^\circ \\,
\end{equation}
where $\pi_0(\hat S)^\circ$ and $\pi_0(\hat S_i)^\circ$ denote the set of all components of $\hat S$ and $\hat S_i$ which do not contain $v_0$ and $v_{i,0}$, respectively. The bijection~\eqref{eq:counting_argument} sends a component $C\in\pi_0(\hat S)^\circ$ to $C\cap F^jT_i\in\pi_0(F^j\hat S_i)^\circ$, where $i$ and $j$ are as in the claim.

Now there are two cases to consider. If $v_0\in\hat S$, then we have $\pi_0(\hat S)^\circ=\pi_0(\hat S)\setminus\{*\}$ and $\pi_0(\hat S_i)^\circ=\pi_0(\hat S_i)\setminus\{*\}$ for all $i$. Thus, taking the number of $F$-orbits on either side of~\eqref{eq:counting_argument} shows that $r=\sum_ir_i$, and taking the product of the sizes of the $F$-orbits gives $Q=\prod_i\left(Q_i\cdot q_i^{r_i}\right)$, proving~\eqref{pfpart:r} and~\eqref{pfpart:Q} in this case.

If instead $v_0\notin\hat S$, then there is a unique index $i_0$ such that $v_0\in T_i$, and we have $q_{i_0}=1$. In this case we have $\pi_0(\hat S)^\circ=\pi_0(\hat S)$ and $\pi_0(\hat S_{i_0})^\circ=\pi_0(\hat S_{i_0})$, and $\pi_0(\hat S_i)^\circ=\pi_0(\hat S_i)\setminus\{*\}$ for all other $i$. Again, taking the number of $F$-orbits and the product of the sizes of the $F$-orbits on either side of~\eqref{eq:counting_argument} gives~\eqref{pfpart:r} and~\eqref{pfpart:Q}, which completes the proof.
\end{proof}
\end{proposition}

\begin{corollary}\label{cor:reduce_to_simple}
Let us say that a BY tree $T=(T,S)$ is \emph{simple} just when $S$ is a non-empty set of degree~$1$ vertices of $T$ (with no edges). If theorem \ref{thm:formula} is true for simple even BY trees and even automorphisms, then it is true in general (for even BY trees and even automorphisms).
\begin{proof}
This follows from lemma~\ref{lem:reduction_LHS} and proposition~\ref{prop:reduction_RHS}, recalling that if $T$ and $\epsilon F$ are even, then so too are each $T_i$ and $(\epsilon F)^{q_i}$, and each $T_i$ is simple (proposition~\ref{prop:basic_reduction}).
\end{proof}
\end{corollary}

\subsubsection{Simple BY trees: positive case}\label{ss:positive}

Having reduced the proof of theorem~\ref{thm:formula} to the case of simple even BY trees and even automorphisms (see corollary~\ref{cor:reduce_to_simple}), we now proceed to prove it in the case when $T$ is simple and the signed automorphism $\epsilon F=+F$ is even with sign $\epsilon=+1$ everywhere. For technical reasons, we do not assume that $T$ is even (see remark~\ref{rmk:sometimes_not_even}).

Our proof of theorem~\ref{thm:formula} for this $T$ will be largely cohomological, for which we adopt the following notation.

\begin{notation}
If $M$ is a $\Z$-module on which an automorphism $\sigma$ (usually $\pm F$) acts with finite order, we will denote by $\H^j(\sigma,M)$ the continuous Galois cohomology of the continuous action of the profinite cyclic group $\hat\Z$ on the discrete group $M$ where a generator acts by $\sigma$.% Since every continuous map $\hat\Z^j\to M$ factors through a map $(\hat\Z/d\hat\Z)^j\to M$ for some $d\in\N$, the natural map $\varinjlim\H^j(\hat\Z/d\hat\Z,M^{d\hat\Z})\isoarrow\H^j(\hat\Z,M)$ is an isomorphism, so $\H^j(\sigma,M)$ is the direct limit of the cohomology of cyclic groups.
\end{notation}

\begin{remark}
One could prove all the main results of this paper using cohomology of finite cyclic groups in place of the continuous cohomology of the profinite cyclic group $\hat\Z$. However, using continuous cohomology of $\hat\Z$ allows us to avoid a considerable amount of bookkeeping, in that we don't have to keep track of exactly which cyclic group was acting on each object. Indeed, cyclic group cohomology in general depends on which cyclic group is acting: if $M$ is a $\Z$-module with an automorphism $\sigma$ of finite order $d$, then $M$ can be thought of as carrying an action of the cyclic group $\Z/rd\Z$ for any $r\in\N$, but the cohomology $\H^1(\Z/rd\Z,M)$ does depend on $r$ in general. For example, if the action of $\sigma$ is trivial ($d=1$), then $\H^1(\Z/r\Z,M)=M[r]$ is the $r$-torsion in $M$. Thus, continuous cohomology of $\hat\Z$ is the ``natural'' cohomology theory on the category of $\Z$-modules with finite-order automorphisms.
\end{remark}

Now fix a simple BY tree $T=(T,S)$, and let $F=+F$ be an even automorphism of $T$. We do not assume that $T$ is even. Note that, in the notation of construction~\ref{cons:negative_part}, we have $\hat S=S$, and in particular the term $\tilde c$ in theorem~\ref{thm:formula} is always~$1$ (the empty product).

Let $\Lambda=\H_1(T,S,\Z)$ denote the relative homology lattice, so that we have an exact sequence
\[
0\rightarrow\Lambda\rightarrow\Lambda^\vee\rightarrow\gNeron\rightarrow0
\]
with $\gNeron$ the Jacobian of $T$. Taking cohomology gives an exact sequence
\[
0\rightarrow\Lambda^F\rightarrow(\Lambda^\vee)^F\rightarrow\gNeron^F\rightarrow\H^1(+F,\Lambda)\rightarrow\H^1(+F,\Lambda^\vee)\,,
\]
so that
\begin{equation}\label{eq:positive_c}
c_T = \#\coker\left(\Lambda^F\hookrightarrow(\Lambda^\vee)^F\right)\cdot\#\ker\left(\H^1(+F,\Lambda)\rightarrow\H^1(+F,\Lambda^\vee)\right)\,.
\end{equation}
We calculate these two terms separately, in the following two propositions.

\begin{proposition}\label{prop:positive_ker}
Keep notation as above, and let $m$ be the greatest common divisor of the sizes of the $F$-orbits in $S$. Then $\H^1(F,\Lambda)$ is cyclic of order $m$, and the map\[\H^1(F,\Lambda)\rightarrow\H^1(F,\Lambda^\vee)\]induced by the intersection length pairing is the zero map. In particular,
\[
\#\ker\left(\H^1(+F,\Lambda)\rightarrow\H^1(+F,\Lambda^\vee)\right)=m\,.
\]
\begin{proof}
First, note that the exact sequence on homology of a pair gives an exact sequence
\begin{equation}\label{eq:LES_pair}
0\rightarrow\Lambda\rightarrow\Z[S]\rightarrow\Z\rightarrow0\,,
\end{equation}
where $\Z[S]$ is the free $\Z$-module on $S$ and the right-hand map is the sum-of-coordinates map. Taking $F$-fixed points, we obtain an exact sequence\[\Z[S]^F\rightarrow\Z\rightarrow\H^1(+F,\Lambda)\rightarrow\H^1(+F,\Z[S]),\]where the right-hand group vanishes by Shapiro's lemma. $\Z[S]^F$ is generated by the sums of elements in each $F$-orbit $\omega$, which maps to $\#\omega$ in $\Z$, and hence we see that $\H^1(+F,\Lambda)$ is cyclic of order $m$, generated by the cocycle associated to $(1-F)y$ for any choice of $y\in S$.

It remains to show that this maps to zero in $\H^1(+F,\Lambda^\vee)$. To do this, pick by any $F$-fixed vertex $y_0$ of $T$, and let $\alpha\in\Lambda^\vee=\Hom(\H_1(T,S,\Z),\Z)$ be the map given by length of intersection with the path from $y_0$ to $y$. Then the image of $(1-F)y$ in $\Lambda^\vee$ is given by intersection length with the path from $Fy$ to $y$, and hence is $(1-F)\alpha$. In other words, the cocycle associated to $(1-F)y$ maps to a coboundary in $\Lambda^\vee$, as desired.
\end{proof}
\end{proposition}

\begin{proposition}\label{prop:positive_coker}
Keep notation as above, and let $m$ and $Q$ be the greatest common divisor and product of the sizes of the $F$-orbits in $S$, respectively. Let $T'\supseteq S'$ be the quotient of $T\supseteq S$ by the action of $F$ as in theorem~\ref{thm:formula}, and endow $T'$ with the metric whereby an edge $e'$ of $T'$ has length $l(e')=1/q_{e'}$. Let $\Lambda'=\H_1(T',S',\Z)$ denote the relative homology lattice of $T'$ relative to $S'$, and write $\langle\cdot,\cdot\rangle'$ for the intersection length pairing on $\Lambda'$. Then\[\#\coker\left(\Lambda^F\hookrightarrow(\Lambda^\vee)^F\right)=\frac Qm\cdot\disc\left(\langle\cdot,\cdot\rangle'\right).\]
\begin{proof}
Let $\rho\colon\Lambda^F\rightarrow\Lambda'$ denote the pushforward map on relative homology induced by the quotient map $(T,S)\rightarrow(T',S')$, and let $\E\colon\Lambda'\rightarrow\Lambda^F\otimes\R$ be the map which takes (the class of) a path in $T'$ to the average of the paths in $T$ lying above it.

Now $\E$ and $\rho$ are adjoint. To see this, observe that $\rho$ and $\E$ naturally extend to all formal sums of oriented edges (not just those with zero boundary), so we need only check that $\langle e,\E e'\rangle=\langle \rho e,e'\rangle'$ for edges $\tilde e$ and $e$ of $T$ and $T'$, respectively. Since $e'$ has $q_{e'}$ preimages in $T$, we have $\langle e,\E e'\rangle=1/q_{e'}=l(e')=\langle \rho e,e'\rangle'$ if $e$ is one of these preimages, and is equal to $0$ otherwise, so that $\E$ and $\rho$ are adjoint as claimed.

It now follows from adjointness that we have a commuting square
\begin{center}
\begin{tikzcd}
\Lambda'\otimes\R \arrow{r}{\E}\arrow{d} & \Lambda^F\otimes\R \arrow{d} \\
(\Lambda')^\vee\otimes\R \arrow{r}{\rho^*} & (\Lambda^\vee)^F\otimes\R.
\end{tikzcd}
\end{center}
All the vector spaces involved are equidimensional and have specified full-rank sublattices, which determine volume forms on each vector space (up to sign), and hence we may talk about the absolute determinant of each of these maps. The leftmost vertical map has determinant $\disc\left(\langle\cdot,\cdot\rangle'\right)>0$ by definition, so by taking determinants both ways around the square we find that\[\#\coker\left(\Lambda^F\hookrightarrow(\Lambda^\vee)^F\right)=\frac{|\det\rho^*|}{|\det\E|}\cdot\disc\left(\langle\cdot,\cdot\rangle'\right).\]

To complete the proof, it suffices to compute $|\det\E|$ and $|\det\rho^*|$. To perform the computation of $|\det\rho^*|$, we note that since $\H^1(+F,\Z)=0$, the exact sequence (dual to sequence~\eqref{eq:LES_pair})\[0\rightarrow\Z\rightarrow\Z^S\rightarrow\Lambda^\vee\rightarrow0\]remains exact when we take $F$-fixed points, and so identifies $(\Lambda^\vee)^F$ as the lattice of $F$-invariant $\Z$-valued functions on $S$, modulo constants. Yet by the same reasoning $(\Lambda')^\vee$ is the lattice of $\Z$-valued functions on $S'$, modulo constants, so that $\rho^*\colon(\Lambda')^\vee\rightarrow(\Lambda^\vee)^F$ is an isomorphism of lattices. It follows that $|\det\rho^*|=1$.

To perform the computation of $|\det\E|$, note that the proof of proposition~\ref{prop:positive_ker} shows that the image of the sum-of-coordinates map $\Z[S]^F\rightarrow\Z$ had image $m\Z$. Thus we have a commuting diagram with exact rows
\begin{center}
\begin{tikzcd}
0 \arrow{r} & \Lambda'\otimes\R \arrow{r}\arrow{d}{\E} & \Z[S']\otimes\R \arrow{r}\arrow{d}{\E} & \Z\otimes\R \arrow{r}\arrow{d} & 0 \\
0 \arrow{r} & \Lambda^F\otimes\R \arrow{r} & \Z[S]^F\otimes\R \arrow{r} & m\Z\otimes\R \arrow{r} & 0.
\end{tikzcd}
\end{center}
Again, each of these vertical maps goes between equidimensional vector spaces with specified full-rank sublattices, so they have well-defined absolute determinants, and the absolute determinant of the central map is the product of those of the outer two maps.

Yet if we let $y_i$ denote a collection of representatives for the $F$-orbits on $S$ and $q_i$ their sizes, it follows that $\rho y_i$ is a basis for $\Z[S']$ and $(1+F+\dots+F^{q_i-1})y_i=q_i\E\rho y_i$ is a basis for $\Z[S]^F$. Hence the central map has absolute determinant $\prod_iq_i^{-1}=Q^{-1}$.

Moreover the rightmost map clearly has absolute determinant $m^{-1}$, so that the absolute determinant of the leftmost map is $|\det\E|=\frac mQ$. Combining this with the computed value of $|\det\rho^*|=1$ and the above formula, this yields the desired result.
\end{proof}
\end{proposition}

In light of the above two propositions and~\eqref{eq:positive_c}, to prove theorem~\ref{thm:formula} for $T$ it remains to characterise the discriminant $\disc(\langle\cdot,\cdot\rangle')$. This is an analogue of the matrix-tree theorem, and is as follows.

\begin{lemma}\label{lem:discriminant}
Let $T$ be a metric tree with a set $S$ of $r+1$ marked points, and let $\langle\cdot,\cdot\rangle\colon\Lambda\otimes\Lambda\rightarrow\R$ denote the intersection-length pairing on $\Lambda=\H_1(T,S,\Z)$. Then\[\disc\left(\langle\cdot,\cdot\rangle\right)=\sum_{e_1,\dots,e_r}\prod_{j=1}^rl(e_j),\]where the sum is taken over all unordered $r$-tuples of edges of $T$ whose removal disconnects the $r+1$ points of $S$ from one another.
\begin{proof}
Note that, for any basis of $\Lambda$, the entries of the pairing matrix with respect to this basis are homogenous linear forms in the edge lengths of $T$, so that $\disc\left(\langle\cdot,\cdot\rangle\right)$ is a degree $r$ homogenous form in the edge lengths. We will find its coefficients by setting the edge lengths of $T$ to suitably chosen values. In doing so, it will be convenient for us to permit ourselves to set certain edge-lengths to $0$, which may cause the intersection-length pairing to become indefinite.

Suppose first that $E$ is a set of edges of $T$ whose removal does not disconnect the points of $S$ from one another (this is certainly the case if $|E|<r$). Let us set the lengths of all edges not in $E$ to $0$, and let those in $E$ be arbitrary. By assumption, there is a path between two points of $S$ not meeting $E$, and this path pairs to $0$ with any other element of $\H_1(T,S,\Z)$. Hence the pairing on $\Lambda$ is degenerate, so its discriminant is $0$ independently of the lengths of the edges in $E$. It follows that $\disc\left(\langle\cdot,\cdot\rangle\right)$ does not contain any monomials only in edge lengths from $E$.

Thus we have shown that the only possible monomials that can appear in $\disc\left(\langle\cdot,\cdot\rangle\right)$ are products $l(e_1)\dots l(e_r)$ where $e_1,\dots,e_r$ are distinct edges whose removal disconnects the points of $S$. It remains to show that each of these monomials has coefficient $1$.

To do this, set the lengths of $e_1,\dots,e_r$ to $1$ and all other edge lengths to $0$. If we contract out all the edges of length $0$, this does not make any of the points of $S$ collide (by assumption), and moreover does not affect the pairing on homology, so we may assume for this that $T$ is a tree with $r$ edges, all of length $1$. But this means that $T$ only has $r+1$ vertices in total, so that $S$ consists of all vertices of $T$. We can then choose a basis of $\Lambda=\H_1(T,S,\Z)$ consisting of oriented edges, and with respect to this basis the intersection length pairing is represented by the identity matrix. It follows that the discriminant of the pairing is $1$, which is what we wanted to show.
\end{proof}
\end{lemma}

\begin{corollary}\label{cor:formula_positive}
Suppose that $T=(T,S)$ is a BY tree and $F=+F$ is an even automorphism, viewed as a signed automorphism with sign $+1$ everywhere (we do not assume that $T$ is even). Then the conclusion of theorem~\ref{thm:formula} holds for $T$.
\begin{proof}
When $T$ is simple, combine~\eqref{eq:positive_c} with propositions~\ref{prop:positive_ker} and~\ref{prop:positive_coker} and lemma~\ref{lem:discriminant}. The general case follows by the same argument as corollary~\ref{cor:reduce_to_simple}.
\end{proof}
\end{corollary}

\subsubsection{Simple BY trees: negative case}\label{ss:negative}

We now turn our attention to proving theorem~\ref{thm:formula} for simple even BY trees with a negative automorphism, where the group cohomology calculations are a little more complicated. In order to carry out these computations, we will use without comment the following calculation of the cohomology of permutation representations.

\begin{proposition}\label{prop:negative_cohomology}
Let $S$ be a finite set with an action by an automorphism $F$, and write $S=S_0\sqcup S_1$ where $S_0$ (resp.\ $S_1$) is the set of elements of $S$ in an even-sized (resp.\ odd-sized) $F$-orbit. Then the low-degree cohomology of the permutation representation $\Z[S]$ is given by
\begin{align*}
\Z[S]^{-F} &\simeq \Z[S_0'] \\
\H^1(-F,\Z[S]) &\cong \F_2[S_1']\,,
\end{align*}
where $S_i'=S_i/F$ denotes the set of $F$-orbits in $S_i$. The first isomorphism sends an orbit $\omega\in S_0'$ of size $2m$ to $s-Fs+F^2s-\dots-F^{2m-1}s$ where $s\in\omega$ is any orbit-representative. The second isomorphism sends an orbit $\omega\in S_1'$ of odd size to the cocycle sending the generator $1\in\hat\Z$ to some orbit-representative $s\in\omega$. A similar description holds for the dual representation $\Z^S$ (which is also a permutation representation).
\end{proposition}

Now fix a simple even BY tree $T=(T,S)$ with a negative even automorphism~$-F$. As in proposition~\ref{prop:negative_cohomology}, we write $S=S_0\sqcup S_1$ for the partition of $S$ into $F$-orbits of even and odd size, respectively. We also write $T_1\subseteq T$ for the largest subtree of $T$ on which $F$ acts with odd order. An observation which we will use regularly is that we have $\hat S=T_1\cup S=T_1\sqcup S_0$, where $\hat S$ is as defined in construction~\ref{cons:negative_part}.

Now, exactly as in~\eqref{eq:positive_c}, the Tamagawa number of $(T,-F)$ factorises as
\begin{equation}\label{eq:negative_c}
c_T=\#\coker\left(\Lambda^{-F}\hookrightarrow(\Lambda^\vee)^{-F}\right)\cdot\#\ker\left(\H^1(-F,\Lambda)\rightarrow\H^1(-F,\Lambda^\vee)\right)\,,
\end{equation}
where $\Lambda=\H_1(T,S,\Z)$ as usual. We calculate these two terms separately, in the following two propositions.

\begin{proposition}\label{prop:negative_ker}
Keep notation as above, and write $T_1'\supseteq S_1'$ for the quotient of $T_1\supseteq S_1$ by the action of $F$. Let $\Lambda_{1,\F_2}:=\H_1(T_1',S_1',\F_2)$ denote the mod $2$ relative homology of $(T_1',S_1')$, which carries a mod $2$ intersection length pairing $\langle\cdot,\cdot\rangle_{1,\F_2}\colon\Lambda_{1,\F_2}\otimes\Lambda_{1,\F_2}\rightarrow\F_2$. Then there are canonical isomorphisms
\[
\H^1(-F,\Lambda)\cong\Lambda_{1,\F_2}\text{  and  }\H^1(-F,\Lambda^\vee)\cong\Lambda_{1,\F_2}^\vee := \Hom(\Lambda_{1,\F_2},\F_2)
\]
for which the map $\H^1(-F,\Lambda)\rightarrow\H^1(-F,\Lambda^\vee)$ induced by the intersection length pairing is identified with the map $\Lambda_{1,\F_2}\rightarrow\Lambda_{1,\F_2}^\vee$ induced by the mod $2$ intersection length pairing.
\begin{proof}
Taking cohomology of the sequence~\eqref{eq:LES_pair} gives an exact sequence
\[
0\rightarrow\H^1(-F,\Lambda)\rightarrow\F_2[S_1']\rightarrow\F_2 \,.
\]
On the other hand, the exact sequence on the mod $2$ homology of the pair $(T_1',S_1')$ gives an exact sequence
\[
0\rightarrow\Lambda_{1,\F_2}\rightarrow\F_2[S_1']\rightarrow\F_2 \,.
\]
In both cases, the right-hand map is the sum-of-coordinates map, and hence there is a canonical identification $\Lambda_{1,\F_2}\cong\H^1(-F,\Lambda)$. Concretely, this isomorphism takes the class of a path $\gamma'$ in $T_1'$ to the cohomology class $\xi$ taking the generator $1\in\hat\Z$ to any lift $\gamma$ of $\gamma'$ to a path in $T_1$.

Similarly, taking the cohomology of the dual sequence to~\eqref{eq:LES_pair} gives an exact sequence
\[
\F_2\rightarrow\F_2^{S_1'}\rightarrow\H^1(-F,\Lambda^\vee)\rightarrow0\,
\]
By identifying this sequence with the exact sequence on cohomology of the pair $(T_1',S_1')$, we obtain a canonical identification $\H^1(-F,\Lambda^\vee)\cong\Lambda_{1,\F_2}^\vee$. Concretely, this sends the class of a cocycle $\xi$ to the mod $2$ relative cohomology class given by $\gamma'\mapsto\sum_{\gamma}\xi(1)(\gamma)$, where $\gamma'$ is a path in $T_1'$ with endpoints in $S_1'$ and the summation is over all lifts $\gamma$ of $\gamma'$ to a path in $T_1$.

It follows from the explicit descriptions that the map $\H^1(-F,\Lambda)\rightarrow\H^1(-F,\Lambda^\vee)$ corresponds to the map induced by the mod $2$ intersection pairing on $\Lambda_{1,\F_2}$, completing the proof.
\end{proof}
\end{proposition}

\begin{corollary}\label{cor:negative_ker}
Keep notation as above, and let $\tilde c$ be as in theorem~\ref{thm:formula}. Then
\[
\#\ker\left(\H^1(-F,\Lambda)\rightarrow\H^1(-F,\Lambda^\vee)\right) = \begin{cases}\tilde c&\text{if $S_1\neq\emptyset$,}\\\tilde c/2&\text{if $S_1=\emptyset$.}\end{cases}
\]
\begin{proof}
Let $M\colon\Lambda_{1,\F_2}\rightarrow\Lambda_{1,\F_2}^\vee$ be the $\F_2$-linear map induced by the mod $2$ intersection length pairing, as in proposition~\ref{prop:negative_ker}, so that the left-hand side of the desired equality is equal to $2^{n(M)}$ with $n(M)$ the nullity of $M$.

We divide into several cases. In all cases, we use that $\hat S'\setminus S'=T_1'\setminus S_1'$ consists of at most one component, allowing us to read off the value of $\tilde c$ from the definition in theorem~\ref{thm:formula}. We begin by disposing of a few small cases.

\begin{itemize}
	\item If $S_1'=\emptyset$, then $\tilde c=\gcd(0,2)=2$ and $n(M)=0$, so $\tilde c/2=2^{n(M)}$ in this case, as claimed.
	\item If $\#S_1'=1$, then $\tilde c=1$ and $n(M)=0$, so $\tilde c=2^{n(M)}$ in this case, as claimed.
	\item If $\#S_1'=2$, then $\tilde c=\gcd(l,2)$ with $l$ the distance between the two points of $S_1'$, and $M$ has nullity $0$ or $1$ according as $l$ is odd or even. It follows that $\tilde c=2^{n(M)}$ in this case, as claimed.
\end{itemize}

In the remaining cases, $T_1'$ contains a vertex of degree $\geq3$ and $S_1'$ is a nonempty set of degree $1$ vertices of $T_1'$. Moreover, the condition that $T$ was even forces that the distance between any two vertices of $T_1'$ of degree $\geq3$ is also even.

Now enumerate the elements of $S_1'$ as $v_1,v_2,\dots,v_b$, so that a basis of $\Lambda_{1,\F_2}$ is given by the paths from $v_1$ to $v_j$ for $j\geq2$. Supposing firstly that $S_1'$ has $a>0$ elements lying an even distance from a vertex of $T_1'$ of degree $\geq3$, then without loss of generality we may assume that $v_1$ is such an element. With respect to the above basis, the map $M$ is given by the matrix with $b-a$ diagonal entries equal to $1$ and all other entries $0$, so that $n(M)=a-1$. Thus $\tilde c=2^{a-1}=2^{n(M)}$ as claimed.

Supposing instead that $S_1'$ has no element lying an even distance from a vertex of $T_1'$, then with respect to the above basis the map $M$ is given by the $(b-1)\times(b-1)$ matrix whose diagonal entries are all $0$ and whose off-diagonal entries are all $1$. A simple calculation verifies that $n(M)=0$ if $b$ is odd and $n(M)=1$ if $b$ is even, whence $\tilde c=\gcd(b,2)=2^{n(M)}$ as desired.
\end{proof}
\end{corollary}

This completes the calculation of $\#\ker\left(\H^1(-F,\Lambda)\rightarrow\H^1(-F,\Lambda^\vee)\right)$. The corresponding calculation for $\#\coker\left(\Lambda^{-F}\hookrightarrow(\Lambda^\vee)^{-F}\right)$ is as follows.

\begin{proposition}\label{prop:negative_coker}
Let $\hat T$ denote the BY tree $(T,\hat S)$, where $\hat S$ is as in construction~\ref{cons:negative_part}. Then
\[
\#\coker\left(\Lambda^{-F}\hookrightarrow(\Lambda^\vee)^{-F}\right) = 
\begin{cases}
c_{\hat T,+F} & \text{if $S_1\neq\emptyset$,} \\
2c_{\hat T,+F} & \text{if $S_1=\emptyset$,}
\end{cases}
\]
where $+F$ denotes the induced even automorphism of $\hat T$, endowed with a sign of $+1$ everywhere.
\begin{proof}
Let $\Lambda_+=\H_1(T,\hat S,\Z)$ denote the relative homology lattice of $(T,\hat S)$, endowed with its intersection-length pairing $\langle\cdot,\cdot\rangle_+$. Since $\hat S=T_1\sqcup S_0$, the exact sequence on the homology of a pair shows that $\Lambda_+\cong\Z[S_0]$.

By construction, all $F$-orbits in $\pi_0(T\setminus\hat S)$ have even size, so $c_{\hat T,+F}=c_{\hat T,-F}$ by lemma~\ref{lem:product_of_signs}. Moreover, since $\H^1(-F,\Lambda_+)=0$ by proposition~\ref{prop:negative_cohomology}, the sequence
\[
0\rightarrow\Lambda_+\rightarrow\Lambda_+^\vee\rightarrow\gNeron_{\hat T}\rightarrow0
\]
remains exact on taking $-F$-invariants and so
\begin{equation}\label{eq:positive=negative}\tag{$\ast$}
c_{\hat T,+F}=c_{\hat T,-F}=\#\coker\left(\Lambda_+^{-F}\hookrightarrow(\Lambda_+^\vee)^{-F}\right)\,.
\end{equation}

Now the inclusion of pairs $\iota\colon(T,S)\hookrightarrow (T,\hat S)$ induces a map $\iota_*\colon\Lambda\rightarrow\Lambda_+$, as well as a dual map $\iota^*\colon\Lambda_+^\vee\rightarrow\Lambda^\vee$. It is straightforward to check that $\iota_*$ is the composite of the boundary map $\Lambda\hookrightarrow\Z[S]$ from the long exact sequence~\eqref{eq:LES_pair} of the pair $(T,S)$ and the projection $\Z[S]\twoheadrightarrow\Z[S_0]$. Taking $-F$-fixed points in~\eqref{eq:LES_pair} and using proposition~\ref{prop:negative_cohomology}, we see that the map $\iota_*\colon\Lambda^{-F}\xrightarrow\sim\Lambda_+^{-F}$ is an isomorphism. A similar argument with the dual sequence shows that we have an exact sequence
\[
0\rightarrow(\Lambda_+^\vee)^{-F}\xrightarrow{\iota^*}(\Lambda^\vee)\rightarrow\F_2\rightarrow\F_2[S_1']\,,
\]
with the right-hand map the diagonal map. In particular, $\iota^*\colon(\Lambda_+^\vee)^{-F}\rightarrow(\Lambda^\vee)$ is an isomorphism if $S_1\neq\emptyset$, and is injective with cokernel $\F_2$ if $S_1=\emptyset$.

Finally, we claim that the pairings on $\Lambda$ and $\Lambda_+$ are compatible on $-F$-fixed elements, in the sense that the square
\begin{equation}\label{eq:pairing_square}\tag{$\ast\ast$}
\begin{tikzcd}
\Lambda^{-F} \arrow[hook]{r}\arrow{d}{\iota_*}[swap]{\wr} & (\Lambda^\vee)^{-F} \\
\Lambda_+^{-F} \arrow[hook]{r} & (\Lambda_+^\vee)^{-F} \arrow[hook]{u}{\iota^*}
\end{tikzcd}
\end{equation}
commutes. In other words, we claim that
\[
\langle\iota_*(\gamma_0),\iota_*(\gamma_1)\rangle_+ = \langle\gamma_0,\gamma_1\rangle
\]
for all $\gamma_0\in\Lambda^{-F}$ and all $\gamma_1\in\Lambda$.

To show this, suppose that $e$ is an edge in $\hat S\setminus S$, so that $e$ lies in an $F$-orbit of odd size, say $2m+1$. In the notation of definition~\ref{def:intersection}, the multiplicity $m_e(\gamma_0)$ of $e$ in $\gamma_0$ satisfies
\[
m_e(\gamma_0)=m_{F^{2m+1}e}(F^{2m+1}\gamma_0)=-m_e(\gamma_0)\,,
\]
and so $m_e(\gamma_0)=0$. It follows that we have
\[
\langle\gamma_0,\gamma_1\rangle = \sum_{e\in E(T)\setminus E(\hat S)}m_e(\gamma_0)m_e(\gamma_1) = \langle\iota_*(\gamma_0),\iota_*(\gamma_1)\rangle_+\,.
\]
This shows that the square~\eqref{eq:pairing_square} commutes. In particular, we find using~\eqref{eq:positive=negative} that
\[
\#\coker\left(\Lambda^{-F}\hookrightarrow(\Lambda^\vee)^{-F}\right) = \#\coker(\iota^*)\cdot c_{\hat T,+F}\,.
\]
Combined with the description of $\coker(\iota^*)$ above, this yields the desired result.
\end{proof}
\end{proposition}

\begin{corollary}
Theorem~\ref{thm:formula} is true when $T$ is simple and $\epsilon F=-F$ is negative.
\begin{proof}
Combine~\eqref{eq:negative_c} with corollary~\ref{cor:negative_ker} and proposition~\ref{prop:negative_coker}, and corollary~\ref{cor:formula_positive} (applied to the positive BY tree $\hat T$).
\end{proof}
\end{corollary}

\begin{remark}\label{rmk:sometimes_not_even}
In certain special cases, the BY tree $\hat T=(T,\hat S)$ can fail to be even, or fail to be simple. For example, $T$ could be consists of a chain of $2l$ edges ($l$ odd) connecting the $2$ points of $S$, and $F$ could reverse the chain of edges in $T$. This is why we couldn't assume that $T$ was even in \S\ref{ss:positive}.
\end{remark}

By corollary~\ref{cor:reduce_to_simple}, this concludes the proof of theorem~\ref{thm:formula}.\qed

\subsection{Full proof of corollary~\ref{cor:qualitative}}\label{ss:qualitative_proof}

As mentioned in the earlier proof, corollary~\ref{cor:qualitative} in fact holds even when $T_0$ and $\epsilon F$ are not even. Although this is not the main thrust of this article, we briefly outline here how one can prove this in full generality. The key point is a structural result for the cohomology of the lattices $\Lambda$ associated to BY trees $T$.

\begin{lemma}\label{lem:B_is_2-torsion}
Let $T$ be a BY tree, with associated lattice $\Lambda=\H_1(T,S,\Z)$, and let $\epsilon F$ be a signed automorphism of $T$. Then the image of the map
\[
\H^1(\epsilon F,\Lambda)\rightarrow\H^1(\epsilon F,\Lambda^\vee)
\]
induced by the intersection-length pairing on $\Lambda$ is contained in the $2$-torsion subgroup of $\H^1(\epsilon F,\Lambda^\vee)$.
\begin{proof}[Proof (sketch)]
The proof of lemma~\ref{lem:reduction_LHS} shows that the map $\H^1(\epsilon F,\Lambda)\rightarrow\H^1(\epsilon F,\Lambda^\vee)$ is the direct sum of the corresponding maps for simple BY trees. It thus suffices to prove this for simple BY trees. In the case where $\epsilon=-1$, this follows since $\H^1(-F,\Lambda)\leq\H^1(-F,\Z[S])=\F_2^{\#\text{odd orbits in $S$}}$, as in the proof of proposition~\ref{prop:negative_ker}. In the case where $\epsilon=+1$, the same argument as in the proof of proposition~\ref{prop:positive_ker} shows that $\H^1(+F,\Lambda)$ is cyclic, generated by the cocycle associated to $(1-F)y$, for $y$ some vertex of $T$. If we let $y_0$ be an $F$-fixed point of $T$, which may be the midpoint of an edge, then intersection with the path from $y_0$ to $y$ gives an element $\alpha\in\frac12\Lambda^\vee$ such that $(1-F)\alpha$ is the image of $(1-F)y$ under the map $\Lambda\hookrightarrow\Lambda^\vee$. It follows that the image of the cocycle associated to $2(1-F)y$ is the coboundary of $2\alpha$, and hence the zero element of $\H^1(+F,\Lambda^\vee)$.
\end{proof}
\end{lemma}

\begin{remark}
The image of the map $\H^1(\epsilon F,\Lambda)\rightarrow\H^1(\epsilon F,\Lambda^\vee)$ is the group $\mathfrak B_\Lambda$ studied in \cite[Definition~1.4.1]{finite_quotients}.
\end{remark}

To see how lemma~\ref{lem:B_is_2-torsion} implies corollary~\ref{cor:qualitative}, we recall that the subdivision $T_0^{(l)}$ is isometric to the metric BY tree $(T_0,l)$ (remark~\ref{rmk:metric_vs_subdivisions}), and hence we have
\[
c_{T_0^{(l)}}=\#\coker\left(\beta_l\colon\Lambda^{\epsilon F}\hookrightarrow(\Lambda^\vee)^{\epsilon F}\right)\cdot\#\ker\left(\beta_{l,*}\colon\H^1(\epsilon F,\Lambda)\rightarrow\H^1(\epsilon F,\Lambda^\vee)\right) \,,
\]
where $\Lambda=\H_1(T_0,S_0,\Z)$ and $\beta_l\colon\Lambda\hookrightarrow\Lambda^\vee$ is the map induced by the intersection length pairing of the metric BY tree $(T_0,l)$. Note that the map $\beta_l\in\Hom(\Lambda,\Lambda^\vee)$ depends linearly on the function~$l$.

Now the order of the cokernel of $\beta_l\colon\Lambda^{\epsilon F}\hookrightarrow(\Lambda^\vee)^{\epsilon F}$ is given by the absolute determinant of the matrix representation $\beta_l$ with respect to bases of $\Lambda^{\epsilon F}$ and $(\Lambda^\vee)^{\epsilon F}$. Since $\beta_l$ depends linearly on $l$, it follows that $\#\coker\left(\beta_l\colon\Lambda^{\epsilon F}\hookrightarrow(\Lambda^\vee)^{\epsilon F}\right)$ is a homogenous polynomial in~$l$.

On the other hand, by lemma~\ref{lem:B_is_2-torsion}, the map $\beta_{l,*}\colon\H^1(\epsilon F,\Lambda)\rightarrow\H^1(\epsilon F,\Lambda^\vee)[2]$ is a map of finite $\Z$-modules whose codomain is $2$-torsion, depending linearly on $l$, and hence $\#\ker\left(\beta_{l,*}\colon\H^1(\epsilon F,\Lambda)\rightarrow\H^1(\epsilon F,\Lambda^\vee)\right)$ is a constant times a power of $2$ depending only on the values of $l$ mod $2$. Combining these shows that $c_{T_0^{(l)}}$ is a function of $l$ of the desired form.\qed

\subsection{Worked examples}
\label{ss:examples}

\begin{example}\label{ex:simple}
Suppose that $p\equiv-1$ modulo $4$, and let $f\in\Z_p[x]$ be a monic polynomial of degree eight such that $f(0)$ is a square in~$\Z_p$. Assume that the Newton polygon of $f(x)$ (resp.\ $f(x-1)$, resp.\ $f(x-i)$ with $i$ a square root of $-1$) consists of six segments of slope $0$ and two segments of slope $a/2>0$ (resp.\ $b/2>0$, resp.\ $c/2>0$). This says that the eight roots of $f$ come in four pairs: the first lie in the residue disc of $0$ and are equidistant from it, and the same is true for the other three pairs in the residue discs of $1$, $i$ and $-i$. The root cluster machinery of \cite{m2d2} shows that $X_f$ is semistable \cite[Theorem~1.8(1)]{m2d2}, and allows us to write down the BY tree of the hyperelliptic curve $X_f$ with affine equation $y^2=f(x)$ (figure~\ref{fig:simple}) \cite[Theorems~5.18~\&~6.9]{m2d2}. For ease of calculation, we work with the metric BY tree produced from $f$ via \cite[Definition~D.4]{m2d2}, which is isometric to the BY tree associated to $X_f$.

%\vspace{-0.5cm}
\begin{figure}[!h]
\caption{The metric BY tree associated to $X_f$}
\label{fig:simple}
\vspace{0.3cm}
\begin{center}
\begin{tikzpicture}[scale=1.0]
	\YellowVertices
	\Vertex[x=1.0,y=1.0,L=\relax]{o};
	\BlueVertices
	\Vertex[x=0.0,y=0.0,L=\relax]{sw};
	\Vertex[x=2.0,y=0.0,L=\relax]{se};
	\Vertex[x=0.0,y=2.0,L=\relax]{nw};
	\Vertex[x=2.0,y=2.0,L=\relax]{ne};
	
	\YellowEdges
	\Edge(o)(ne)
	\Edge(o)(nw)
	\Edge(o)(se)
	\Edge(o)(sw)
	
	\TreeEdgeSignW(o)(ne){0.5}{c}
	\TreeEdgeSignW(o)(se){0.5}{c}
	\TreeEdgeSignW(o)(sw){0.5}{a}
	\TreeEdgeSignW(o)(nw){0.5}{b}
	
	\TreeEdgeSignW(o)(nw){0.9}{+}
	
	\ESwap{o}{ne}{o}{se}{in=45,out=315}
\end{tikzpicture}
\end{center}
\end{figure}

In this diagram, the whole graph represents the tree $T$, while the blue/solid vertices represent the vertices of $S$ (which has no edges in this example) -- by contrast, the vertices of $T$ not in $S$ are represented by yellow/open circles and the edges of $T$ not in $S$ are represented by yellow/squiggly lines. The lengths of the edges are indicated by the parameters $a$, $b$ and $c$, while the signed automorphism is indicated both with double-headed arrows for the underlying unsigned automorphism of $(T,S)$ (which here has order $2$) and with $\pm$ signs next to each connected component\footnote{Here our diagrammatic conventions differ slightly from those in \cite{semistable_types}, where the sign labels are attached to each \emph{orbit} of connected components of $T\setminus S$, and record the \emph{total} sign of the automorphism over the entire orbit.} of $T\setminus S$ (so here the sign is $+$).

According to theorem~\ref{thm:tamagawa_numbers_via_graphs} and proposition~\ref{prop:hyperelliptic-BY_equivalence_jacobians}, the Tamagawa number of $X_f/\Q_p$ is the same as the Tamagawa number of $T$, which we calculate using theorem~\ref{thm:formula} (metric version). In the notation of that theorem, we have $Q=2$ and $\tilde c=1$ (the empty product, since $\hat S=S$). The quotient tree~$T'$ is
\begin{center}
\begin{tikzpicture}[scale=1.0]
	\YellowVertices
	\Vertex[x=0.0,y=0.0,L=\relax]{o};
	\BlueVertices
	\Vertex[x=1.5,y=0.0,L=\relax]{e};
	\Vertex[x=-0.75,y=1.3,L=\relax]{nw};
	\Vertex[x=-0.75,y=-1.3,L=\relax]{sw};
	
	\YellowEdges
	\Edge(o)(e)
	\Edge(o)(sw)
	\Edge(o)(nw)
	
	\TreeEdgeSignS(o)(e){0.5}{c/2}
	\TreeEdgeSignW(o)(sw){0.4}{a}
	\TreeEdgeSignW(o)(nw){0.4}{b}
\end{tikzpicture}
\end{center}
where again the blue/solid vertices indicate the subset $S'=\hat S'\subseteq T'$, and the labels indicate the quantities $l(e')/q_{e'}$. The removal of any two of the three edges of $T'$ disconnects the three points of $S'$ from one another, and hence the formula in theorem \ref{thm:formula} provides that the Tamagawa number of $X_f$ is\[c_{X_f}=2\cdot\left(ab+b\frac c2+\frac c2a\right)=2ab+bc+ca.\]
\end{example}

\begin{example}\label{ex:comprehensive}
As a more comprehensive example to illustrate all the aspects of our formula, let us find the Tamagawa number of the following metric BY tree $T$ (which would arise from a semistable hyperelliptic curve of genus $\geq11$). This is even (definition~\ref{def:parity}) provided that the edge-lengths $a$, $w$ and $x$ are all even, which we now assume.
%\begin{center}
%\vspace{-0.4cm}
\begin{figure}[!h]
\label{fig:comprehensive_example}
\caption{A more comprehensive example of a BY tree}
\vspace{0.3cm}
\begin{tikzpicture}[scale=1.0]
	\YellowVertices
	\Vertex[x=-3.0,y=0.0,L=\relax]{l};
	\Vertex[x=0.75,y=1.3,L=\relax]{ur};
	\Vertex[x=0.75,y=-1.3,L=\relax]{dr};
	\BlueVertices
	\Vertex[x=-4.5,y=0.0,L=\relax]{lw};
	\Vertex[x=-3.75,y=1.3,L=\relax]{lnw};
	\Vertex[x=-2.25,y=1.3,L=\relax]{lne};
	\Vertex[x=-1.5,y=0.0,L=\relax]{le};
	\Vertex[x=-2.25,y=-1.3,L=\relax]{lse};
	\Vertex[x=-3.75,y=-1.3,L=\relax]{lsw};
	\Vertex[x=-0.55,y=2.05,L=\relax]{urw};
	\Vertex[x=1.5,y=2.6,L=\relax]{urn};
	\Vertex[x=2.05,y=0.55,L=\relax]{ure};
	\Vertex[x=0.0,y=0.0,L=\relax]{urs=drn};
	\Vertex[x=2.05,y=-0.55,L=\relax]{dre};
	\Vertex[x=1.5,y=-2.6,L=\relax]{drs};
	\Vertex[x=-0.55,y=-2.05,L=\relax]{drw};
	
	\BlueEdges
	\Edge(le)(urs=drn)
	\YellowEdges
	\Edge(l)(le)
	\Edge(l)(lne)
	\Edge(l)(lnw)
	\Edge(l)(lw)
	\Edge(l)(lsw)
	\Edge(l)(lse)
	\Edge(ur)(urw)
	\Edge(ur)(urn)
	\Edge(ur)(ure)
	\Edge(ur)(urs=drn)
	\Edge(dr)(urs=drn)
	\Edge(dr)(dre)
	\Edge(dr)(drs)
	\Edge(dr)(drw)
	
	\TreeEdgeSignS(le)(urs=drn){0.5}{w}
	\TreeEdgeSignS(l)(lw){0.6}{z}
	\TreeEdgeSignW(l)(lnw){0.5}{y}
	\TreeEdgeSignE(l)(lne){0.5}{y}
	\TreeEdgeSignS(l)(le){0.5}{x}
	\TreeEdgeSignW(l)(lse){0.7}{z}
	\TreeEdgeSignW(l)(lsw){0.5}{z}
	\TreeEdgeSignN(ur)(urw){0.5}{b}
	\TreeEdgeSignW(ur)(urn){0.5}{c}
	\TreeEdgeSignS(ur)(ure){0.4}{c}
	\TreeEdgeSignW(ur)(urs=drn){0.5}{a}
	\TreeEdgeSignW(dr)(urs=drn){0.5}{a}
	\TreeEdgeSignN(dr)(dre){0.4}{c}
	\TreeEdgeSignW(dr)(drs){0.5}{c}
	\TreeEdgeSignS(dr)(drw){0.5}{b}
	
	\TreeEdgeSignE(l)(lnw){0.9}{-}
	\TreeEdgeSignS(ur)(urw){0.9}{-}
	\TreeEdgeSignN(dr)(drw){0.9}{-}
	
	\ESwap{l}{lnw}{l}{lne}{in=150,out=30}
	\EArr{l}{lw}{l}{lsw}{in=150,out=270}
	\EArr{l}{lsw}{l}{lse}{in=210,out=330}
	\EArr{l}{lse}{l}{lw}{in=315,out=165}
	\ESwap{urs=drn}{ur}{urs=drn}{dr}{in=90,out=270}
	\ESwap{ur}{urw}{dr}{drw}{in=120,out=240}
	\EArr{ur}{urn}{dr}{dre}{in=75,out=330}
	\EArr{dr}{dre}{ur}{ure}{in=270,out=90}
	\EArr{ur}{ure}{dr}{drs}{in=30,out=285}
	\EArr{dr}{drs}{ur}{urn}{in=270,out=90}
\end{tikzpicture}
\end{figure}
%\vspace{-0.1cm}
%\end{center}

Here, as before, the subgraph $S\subseteq T$ is represented by blue/solid vertices and now also has a single edge, rendered blue/straight -- the yellow/open circles and yellow/squiggly lines indicate the vertices and edges of $T$ not in $S$, respectively. The parameters $a$, $b$, $c$, $w$, $x$, $y$, $z$ indicate edge-lengths. The signed automorphism is both indicated by the arrows (identifying the orbits of the underlying unsigned automorphism as it acts on edges -- these orbits have sizes $2$, $2$, $2$, $3$ and $4$ respectively), and the signs next to each connected component of $T\setminus S$ (all $-$).

Now the quotient $T'$ of $T$ by the action of $F$ is as depicted below. Here again the blue/solid vertices and edges indicate the subgraph $S'$, while the green/square vertices and green/dashed edges indicate the extra vertices and edges in $\hat S'$. The labels on yellow/squiggly edges $e'$ (edges outside $\hat S'$) indicate the quantity $l(e')/q_{e'}$, while the labels on green/dashed edges indicate the quantity $l(e')$. We don't attach labels to blue/solid edges, since the values of $l(e')$ and $q_{e'}$ for these edges don't contribute to theorem~\ref{thm:formula}.

\begin{center}
\begin{tikzpicture}[scale=1.0]
	\BlueVertices
	\Vertex[x=-0.75,y=0,L=\relax]{cl};
	\Vertex[x=0.75,y=0,L=\relax]{cr};
	\Vertex[x=3,y=1.3,L=\relax]{ru};
	\Vertex[x=3,y=-1.3,L=\relax]{rd};
	\Vertex[x=-3,y=1.3,L=\relax]{lu};
	\Vertex[x=-3,y=-1.3,L=\relax]{ld};
	\GreenVertices
	\Vertex[x=-2.25,y=0,L=\relax]{l};
	\YellowVertices
	\Vertex[x=2.25,y=0,L=\relax]{r};
	
	\BlueEdges
	\Edge(cl)(cr)
	\GreenEdges
	\Edge(l)(cl)
	\Edge(l)(ld)
	\YellowEdges
	\Edge(r)(cr)
	\Edge(r)(ru)
	\Edge(r)(rd)
	\Edge(l)(lu)
	
	\TreeEdgeSignS(cr)(r){0.5}{a/2}
	\TreeEdgeSignE(r)(ru){0.5}{b/2}
	\TreeEdgeSignE(r)(rd){0.5}{c/4}
	\TreeEdgeSignS(l)(cl){0.5}{x}
	\TreeEdgeSignW(l)(lu){0.5}{y/2}
	\TreeEdgeSignW(l)(ld){0.5}{z}
\end{tikzpicture}
\end{center}

Following through the description in theorem~\ref{thm:formula} (metric version), we find that $Q=16$ and $\tilde c=\gcd(z,2)$ (since $x$ is even) and $r=3$. In order for a triple of edges to disconnect the four components of $\hat S'$, it must contain the left-hand yellow/squiggly edge and two of the three yellow/squiggly edges on the right. Thus we find from theorem~\ref{thm:formula} (metric version) that
\[
c_T = 16\cdot\gcd(z,2)\cdot\left(\frac{yab}8+\frac{ybc}{16}+\frac{yca}{16}\right) = y\cdot\gcd(z,2)\cdot(2ab+bc+ca)\,.
\]

\smallskip

Alternatively, one can use lemma~\ref{lem:reduction_LHS} to simplify the calculation of the Tamagawa number of $T$, finding that it is the product of the Tamagawa numbers of the following two metric BY trees.
\begin{center}
\begin{tikzpicture}[scale=1.0]
	\YellowVertices
	\Vertex[x=-3.5,y=0.0,L=\relax]{l};
	\Vertex[x=0.0,y=0.0,L=\relax]{m};
	\BlueVertices
	\Vertex[x=-5.0,y=0.0,L=\relax]{lw};
	\Vertex[x=-4.25,y=1.3,L=\relax]{lnw};
	\Vertex[x=-2.75,y=1.3,L=\relax]{lne};
	\Vertex[x=-2.0,y=0.0,L=\relax]{le};
	\Vertex[x=-2.75,y=-1.3,L=\relax]{lse};
	\Vertex[x=-4.25,y=-1.3,L=\relax]{lsw};
	\Vertex[x=-1.0,y=1.0,L=\relax]{mnw};
	\Vertex[x=1.0,y=1.0,L=\relax]{mne};
	\Vertex[x=1.0,y=-1.0,L=\relax]{mse};
	\Vertex[x=-1.0,y=-1.0,L=\relax]{msw};
	
	\YellowEdges
	\Edge(l)(le)
	\Edge(l)(lne)
	\Edge(l)(lnw)
	\Edge(l)(lw)
	\Edge(l)(lsw)
	\Edge(l)(lse)
	\Edge(m)(mnw)
	\Edge(m)(mne)
	\Edge(m)(mse)
	\Edge(m)(msw)
	
	\TreeEdgeSignS(l)(lw){0.6}{z}
	\TreeEdgeSignW(l)(lnw){0.5}{y}
	\TreeEdgeSignE(l)(lne){0.5}{y}
	\TreeEdgeSignS(l)(le){0.5}{x}
	\TreeEdgeSignW(l)(lse){0.7}{z}
	\TreeEdgeSignW(l)(lsw){0.5}{z}
	\TreeEdgeSignW(m)(mnw){0.5}{b}
	\TreeEdgeSignW(m)(mne){0.5}{c}
	\TreeEdgeSignW(m)(mse){0.5}{c}
	\TreeEdgeSignW(m)(msw){0.5}{a}
	
	\TreeEdgeSignE(l)(lnw){0.9}{-}
	\TreeEdgeSignW(m)(mnw){0.9}{+}
	
	\ESwap{l}{lnw}{l}{lne}{in=150,out=30}
	\EArr{l}{lw}{l}{lsw}{in=150,out=270}
	\EArr{l}{lsw}{l}{lse}{in=210,out=330}
	\EArr{l}{lse}{l}{lw}{in=315,out=165}
	\ESwap{m}{mne}{m}{mse}{in=45,out=315}
\end{tikzpicture}
\end{center}
We computed the Tamagawa number of the right-hand BY tree as $2ab+bc+ca$ in example~\ref{ex:simple}, and theorem~\ref{thm:formula} shows that the Tamagawa number of the left-hand BY tree is $y\cdot\gcd(x+z,2)=y\cdot\gcd(z,2)$ since $x$ is even (we omit details). This recovers the above value for $c_T$.
\end{example}

\begin{example}
Up to signed homeomorphism, there are only finitely many BY trees (with signed automorphism) arising from semistable hyperelliptic curves of a given genus (in residue characteristic $p\neq2$). Although this encompasses infinitely many signed \emph{iso}morphism types, one can use theorem~\ref{thm:formula} to write down a formula for the Tamagawa numbers of all BY trees in a fixed homeomorphism class (as a function of the ``edge-lengths''), and hence produce a complete list of the Tamagawa numbers of all semistable hyperelliptic curves of a given genus, as a function in the reduction type (signed homeomorphism class of BY tree with signed automorphism). For instance, in genus~$2$, there are twenty-three possible reduction types, listed along with their Tamagawa numbers in \cite[Table~9.3]{semistable_types}. The reader can verify that the formula in theorem~\ref{thm:formula} recovers the expressions for the Tamagawa numbers in \cite[Table~9.3]{semistable_types}. Theorem~\ref{thm:formula} makes it essentially routine to produce similar tables in higher genus: one lists all BY trees of a given genus (in the sense of \cite[Definitions~3.18 and~3.23]{semistable_types}) and then reads off the corresponding Tamagawa numbers. In genus $3$ there are one hundred and eighty-five possible reduction types, though many are degenerate forms of one another.
\end{example}

\section{Growth of Tamagawa numbers in towers}
\label{s:growth}

In this final part of this paper, we use the above techniques to examine how Tamagawa numbers of semistable hyperelliptic curves $X/K$ vary as we enlarge the ground field $K$. On the combinatorial side, changing the base field corresponds to changing the BY tree in a particularly simple manner.

\begin{lemma}\label{lem:change_of_BY_trees_in_field_extensions}
Let $X/K$ be a semistable hyperelliptic curve. Let $T=(T,S)$ be its associated BY tree, and $\epsilon F$ the signed automorphism of $T$ induced by the Frobenius over $K$, as in \S\ref{sss:BY_tree_of_curve}. Then, for all finite extensions\footnote{For the proof, we will assume that $L$ is a subfield of $\overline K$, so that the residue field of $L$ is a subfield of $\overline k$.} $L/K$, the BY tree associated to $X_L/L$ is the BY tree $T^{(e)}=(T^{(e)},S^{(e)})$ formed from $T=(T,S)$ by replacing each edge of $T$ by a chain of $e$ edges, where $e=e(L/K)$ is the ramification degree, as in \S\ref{sss:graph_conventions}. The signed automorphism of $T^{(e)}$ induced by the Frobenius over $L$ is $(\epsilon F)^f$, where $f=f(L/K)$ is the residue class degree. Here, by abuse of notation, we identify $\Aut^{\pm}(T)=\Aut^{\pm}(T^{(e)})$ in the obvious way.
\begin{proof}
Let $\mathfrak X/\O_K$ be the minimal regular model of $X/K$. This is a semistable model \cite[Theorem~10.3.34(a)]{liu:algebraic-arithmetic}, and regularity implies that every singular point of the special fibre has thickness $1$, i.e.\ has an \'etale neighbourhood given by the equation $xy=\varpi_K$ with $\varpi_K$ a uniformiser of $K$. The base change $\mathfrak X_{\O_L}$ is a semistable model of $X_L$ \cite[Corollary~10.3.36(a)]{liu:algebraic-arithmetic}, but not in general regular, since it is easy to see that all singular points of its special fibre have thickness $e$. It follows from \cite[Corollary~10.3.25]{liu:algebraic-arithmetic} that the geometric special fibre of the minimal regular model of $X_L$ is the same as $\mathfrak X_{\overline k}$, but with each intersection of two components replaced by a chain of $e-1$ projective lines connecting these components (if $e=1$ then there is no change to the geometric special fibre). In the case of self-intersections (nodal singularities of components), this means that the self-intersecting component is replaced with its normalisation, with a chain of $e-1$ projective lines connecting the two points lying above each node. The action of the Frobenius $\Frob_L\in\mathrm{Gal}(\overline K/L)$ on the geometric special fibre is the $f$th power of the Frobenius element of $\mathrm{Gal}(\overline k/k)$.

In graph-theoretic terms, this says that the reduction graph of $X_L/L$ is the $e$th subdivision $\G^{(e)}$ of the reduction graph $\G$ of $X/K$, and that $\Frob_L$ acts on $\G^{(e)}$ via (the $e$th subdivision of) $\Frob_K^f$. The corresponding assertions regarding the BY trees follow.
\end{proof}
\end{lemma}

Thus, theorem~\ref{thm:growth_in_towers} boils down to the purely combinatorial question of determining how the Tamagawa numbers $c_{T^{(e)},(\epsilon F)^f}$ depend on the parameters $e,f$. This is given by the following result, of which theorem~\ref{thm:growth_in_towers} is an immediate corollary.

\begin{theorem}\label{thm:combinatorial_growth_in_towers}
Let $T=(T,S)$ be a BY tree and let $\epsilon F$ be a signed automorphism of $T$. Suppose that $F$ is even (see Definition~\ref{def:parity}). Then there are $(a_d,r_d,s_d)\in\N\times\N_0\times\Z$ for each $d\in\N$ (equal to $(1,0,0)$ for almost all $d$) such that
\[
c_{T^{(e)},(\epsilon F)^f} = \prod_{d\mid f}\left(a_d\cdot e^{r_d}\cdot\gcd(e,2)^{s_d}\right)^{\varphi(d)}
\]
for all $e,f\in\N$.
\end{theorem}

\subsection{Totally ramified extensions}\label{ss:tot_ram}

We will first prove theorem~\ref{thm:combinatorial_growth_in_towers} in two special cases: the ``totally ramified'' case $f=1$, and the ``unramified case'' $e=1$. The former is an immediate consequence of corollary~\ref{cor:qualitative}.

\begin{proposition}\label{prop:growth_in_totally_ramified_towers}
Let $T=(T,S)$ be a BY tree and let $\epsilon F$ be a signed automorphism of $T$. Then there are $(a,r,s)\in\N\times\N_0\times\Z$ such that
\[
c_{T^{(e)},\epsilon F} = a\cdot e^r\cdot\gcd(e,2)^s
\]
for all $e\in\N$.
\end{proposition}

\subsection{Unramified extensions}

Having dealt with the dependency on $e$ in theorem~\ref{thm:combinatorial_growth_in_towers}, it now remains to control the dependency on $f$, fixing $e=1$. In other words, by definition~\ref{def:intersection}, we want to control the number of $(\epsilon F)^f$-fixed points in the Jacobian group $\gNeron_T$ as a function of $f$. We will develop tools to deal with such problems in appendix \ref{app:fixpoints}, but for now let us just record the definitions and basic properties we will need.

\begin{definition}[Fixpoint filtrations]
Let $A$ be a $\Z$-module with an endomorphism $\sigma$ (for us, always a finite-order automorphism). We define the \emph{$\sigma$-fixpoint filtration} of $A$ to be the family of sub-$\Z[\sigma]$-modules $A^{\sigma^f}$, which come with inclusions $A^{\sigma^d}\leq A^{\sigma^f}$ whenever $d\mid f$. We also define the \emph{partial quotients} of $A$ with respect to this filtration to be\[\Gr^\sigma_f(A):=\frac{A^{\sigma^f}}{\sum_{d\mid f,d\neq f}A^{\sigma^d}}.\]
\end{definition}

\begin{lemma}\label{lem:fixpoint_properties}
Let $A$ be a $\Z$-module with an endomorphism $\sigma$. Then, for any $f\in\N$, $A^{\sigma^f}$ has an exhaustive, separated\footnote{Recall that an increasing filtration $\dots\leq\Fil_{-1}M\leq\Fil_0M\leq\Fil_1M\leq\dots$ on a module~$M$ is called \emph{exhaustive} just when $\bigcup_i\Fil_iM=M$, and \emph{separated} just when $\bigcap_i\Fil_iM=0$.} filtration with partial quotients $\Gr^\sigma_d(A)$ for $d\mid f$. In particular, if $A^{\sigma^f}$ is finite then\[\#A^{\sigma^f}=\prod_{d\mid f}\#\Gr^\sigma_d(A).\]

Moreover, the $\Z[\sigma]$-module structure on the partial quotient $\Gr^\sigma_f(A)$ factors canonically through the quotient $\Z[\sigma]\twoheadrightarrow\Z[\mu_f]:=\Z[\sigma]/(\Phi_f(\sigma))$, where $\Phi_f$ is the $f$th cyclotomic polynomial.

Finally, if $B$ is another $\Z[\sigma]$-module and $C$ is a $\Z[\sigma^q]$-module for some $q\in\N$, we have $\Z[\sigma]$-module isomorphisms\[\Gr^\sigma_f(A\oplus B)\cong\Gr^\sigma_f(A)\oplus\Gr^\sigma_f(B)\]and\[\Gr^\sigma_f\left(\Ind_{\sigma^q}^\sigma C\right)\cong\Gr^{\sigma^q}_{\num{f/q}}(C)\otimes_{\Z[\mu_{\num{f/q}}]}\Z[\mu_f]\]where $\num{f/q}$ denotes the numerator of $f/q$.
\begin{proof}
Deferred to appendix.
\end{proof}
\end{lemma}

The concept of fixpoint filtrations and partial quotients gives an integral analogue of the isotypic decomposition of representations of cyclic groups on $\Q$-vector spaces. However, while the isotypic pieces of such representations are well-behaved, being free modules over the vector spaces $\Q[\mu_n]$, much less can be expected in general for the partial quotients of $\Z[\sigma]$-modules. Thus, it will be useful for us to isolate a class of $\Z[\sigma]$-modules where the fixpoint filtration is well-behaved.

\begin{definition}\label{def:fixpoint-regular}
Let $A$ be a $\Z$-module with a finite-order automorphism $\sigma$. We say that $A$ is \emph{fixpoint-regular} just when there is a collection $(A_d)_{d\in\N}$ of $\Z$-modules such that we have $\Z[\mu_d]$-module isomorphisms $\Gr^\sigma_d(A)\isoarrow A_d\otimes_\Z\Z[\mu_d]$ for all $d$. (The $A_d$ are necessarily trivial for $d$ not dividing the order of $\sigma$.)

It is clear from lemma \ref{lem:fixpoint_properties} that the class of fixpoint-regular $\Z[\sigma]$-modules is closed under direct sums and induction from $\Z[\sigma^q]$ to $\Z[\sigma]$ for any $q\in\N$.
\end{definition}

The main theorem of this section asserts that Jacobian groups of BY trees are well-behaved in the above sense.

\begin{theorem}
\label{thm:structure_of_partial_quotients_precise}
Let $T=(T,S)$ be a BY tree and let $\epsilon F$ be an even signed automorphism of $T$. Then the Jacobian group $\gNeron_T$ is fixpoint-regular (for the action of $\epsilon F$).
\end{theorem}

\begin{corollary}\label{cor:growth_in_unramified_towers}
Let $T=(T,S)$ be a BY tree and let $\epsilon F$ be an even signed automorphism of $T$. Then there are $(a_d)_{d\in\N}\in\N$ (equal to $1$ for $d$ not dividing the order of $\epsilon F$) such that
\[
c_{T,(\epsilon F)^f} = \prod_{d\mid f}a_d^{\varphi(d)}
\]
for all $f\in\N$.
\begin{proof}
This follows from lemma~\ref{lem:fixpoint_properties}. The desired $a_d$ are the sizes of the groups $A_d$ from definition~\ref{def:fixpoint-regular}.
\end{proof}
\end{corollary}

Our proof of theorem~\ref{thm:structure_of_partial_quotients_precise} is ultimately inductive, and revolves around the following key lemma, allowing us to add and remove $F$-fixed points from the set $S$.

\begin{lemma}\label{lem:adding_points}
Let $T=(T,S)$ be a BY tree with an even signed automorphism $\epsilon F$, and suppose that $*$ is an $F$-fixed vertex of $T$ not lying in $S$. Write $T_*$ for the BY tree $(T,S\cup\{*\})$. Then $\gNeron_T$ is fixpoint-regular for the action of $\epsilon F$ if and only if $\gNeron_{T_*}$ is. Here, $\epsilon F$ is viewed as a signed automorphism of $T_*$ whose sign function is the composite $\pi_0(T\setminus(S\cup\{*\}))\twoheadrightarrow\pi_0(T\setminus S)\xrightarrow\epsilon\{\pm1\}$.
\begin{proof}
The intersection-length pairings on the lattices $\Lambda=\H_1(T,S,\Z)$ and $\Lambda_*=\H_1(T,S\cup\{*\},\Z)$ induce a commuting square
\begin{center}
\begin{tikzcd}
\Lambda \arrow[hook]{r}\arrow[hook]{d} & \Lambda^\vee \\
\Lambda_* \arrow[hook]{r} & \Lambda_*^\vee \arrow[two heads]{u}
\end{tikzcd}
\end{center}
in which all arrows are equivariant for the actions of $\epsilon F$. If we define $\Pi:=\Lambda_*^\vee/\Lambda$, then there are equivariant surjections $\Pi\twoheadrightarrow\Lambda^\vee/\Lambda=\gNeron_T$ and $\Pi\twoheadrightarrow\Lambda_*^\vee/\Lambda_*=\gNeron_{T_*}$, whose kernels are the free rank $1$ $\Z$-modules $\ker\left(\Lambda_*^\vee\twoheadrightarrow\Lambda^\vee\right)$ and $\Lambda_*/\Lambda$, respectively. We will see in proposition~\ref{prop:rank_1_surj} that this implies that $\Pi$ is fixpoint-regular if and only if $\gNeron_T$, respectively $\gNeron_{T_*}$, is, and hence we are done.
\end{proof}
\end{lemma}

\begin{proof}[Proof of theorem~\ref{thm:structure_of_partial_quotients_precise}]
We proceed by strong induction on the quantity
\[
\alpha(T):=\#E(T)+\#V(T)-\#V(S)\,.
\]
Let $T$ be a BY tree, and suppose that the result is true for all BY trees $T'$ such that $\alpha(T')<\alpha(T)$ (with respect to all even signed automorphisms of $T'$). Let $\epsilon F\in\Aut^{\pm}(T)$ be even. We divide into three cases.

Firstly, if $T$ has an $F$-fixed vertex $*$ not lying in $S$, then our inductive assumption ensures that $\gNeron_{(T,S\cup\{*\})}$ is fixpoint-regular for the action of $\epsilon F$. Lemma~\ref{lem:adding_points} implies that the same is true of $\gNeron_T$ and we are done in this case.

Secondly, if $T$ is not simple, then the proof of corollary~\ref{cor:reduce_to_simple} shows that
\[
\gNeron_T\cong\bigoplus_{i\in I}\Ind_{(\epsilon F)^{q_i}}^{\epsilon F}\gNeron_{T_i}
\]
where each $T_i$ is the closure of a connected component of $T\setminus S$. It is easy to see that $\alpha(T_i)<\alpha(T)$ for all $i$, and hence $\gNeron_T$ is a direct sum of inductions of fixpoint-regular modules, so is itself fixpoint-regular.

It remains to deal with the case that $T$ is simple, and that all of its $F$-fixed vertices lie in $S$. This implies that every $F$-fixed vertex must have degree $1$ in $T$, which is only possible if $T$ consists of a single edge, $S$ contains the endpoints of $T$, and $F$ is the identity. In this case, $\epsilon F$ acts as $\pm1$ on $\gNeron_T$, and hence is fixpoint-regular.
\end{proof}

\subsection{General extensions}

We now formally combine the results of the previous two sections to prove theorem~\ref{thm:structure_of_partial_quotients_precise} in general.

\begin{proof}[Proof of theorem~\ref{thm:structure_of_partial_quotients_precise}]
By corollary \ref{cor:growth_in_unramified_towers}, we know that, for each $e$ there are $a_d(e)\in\N$ (equal to $1$ for $d$ not dividing the order of $\epsilon F$) such that\[c_{T^{(e)},(\epsilon F)^f}=\prod_{d\mid f}a_d(e)^{\varphi(d)}\]for all $f$. But by proposition \ref{prop:growth_in_totally_ramified_towers}, for fixed $f$ the Tamagawa number $c_{T^{(e)},(\epsilon F)^f}$ is of the form $ae^r\gcd(e,2)^s$ for $a\in\Q^\times$ and $r,s\in\Z$. Applying the M\"obius inversion formula to the product representation of $c_{T^{(e)},(\epsilon F)^f}$ above, we see that $a_d(e)^{\varphi(d)}$ must also be a function of $e$ of such a form. But $a_d(e)$ is a positive integer for all $e$, which forces it to be of the form $a_d(e)=a_de^{r_d}\gcd(e,2)^{s_d}$ where $a_d\in\N$ and $r_d\in\N_0$ as desired.
\end{proof}

\appendix

\section{Fixpoint filtrations}\label{app:fixpoints}

In this appendix, we set out the basic properties of fixpoint filtrations, aiming to justify the content of lemma \ref{lem:fixpoint_properties}. Recall that we are considering $\Z$-modules $A$ endowed with an endomorphism\footnote{The reader can feel free to replace ``endomorphism'' with ``automorphism'' in this section, as our definitions will only see the submodule $\bigcup_fA^{\sigma^f}$ of $A$ consisting of those elements on which $\sigma$ acts as an automorphism of finite order.} $\sigma$, and that we're interested in the family of $\Z[\sigma]$-submodules $A^{\sigma^f}$, which we call the \emph{fixpoint filtration} of $A$ (indexed by $\N$ with the divisibility ordering). We are also interested in the \emph{partial quotients} of this filtration, by which we mean the $\Z[\sigma]$-modules\[\Gr^\sigma_f(A):=\frac{A^{\sigma^f}}{\sum_{d\mid f,d\neq f}A^{\sigma^d}}.\]Our chief method of proof is careful calculations involving cyclotomic polynomials, for which we need a preparatory proposition.

\begin{proposition}\label{prop:coprime_cyclotomics}
Let $P_1,\dots,P_m\in\Z[t]$ be integer polynomials, each of which is a product of some cyclotomic polynomials $P_i=\prod_j\Phi_{d_{ij}}$. Then the $P_i$ generate a proper ideal of $\Z[t]$ if and only if there exist:
\begin{itemize}
	\item a prime number $\ell$; and
	\item for each index $1\leq i\leq m$, an index $j_i$;
\end{itemize}
such that $d_{ij_i}/d_{i'j_{i'}}$ is a power of~$\ell$ for all $i,i'$ (i.e.\ is $\ell^n$ for some $n\in\Z$).
\begin{proof}
Let $R=\Z[t]/(P_1,\dots,P_m)$. If the ideal $(P_1,\dots,P_m)$ is proper, then $R\neq0$ and so there is a surjection $R\twoheadrightarrow F$ from $R$ to a field $F$. The field $F$, being a finitely generated $\Z$-algebra, is a finite field, and hence a subfield of $\overline\F_\ell$ for some $\ell$.

The image $\zeta$ of $t$ under the map $R\twoheadrightarrow F\subseteq\overline\F_\ell$ is a common root of the polynomials $P_i$ in $\overline\F_\ell$. For each $i$, the factorisation $P_i=\prod_j\Phi_{d_{ij}}$ ensures that there is an index $j_i$ such that $\zeta$ is a root of $\Phi_{d_{ij_i}}$.

But the roots of $\Phi_{d_{ij_i}}$ in $\overline\F_\ell$ are exactly the primitive $d_{ij_i}^\circ$th roots of unity, where $d^\circ$ denotes the largest $\ell$-free factor of $d$. Since $\zeta$ is a primitive $d_{ij_i}^\circ$th root of unity for all $i$, this implies that the $d_{ij_i}^\circ$ are all equal. In other words, $d_{ij_i}/d_{i'j_{i'}}$ is a power of $\ell$ for all $i,i'$, which is what we wanted to prove.

In the converse direction, suppose that $\ell$, $j_i$ and $n_{i,i'}$ exist satisfying $d_{ij_i}/d_{i'j_{i'}}=\ell^{n_{i,i'}}$ for all $i,i'$. This says that the $d_{ij_i}^\circ$ are all equal to one another, and we write $d^\circ$ for this common value. Now let $\zeta\in\overline\F_\ell$ be a primitive $d^\circ$th root of unity. It follows that $\zeta$ is a root of $\Phi_{ij_i}$, and hence of $P_i$, for all $i$. Hence there is a ring homomorphism $R\rightarrow\overline\F_\ell$ sending $t$ to $\zeta$. The existence of such a homomorphism implies that $R\neq0$, and hence that $(P_1,\dots,P_m)$ is proper.
\end{proof}
\end{proposition}

For us, the main consequence of this proposition is that submodules of $\Z[\sigma]$-modules cut out by products of cyclotomic polynomials are particularly well-behaved.

\begin{lemma}\label{lem:cups_and_caps}
For a finite subset $S\subseteq\N$ let $\Phi_S:=\prod_{d\in S}\Phi_d$, where as usual we take $\Phi_S=1$ if $S$ is empty. If $S,S'\subseteq\N$ are finite subsets closed under divisors, then for every $\Z[\sigma]$-module $A$ we have\[A[\Phi_{S\cap S'}(\sigma)]=A[\Phi_S(\sigma)]\cap A[\Phi_{S'}(\sigma)]\]and\[A[\Phi_{S\cup S'}(\sigma)]=A[\Phi_S(\sigma)]+A[\Phi_{S'}(\sigma)].\]

Here, $A[P]$ denotes the kernel of the multiplication-by-$P$ map $A\rightarrow A$, for $P\in\Z[\sigma]$.
\begin{proof}
Let $P=\Phi_{S\setminus S'}=\Phi_S/\Phi_{S\cap S'}=\Phi_{S\cup S'}/\Phi_{S'}$ and $P'=\Phi_{S'\setminus S}=\Phi_{S'}/\Phi_{S\cap S'}=\Phi_{S\cup S'}/\Phi_S$. Since no element of $S\setminus S'$ divides any element of $S'\setminus S$ and vice versa, proposition \ref{prop:coprime_cyclotomics} ensures that $P$ and $P'$ generate the unit ideal of $\Z[t]$: there are integer polynomials $Q$ and $Q'$ such that $QP+Q'P'=1$.

For the first equality, multiplying $QP+Q'P'=1$ by $\Phi_{S\cap S'}$ we see that we have $Q\Phi_S+Q'\Phi_{S'}=\Phi_{S\cap S'}$ and hence $A[\Phi_S(\sigma)]\cap A[\Phi_{S'}(\sigma)]\leq A[\Phi_{S\cap S'}(\sigma)]$. As $\Phi_{S\cap S'}\mid\Phi_S,\Phi_{S'}$, the converse inclusion is clear.

For the second equality, consider some $a\in A[\Phi_{S\cup S'}(\sigma)]$. Now we have that $\Phi_S(\sigma)Q(\sigma)P(\sigma)a=Q(\sigma)\Phi_{S\cup S'}(\sigma)a=0$ and $\Phi_{S'}(\sigma)Q'(\sigma)P'(\sigma)a=0$ similarly, so $a=Q(\sigma)P(\sigma)a+Q'(\sigma)P'(\sigma)a\in A[\Phi_S(\sigma)]+A[\Phi_{S'}(\sigma)]$. Hence we have the inclusion $A[\Phi_{S\cup S'}(\sigma)]\leq A[\Phi_S(\sigma)]+A[\Phi_{S'}(\sigma)]$, and the other inclusion is clear.
\end{proof}
\end{lemma}

\begin{corollary}
Let $S$ be a finite subset of $\N$ closed under divisors, and $A$ a $\Z[\sigma]$-module. Then\[A[\Phi_S(\sigma)]=\sum_{d\in S}A^{\sigma^d}\]
\begin{proof}
For each $d\in S$ let $S_d$ be the set of divisors of $d$, so that $S=\bigcup_{d\in S}S_d$. Since $A[\Phi_{S_d}(\sigma)]=A[\sigma^d-1]=A^{\sigma^d}$, an iterated application of the preceding lemma provides the desired equality.
\end{proof}
\end{corollary}

Lemma \ref{lem:cups_and_caps} has further important consequences. Firstly, we can use this lemma to turn the partially ordered fixpoint filtration $A^{\sigma^f}$ on a $\Z[\sigma]$-module $A$ into a totally ordered one, thereby justifying our use of the phrase ``partial quotients'' to describe the subquotients $\Gr^\sigma_f(A)$, and secondly, we find that these partial quotients give something akin to an isotypic decomposition of the $\Z[\sigma]$-module $A$, in that the $\Z[\sigma]$-module structure on $\Gr^\sigma_f(A)$ factors through $\Z[\mu_f]$.

\begin{corollary}
Let $A$ be a $\Z[\sigma]$-module. Then for each $f\in\N$, $A^{\sigma^f}$ possesses an exhaustive separated filtration whose partial quotients are $\Gr^\sigma_d(A)$ for $d\mid f$ in some order.
\begin{proof}
Pick a sequence $\emptyset=S_0\subsetneq S_1\subsetneq\dots\subsetneq S_m$ of subsets of $\N$, each closed under divisors, so that each $S_{i+1}\setminus S_i=\{d_i\}$ has size $1$, and $S_m$ is the set of divisors of $f$. For instance, we might take $S_i$ to be the set of the first $i$ divisors of $f$. We consider the $\Z$-indexed filtration $0=A_0\leq A_1\leq\dots\leq A_m$ of $A^{\sigma^f}$ defined by\[A_i=A[\Phi_{S_i}(\sigma)]=\sum_{d\in S_i}A^{\sigma^d}\]so that $A_0=0$ and $A_m=A^{\sigma^f}$.

Now lemma \ref{lem:cups_and_caps} (applied to $S_i$ and the set of divisors of $d_i$) shows that $A_i\cap A^{\sigma^{d_i}}=A\left[\frac{\sigma^{d_i}-1}{\Phi_{d_i}(\sigma)}\right]=\sum_{d\mid d_i,d\neq d_i}A^{\sigma^d}$ and $A_i+A^{\sigma^{d_i}}=A_{i+1}$. Hence by the second isomorphism theorem, $A_{i+1}/A_i=A^{\sigma^{d_i}}/\sum_{d\mid d_i,d\neq d_i}A^{\sigma^d}=\Gr^\sigma_{d_i}(A)$. Since $d_i$ runs through all the divisors of $f$, it follows that the partial quotients of the filtration are $\Gr^\sigma_d(A)$ for each $d\mid f$, as desired.
\end{proof}
\end{corollary}

\begin{corollary}\label{cor:Phi_annihilates_partial_quotients}
If $A$ is any $\Z[\sigma]$-module then $\Phi_f(\sigma)$ annihilates the partial quotient $\Gr^\sigma_f(A)$ for all $f\in\N$.
\begin{proof}
$\Phi_f(\sigma)A^{\sigma^f}\leq A\left[\frac{\sigma^f-1}{\Phi_f(\sigma)}\right]=\sum_{d\mid f,d\neq f}A^{\sigma^d}$. Hence the induced action of $\Phi_f(\sigma)$ on $\Gr^\sigma_f(A)$ is zero.
\end{proof}
\end{corollary}

With these results, we have now justified all of lemma \ref{lem:fixpoint_properties} save the behaviour of the partial quotients under direct sums and inductions. The case of direct sums is trivial, while the case of inductions involves some more technical manipulations.

\begin{lemma}\label{lem:partial_quotients_of_inductions}
Let $A$ be an abelian group with an action by $\sigma^q$, so that by corollary \ref{cor:Phi_annihilates_partial_quotients} $\Gr^{\sigma^q}_f(A)$ can be viewed as a $\Z[\mu_f]$-module with the $\sigma^q$-action given by multiplication by $\zeta_f$. Then\[\Gr^\sigma_f\left(\Ind_{\sigma^q}^\sigma A\right)\simeq\Gr^{\sigma^q}_{\num{f/q}}(A)\otimes_{\Z[\mu_{\num{f/q}}]}\Z[\mu_f]\]where $\num{f/q}$ denotes the numerator of $f/q$. In particular, when $q$ is prime we have\[\Gr^{\sigma^q}_f\left(\Ind_{\sigma^q}^\sigma A\right)\simeq\begin{cases}\Gr^{\sigma^q}_{f/q}(A)\otimes_{\Z[\mu_{f/q}]}\Z[\mu_f]&\text{if $q\mid f$}\\\Gr^{\sigma^q}_f(A)&\text{else}\end{cases}\]
\begin{proof}
It suffices to prove the case when $q$ is prime. Note that\[\left(\Ind_{\sigma^q}^\sigma A\right)^{\sigma^d}=\begin{cases}\Ind_{\sigma^q}^\sigma A^{\sigma^d}&\text{if $q\mid d$}\\(1+\sigma^d+\dots+\sigma^{(q-1)d})A^{\sigma^{qd}}\cong A^{\sigma^{qd}}&\text{if $q\nmid d$}\end{cases}.\]

Combining this classification with the definition\[\Gr^\sigma_f\left(\Ind_{\sigma^q}^\sigma A\right)=\frac{\left(\Ind_{\sigma^q}^\sigma A\right)^{\sigma^f}}{\sum_{d\mid f,d\neq f}\left(\Ind_{\sigma^q}^\sigma A\right)^{\sigma^d}}\]we see immediately that when $q\nmid f$ we have $\Gr^\sigma_f\left(\Ind_{\sigma^q}^\sigma A\right)=\frac{A^{(\sigma^q)^f}}{\sum_{d\mid f,d\neq f}A^{(\sigma^q)^d}}=\Gr^{\sigma^q}_f(A)$, as desired.

When $q^2\mid f$ instead, then on the denominator we need only take those $d$ such that $q\mid d$, and hence by exactness of $\Ind_{\sigma^q}^\sigma$ we have that $\Gr^\sigma_f\left(\Ind_{\sigma^q}^\sigma A\right)=\Ind_{\sigma^q}^\sigma\frac{A^{\sigma^f}}{\sum_{d\mid f/q,d\neq f/q}A^{\sigma^{qd}}}=\Ind_{\sigma^q}^\sigma\Gr^{\sigma^q}_{f/q}(A)$. Since $\Gr^{\sigma^q}_{f/q}(A)$ is a $\Z[\mu_{f/q}]$-module, this is the same as $\Gr^{\sigma^q}_{f/q}(A)\otimes_{\Z[\mu_{f/q}]}\Z[\mu_f]$, as desired.

The most difficult case is when $q$ exactly divides $d$. Here on the denominator we need only take those $d$ such that $q\mid d$, along with $d=f/q$. In other words, we can identify $\Gr^\sigma_f\left(\Ind_{\sigma^q}^\sigma A\right)$ as the cokernel of the natural map\[\frac{\left(\Ind_{\sigma^q}^\sigma A\right)^{\sigma^{f/q}}}{\sum_{d\mid f/q,d\neq f/q}\left(\Ind_{\sigma^q}^\sigma A\right)^{\sigma^d}}\longrightarrow\frac{\left(\Ind_{\sigma^q}^\sigma A\right)^{\sigma^f}}{\sum_{q\mid d\mid f,d\neq f}\left(\Ind_{\sigma^q}^\sigma A\right)^{\sigma^d}}.\]Exactly as we did above, we can identify the leftmost of these groups with $\Gr^{\sigma^q}_{f/q}(A)$, the rightmost with $\Ind_{\sigma^q}^\sigma\Gr^{\sigma^q}_{f/q}(A)$, and the map between them as multiplication by $1+\sigma^f+\dots+\sigma^{(q-1)f}$. 

Yet we have an exact sequence of $\Z[\mu_{f/q}]$-modules\[0\longrightarrow\Z[\mu_{f/q}]\longrightarrow\Ind_{\sigma^q}^\sigma\Z[\mu_{f/q}]\longrightarrow\Z[\mu_f]\longrightarrow0\]where $\Z[\mu_{f/q}]$ and $\Z[\mu_f]$ are given $\sigma^q$- and $\sigma$-actions in the usual way. The first arrow is multiplication by $1+\sigma^f+\dots+\sigma^{(q-1)f}$, and the second (which is $\sigma$-equivariant) sends $\sigma^i\mapsto\zeta_q^i$. Since this sequence is an exact complex of flat $\Z[\mu_{f/q}]$-modules, it remains exact when we tensor with $\Gr^{\sigma^q}_{f/q}(A)$ and hence we obtain a $\sigma$-equivariant exact sequence\[0\longrightarrow\Gr^{\sigma^q}_{f/q}(A)\longrightarrow\Ind_{\sigma^q}^\sigma\Gr^{\sigma^q}_{f/q}(A)\longrightarrow \Gr^{\sigma^q}_{f/q}(A)\otimes_{\Z[\mu_{f/q}]}\Z[\mu_f]\longrightarrow0.\]But we identified $\Gr^\sigma_f\left(\Ind_{\sigma^q}^\sigma A\right)$ as the cokernel of the left-hand arrow, so that $\Gr^\sigma_f\left(\Ind_{\sigma^q}^\sigma A\right)\simeq\Gr^{\sigma^q}_{f/q}(A)\otimes_{\Z[\mu_{f/q}]}\Z[\mu_f]$ as desired.
\end{proof}
\end{lemma}

We finish with a technical result used in the proof of lemma~\ref{lem:adding_points}.

\begin{proposition}\label{prop:rank_1_surj}
Let $\psi\colon\Pi\twoheadrightarrow\gNeron$ be a surjection of $\Z[\sigma]$-modules whose kernel is a free $\Z$-module of rank $1$. Then $\psi$ induces an isomorphism on $\Gr^\sigma_f$ for all $f>2$. In particular, $\Pi$ is fixpoint-regular (definition~\ref{def:fixpoint-regular}) if and only if $\gNeron$ is.
\begin{proof}
The action of $\sigma$ on $Z:=\ker(\Pi\twoheadrightarrow\gNeron)$ is multiplication by $\pm1$; we deal with the two cases separately.

In the case that $\sigma$ acts on $Z$ by $+1$, we note that
\[
\Gr^\sigma_f(A)=\frac{A^{\sigma^f}/A^\sigma}{\sum_{d\mid f,d\neq f}(A^{\sigma^d}/A^\sigma)}
\]
for all $f>1$ and any $\Z[\sigma]$-module $A$. In particular, the map $\Gr^\sigma_f(\psi)$ for $f>1$ fits into a commuting diagram
\begin{center}
\begin{tikzcd}
\bigoplus_{d\mid f,d\neq f}\Pi^{\sigma^d}/\Pi^\sigma \arrow[r]\arrow[d,"\bigoplus(\psi\text{ mod $\gNeron^\sigma$})"] & \Pi^{\sigma^f}/\Pi^\sigma \arrow[r]\arrow[d,"\psi\text{ mod $\gNeron^\sigma$}"] & \Gr^\sigma_f(\Pi) \arrow[d,"\Gr^\sigma_f(\psi)"] \\
\bigoplus_{d\mid f,d\neq f}\gNeron^{\sigma^d}/\gNeron^\sigma \arrow[r] & \gNeron^{\sigma^f}/\gNeron^\sigma \arrow[r] & \Gr^\sigma_f(\gNeron)
\end{tikzcd}
\end{center}
with exact rows. Thus it suffices to prove that the map $\Pi^{\sigma^d}/\Pi^\sigma\rightarrow\gNeron^{\sigma^d}/\gNeron^\sigma$ induced by $\psi$ is an isomorphism for all $d\geq1$. To do this, we note that the exact sequence
\begin{equation}\label{eq:rank_1_surj}
0\rightarrow Z\rightarrow\Pi\rightarrow\gNeron\rightarrow0
\end{equation}
remains exact upon taking $\sigma$- and $\sigma^d$-fixed points, and hence we have a commuting diagram
\begin{center}
\begin{tikzcd}
0 \arrow{r} & Z \arrow{r}\arrow[equals]{d} & \Pi^\sigma \arrow{r}\arrow[hook]{d} & \gNeron^\sigma \arrow{r}\arrow[hook]{d} & 0 \\
0 \arrow{r} & Z \arrow{r} & \Pi^{\sigma^d} \arrow{r} & \gNeron^{\sigma^d} \arrow{r} & 0
\end{tikzcd}
\end{center}
with exact rows. Applying the snake lemma shows that the induced map $\Pi^{\sigma^d}/\Pi^\sigma\rightarrow\gNeron^{\sigma^d}/\gNeron^\sigma$ is an isomorphism, as desired.

Now, in the case that $\sigma$ acts on $Z$ by $-1$, we use the corresponding fact that
\[
\Gr^\sigma_f(A)=\frac{A^{\sigma^f}/A^{\sigma^{\gcd(f,2)}}}{\sum_{d\mid f,d\neq f}(A^{\sigma^d}/A^{\sigma^{\gcd(d,2)}})}
\]
for all $f>2$ and any $\Z[\sigma]$-module $A$. Thus it suffices to prove that the induced map $\Pi^{\sigma^d}/\Pi^{\sigma^{\gcd(d,2)}}\rightarrow\gNeron^{\sigma^d}/\gNeron^{\sigma^{\gcd(d,2)}}$ is an isomorphism for all $d$. When $d$ is even, this follows from the above calculation (applied to $\sigma^2$), so we focus on the case that $d$ is odd.

Taking $\sigma$- and $\sigma^d$-fixed points of~\eqref{eq:rank_1_surj} provides a commuting diagram
\begin{center}
\begin{tikzcd}
0 \arrow{r} & \Pi^\sigma \arrow{r}\arrow[hook]{d} & \gNeron^\sigma \arrow{r}\arrow[hook]{d} & \F_2 \arrow{r}\arrow[equals]{d} & \H^1(\sigma,\Pi)[2] \arrow{d} \\
0 \arrow{r} & \Pi^{\sigma^d} \arrow{r} & \gNeron^{\sigma^d} \arrow{r} & \F_2 \arrow{r} & \H^1(\sigma^d,\Pi)[2]
\end{tikzcd}
\end{center}
with exact rows, where $\H^1(\sigma,-)$ denotes continuous cohomology of the action of $\hat\Z$ where a generator acts via $\sigma$, as usual. The inflation--restriction exact sequence implies that the kernel of the right-hand vertical map is $\H^1(\hat\Z/d\hat\Z,\Pi^{\sigma^d})[2]=0$ since $d$ is odd, and hence this map is injective. This implies that the two rightmost horizontal maps have the same kernel, and so the snake lemma again implies that the induced map $\Pi^{\sigma^d}/\Pi^\sigma\rightarrow\gNeron^{\sigma^d}/\gNeron^\sigma$ is an isomorphism, as desired.

The assertion regarding fixpoint-regularity follows from the observation that fixpoint-regularity of $A$ depends only on $\Gr^\sigma_f(A)$ for $f>2$.
\end{proof}
\end{proposition}

\end{document}